\documentclass[a4paper]{amsart}

\usepackage[a4paper, margin=2cm]{geometry}
\usepackage[british]{babel}

\usepackage{amsmath}
\usepackage{amssymb}
\usepackage{calc}
\usepackage{enumitem}
\usepackage{graphicx}
\usepackage{tikz-cd}
\usepackage{url}
\usepackage{xcolor}
\usepackage{etoolbox}
\usepackage{tabularx}
\usepackage{mathtools}
\usepackage{xcolor}
\usepackage{hyperref}

\usepackage[f]{esvect}

\hypersetup{
  colorlinks   = true, 
  urlcolor     = {blue!50!black}, 
  linkcolor    = {blue!50!black}, 
  citecolor   = {red!50!black} 
}

\usepackage{tikz}
\usepackage{mathdots}
\usepackage{yhmath}
\usepackage{cancel}
\usepackage{color}
\usepackage{array}
\usepackage{multirow}
\usepackage{amssymb}
\usepackage{tabularx}
\usepackage{extarrows}
\usepackage{booktabs}
\usetikzlibrary{fadings}
\usetikzlibrary{patterns}
\usetikzlibrary{shadows.blur}
\usetikzlibrary{shapes}

\numberwithin{equation}{section}

\usepackage{cleveref}

\makeatletter
\def\subsubsection{\@startsection{subsubsection}{3}%
  \z@{.5\linespacing\@plus.7\linespacing}{-.5em}%
  {\normalfont\bfseries}}
\makeatother

\theoremstyle{plain}
\newtheorem{thm}{Theorem}[section]

\newtheorem{prop}[thm]{Proposition}
\newtheorem{lem}[thm]{Lemma}
\newtheorem{fact}[thm]{Fact}
\newtheorem*{lem*}{Lemma}
\newtheorem*{cor*}{Corollary}
\newtheorem*{thm*}{Theorem}
\newtheorem*{fact*}{Fact}

\newtheorem{manualthminner}{Theorem}
\newenvironment{manualthm}[1]{%
  \IfBlankTF{#1}
    {\renewcommand{\themanualthminner}{\unskip}}
    {\renewcommand\themanualthminner{#1}}%
  \manualthminner
}{\endmanualthminner}

\newtheorem{manualfactinner}{Fact}
\newenvironment{manualfact}[1]{%
  \IfBlankTF{#1}
    {\renewcommand{\themanualfactinner}{\unskip}}
    {\renewcommand\themanualfactinner{#1}}%
  \manualfactinner
}{\endmanualfactinner}

\theoremstyle{definition}
\newtheorem{defn}[thm]{Definition}
\newtheorem{eg}[thm]{Example}
\newtheorem{notn}[thm]{Notation}
\newtheorem{qn}[thm]{Question}

\newtheorem*{notn*}{Notation}
\newtheorem*{term*}{Terminology}
\newtheorem*{qn*}{Question}

\theoremstyle{remark}
\newtheorem{rem}[thm]{Remark}

\DeclareMathOperator{\Age}{Age}

\DeclareMathOperator{\Aut}{Aut}

\DeclareMathOperator{\dom}{dom}

\DeclareMathOperator{\GL}{GL}
\DeclareMathOperator{\id}{id}
\DeclareMathOperator{\im}{im}

\DeclareMathOperator{\qftp}{qftp}

\DeclareMathOperator{\supp}{supp}

\DeclareMathOperator{\tp}{tp}

\newcommand{\N}{\mathbb{N}}
\newcommand{\Q}{\mathbb{Q}}
\newcommand{\R}{\mathbb{R}}
\newcommand{\T}{\mathbb{T}}
\newcommand{\Z}{\mathbb{Z}}
\newcommand{\tw}{\frac{2\pi}{n}}
\newcommand{\ic}{\mathbin{\bot}}
\newcommand{\mb}[1]{\mathbb{#1}}
\newcommand{\mc}[1]{\mathcal{#1}}
\newcommand{\ov}[1]{\overline{#1}}

\newcommand{\all}{\forall\,}
\newcommand{\ex}{\exists\,}
\newcommand{\sub}{\subseteq}

\newcommand{\tld}[1]{\tilde{#1}}

\newcommand{\fin}{\subseteq_{\mathrm{fin\!}}}
\newcommand{\fg}{\subseteq_{\mathrm{f.g.\!}}}

\newcommand{\ri}{\circ}
\newcommand{\ra}{\rightarrow}

\newcommand{\ua}{\uparrow}
\newcommand{\la}{\leftarrow}
\newcommand{\di}[1]{\lfloor{#1}\rfloor}

\newcommand{\Mod}[1]{\ (\mathrm{mod}\ #1)}
\newcommand{\nrm}{\trianglelefteq}
\newcommand{\Dn}{\mathbb{D}_n}
\newcommand{\Dpn}{\mathbb{D}'_n}
\newcommand{\Sn}{\mathbb{S}(n)}
\newcommand{\Sth}{\ddot{\mathbb{S}}(2)}
\newcommand{\Srho}{\mathbb{S}_\rho}
\newcommand{\ovST}{\overline{\mathcal{S}}_{+}}
\newcommand{\ST}{\mathcal{S}_{+}}
\newcommand{\ovSTT}{\overline{\mathcal{S}}_{++}}
\newcommand{\STT}{\mathcal{S}_{++}}
\newcommand{\Tn}{\mathbb{T}_n}
\newcommand{\Pthr}{\mathbb{P}(3)}
\newcommand{\ZT}{\mathbb{Z} / 3\mathbb{Z}}
\newcommand{\eps}{\varepsilon}
\newcommand{\emp}{\varnothing}
\newcommand{\Sym}{\mathrm{Sym}}
\newcommand{\Alt}{\mathrm{Alt}}

\def\ind{\mathrel{\raise0.2ex\hbox{\ooalign{\hidewidth$\vert$\hidewidth\cr\raise-0.9ex\hbox{$\smile$}}}}}

\newcommand{\Fr}{Fra\"{i}ss\'{e} }
\newcommand{\Frn}{Fra\"{i}ss\'{e}}

\allowdisplaybreaks

\author{Thomas Bernert}
\address{\parbox{\linewidth}{Thomas Bernert\\
School of Mathematics, University of Leeds\\
Leeds\\
LS2 9JT\\
United Kingdom
}}
\email{ll16tb@leeds.ac.uk}

\author{Rob Sullivan}
\address{\parbox{\linewidth}{Rob Sullivan\\
Institute of Computer Science, Czech Academy of Sciences\\
Pod Vodárenskou věží 271/2\\
182 00 Prague\\
Czech Republic
}}
\email{robertsullivan1990+maths@gmail.com}

\author{Jeroen Winkel}
\address{Jeroen Winkel}
\email{winkeljeroen+maths@gmail.com}

\author{Shujie Yang}
\address{\parbox{\linewidth}{Shujie Yang\\
Institut f\"{u}r Mathematische Logik und Grundlagenforschung\\
Universit\"{a}t  M\"{u}nster\\
Einsteinstra{\ss}e 62\\
48149 M\"{u}nster\\
Germany}}
\email{syang2@uni-muenster.de} 

\thanks{Thomas Bernert is funded by EPSRC Studentship 2712596. Rob Sullivan is funded by Project 24-12591M of the Czech Science Foundation (GA\v{C}R). Shujie Yang is funded by the Deutsche Forschungsgemeinschaft (DFG, German Research Foundation) under Germany’s Excellence Strategy EXC 2044–390685587, Mathematics M\"{u}nster: Dynamics–Geometry–Structure; by CRC 1442 Geometry: Deformations and Rigidity; and by the China Scholarship Council (CSC) 202204910109.}

\subjclass[2020]{03C15, 20B27, 20E32, 03C45}

\keywords{automorphism group, simple, independence relation, hypertournament, semigeneric tournament}

\title{Determining the normal subgroups of the automorphism groups of ultrahomogeneous structures via stabilisers}

\date{\today}

\begin{document}

\begin{abstract}
    We show the simplicity of the automorphism groups of the generic $n$-hypertournament and the semigeneric tournament, and determine the normal subgroups of the automorphism groups of several other ultrahomogeneous oriented graphs. We also give a new proof of the simplicity of the automorphism group of the dense $\frac{2\pi}{n}$-local order $\mathbb{S}(n)$ for $n \geq 2$ (a result due to Droste, Giraudet and Macpherson). Previous techniques of Li, Macpherson, Tent and Ziegler involving stationary weak independence relations (SWIRs) cannot be applied directly to these structures; our approach involves applying these techniques to a certain expansion of each structure, where the expansion has a SWIR and its automorphism group is isomorphic to a stabiliser subgroup of the automorphism group of the original structure.
\end{abstract}

\maketitle

\section{Introduction}

Recall that a first-order structure $M$ is \emph{ultrahomogeneous} if every isomorphism between finitely generated substructures extends to an automorphism of $M$. We call a countably infinite ultrahomogeneous structure a \emph{\Fr structure}, and we assume the reader is familiar with the basic theory of these -- see \cite{Hod93}, \cite{Mac11} for more background. In this paper, we determine the normal subgroups of the automorphism groups of certain \Fr structures: the generic $n$-hypertournament, the semigeneric tournament and several other ultrahomogeneous oriented graphs from Cherlin's list (\cite{Che98}). We also give a new proof that the automorphism group of the dense $\frac{2\pi}{n}$-local order is simple; this is a result due to Droste, Giraudet and Macpherson. We use a new variation on existing techniques of Li, Macpherson, Tent and Ziegler, as explained below.

\subsection{Background}

\subsubsection*{Classical results}

There is a long history of results determining the normal subgroups of automorphism groups of \Fr structures. Some of the main structures considered are as follows:
\begin{itemize}
    \item a countably infinite pure set: the non-trivial proper normal subgroups of the symmetric group $\Sym_\omega$ are the group of finite-support permutations and its subgroup of even permutations  (\cite{Ono29}, \cite{SU33}, \cite{Bae34});
    \item $(\Q, <)$, the dense linear order without endpoints (\cite{Hig54}, \cite{Llo64} -- see also \cite{Gla82}): the non-trivial proper normal subgroups of $\Aut(\Q, <)$ consist of:
    \begin{align*}
        \text{the subgroup of left-bounded automorphisms } L &= \{g \in \Aut(\Q, <) \mid \ex a \in \Q, g|_{(-\infty, a)} = \id\},\\
        \text{the subgroup of right-bounded automorphisms } R &= \{g \in \Aut(\Q, <) \mid \ex a \in \Q, g|_{(a, \infty)} = \id\},\\
        \text{the subgroup of bounded automorphisms } B &= L \cap R;
    \end{align*}
    \item the $\omega$-dimensional vector space $V_\omega$ over a finite field $K$: see \cite{Ros58} for the normal subgroups of $\GL(V_\omega, K)$;
    \item the Rado graph: its automorphism group is simple (\cite{Tru85});
    \item the generic $K_n$-free graph, the random tournament: their automorphism groups are simple (\cite{Rub88}).
    \item $2$-homogeneous trees: the automorphism group of each countably infinite $2$-homogeneous tree has $2^{2^{\aleph_0}}$ normal subgroups which are pairwise $\sub$-incomparable (\cite{DHM89a});
    \item the generic poset: its automorphism group is simple (\cite{GMR93}).
\end{itemize}

\subsubsection*{Modern techniques using stationary independence relations (SIRs)}

The most recent techniques used to determine the normal subgroups of automorphism groups of \Fr structures are descended from an approach of Lascar from the early nineteen-nineties via descriptive set theory. Lascar used this to show the simplicity of $\Aut(\mb{C}/\mb{Q}^{\text{alg}})$ -- see \cite{Las92}, \cite{Las97}. This approach was adapted by Macpherson and Tent in \cite{MT11} to show that the automorphism group of any transitive \Fr structure with free amalgamation is simple (assuming the structure is not an indiscernible set), as are the automorphism groups of the random tournament and the generic $n$-anticlique-free oriented graph. The techniques of Macpherson and Tent were then further extended by Tent and Ziegler in \cite{TZ13a}, \cite{TZ13b} to show the (abstract) simplicity of the quotient of the isometry group of the complete Urysohn space by the normal subgroup of bounded-displacement isometries, as well as the simplicity of the isometry group of the bounded complete Urysohn space. Further results using the methods of \cite{TZ13a} can be found in \cite{GT14}, \cite{EGT16}, \cite{EHK21}, \cite{Yam23}.

One component of the Tent-Ziegler approach is the axiomatisation of a certain independence relation inspired by forking independence for stable theories, called a \emph{stationary independence relation (SIR)}: this is a ternary relation $\ind$ on the set of finitely generated substructures of a \Fr structure $M$, written $B \ind_A C$ for $A, B, C \fg M$, capturing the following notion: $B \ind_A C$ holds exactly when the substructures $B \cup A$ and $C \cup A$ are \emph{canonically amalgamated} over $A$. Here, the canonical amalgam is required to be symmetric: we have $B \ind_A C \Leftrightarrow C \ind_A B$. See Definition \ref{d:SWIR} and Definition \ref{d:SIR} for the formal description.

For example, in the Rado graph the canonical amalgam is the free amalgam: we define that $B \ind_A C$ holds when there is no edge between $B \setminus A$ and $C \setminus A$. In the rational Urysohn space $\mb{U}_\Q$ (the \Fr limit of the class of all finite metric spaces with rational distances), the canonical amalgam is that which has the maximum possible distances while still satisfying the triangle inequality: for $A, B, C \fin \mb{U}_\Q$ with $A \neq \varnothing$, we define $B \ind_A C$ if, for $b \in B \setminus A$, $c \in C \setminus A$, we have $d(b, c) = \min_{a \in A} (d(b, a) + d(c, a))$. Note an important distinction between these two examples which will be particularly relevant in the current paper: for the Rado graph, the independence relation $\ind$ was defined over the empty set, but for the rational Urysohn space we only defined $\ind$ over non-empty bases $A$. In the case where $\ind$ is defined for non-empty base substructures $A$, we call $\ind$ a \emph{local} SIR.

The second main component of the Tent-Ziegler approach is the notion of an automorphism which moves maximally. Let $M$ be a \Fr structure with a (local) SIR. An automorphism $g \in \Aut(M)$ \emph{moves maximally} if, for every type $p(\bar{x}/A)$ over a finitely generated substructure $A \fg M$ which is realisable in $M$, there is a realisation $\bar{b} \models p$ in $M$ with $\bar{b} \ind_A g(A) g(\bar{b})$ and $\bar{b} A \ind_{g(A)} g(\bar{b})$. (Here, we use the following notation: we write $UV$ to mean the substructure generated by $U \cup V$.) The main result of \cite{TZ13a} is as follows:

\begin{fact}[{\cite[Theorem 2.7, Lemma 2.8]{TZ13a}} -- rephrased] \label{f:TZ main thm}
    Let $M$ be a \Fr structure satisfying one of the following two conditions:
    \begin{enumerate}[label=(\roman*)]
        \item $M$ has a SIR;
        \item $M$ has a local SIR and $\Aut(M)$ has a dense conjugacy class.
    \end{enumerate}
    Suppose $g \in \Aut(M)$ moves maximally. Then each element of $\Aut(M)$ is a product of $\leq 8$ conjugates of $g$.
\end{fact}

By the above result, given such a structure $M$, in order to show that $\Aut(M)$ is simple it suffices to show that any non-trivial normal subgroup contains an automorphism which moves maximally. This is in general not immediate, and one usually produces such automorphisms via a back-and-forth construction, requiring some effort. However, if the SIR satisfies an additional axiom (see Definition \ref{d:free SWIR}) that we call \emph{freeness} (following \cite{Con17}), then one does not need to produce maximally-moving automorphisms by hand:

\begin{fact}[{\cite[Lemma 5.1]{TZ13a}}] \label{f:TZ free SIR move max}
    Let $M$ be a \Fr structure with a free SIR. Let $g \in \Aut(M)$. Suppose that $g$ satisfies the following condition: for each $1$-type $p(x/A)$ over $A \fg M$ with infinitely many realisations in $M$, some realisation of $p(x/A)$ is not fixed by $g$. Then there exists $h \in \Aut(M)$ such that the commutator $[g, h]$ moves maximally.
\end{fact}

Any \Fr structure $M$ with free amalgamation has a free SIR, and if $\Aut(M)$ is transitive and $M$ is not an indiscernible set then all non-trivial $g \in \Aut(M)$ satisfy the above condition (see \cite[Section 5]{TZ13a}, \cite[Corollary 2.10]{MT11}). So we have:

\begin{fact}[{\cite[Corollary 5.2]{TZ13a}}] \label{f:TZ free amalg simple}
    Let $M$ be a \Fr structure with free amalgamation which is not an indiscernible set, and suppose that $\Aut(M)$ is transitive. Then, for all $g \in \Aut(M) \setminus \{\id\}$, each element of $\Aut(M)$ is a product of $\leq 16$ conjugates of $g$ and $g^{-1}$, and thus $\Aut(M)$ is simple.
\end{fact}

\subsubsection*{The recent generalisation of Li}

In \cite{Li19} and her PhD thesis \cite{Li21}, Yibei Li generalised the main theorem of \cite{TZ13a}, Theorem 2.7, so as to be able to consider asymmetric structures. She defined the notion of a \emph{stationary weak independence relation (SWIR)}: this is an independence relation satisfying the same axioms as a SIR but excluding the assumption of symmetry (see Definition \ref{d:SWIR}). For example, in $(\Q, <)$ we define $B \ind_A C$ if, for all $b \in B \setminus A$, $c \in C \setminus A$ such that there is no $a \in A$ with $(b < a < c) \vee (b > a > c)$, we have $b < c$. Again, a SWIR corresponds to a notion of canonical amalgamation (see \cite[Section 2]{KSW25}).

Li likewise generalised the definition of a maximally-moving automorphism as follows. Let $g$ be an automorphism of a \Fr structure $M$ with SWIR $\ind$. We say that $g$ \emph{moves almost R-maximally} if every type $p(\bar{x}/A)$ over $A \fg M$ realisable in $M$ has a realisation $\bar{b} \sub M$ with $\bar{b} \ind_A g(\bar{b})$, and we say that $g$ \emph{moves almost L-maximally} if every type $p(\bar{x}/A)$ over $A \fg M$ realisable in $M$ has a realisation $\bar{b}$ with $g(\bar{b}) \ind_A \bar{b}$. The generalisation of the main theorem of Tent and Ziegler is then:

\begin{fact}[{\cite[Thm.\ 4.0.4, Cor.\ 4.2.2, Cor.\ 4.2.4]{Li21}}] \label{f:Li main thm}
    Let $M$ be a \Fr structure with a SWIR. Let $g \in \Aut(M)$. Suppose that at least one of the following conditions holds:
    \begin{enumerate}[label=(\roman*)]
        \item $g$ moves almost R-maximally and $g^{-1}$ moves almost L-maximally; or
        \item $g$ moves both almost R-maximally and almost L-maximally.
    \end{enumerate}
    Then each element of $\Aut(M)$ is a product of $\leq 16$ conjugates of $g$ and $g^{-1}$. 
\end{fact}

(Li only considered relational languages, but her proof works for any first-order language.)

Li used her more general result to show simplicity of the automorphism groups of several \Fr structures:
\begin{enumerate}[label=(\roman*)]
    \item \label{ex nLO} the generic $n$-linear order for $n \geq 2$ (that is, the \Fr limit of the class of all finite structures consisting of a set with $n$ linear orders on it);
    \item \label{ex orderexpfree} the generic order expansion of a transitive \Fr structure with free amalgamation which is not an indiscernible set;
    \item \label{ex semi-free} most of the ``semi-free" examples of Cherlin (see the appendix of \cite{Che98}).
\end{enumerate}

Li did not consider local SWIRs, and did not consider freeness. (We believe Li's main result should straightforwardly generalise to local SWIRs analogously to the Tent-Ziegler result, though we have not checked this in detail.) In \cite{KRR25}, Kaplan, Riahi and Rodr\'{i}guez Fanlo applied Li's techniques to show the simplicity of the automorphism group of the Rado meet-tree: that is, the generic meet-tree expansion of the Rado graph (in fact they show simplicity for the generic meet-tree expansion of any non-unary free amalgamation structure).

A second approach for asymmetric structures can be found in \cite{CKT21}, where Calderoni, Kwiatkowska and Tent also showed simplicity for \ref{ex orderexpfree} in the above list; simplicity for \ref{ex nLO} was likewise proved by Mei{\ss}ner in her master's thesis (\cite{Mei21}) using the machinery of \cite{CKT21}. The paper \cite{CKT21} introduces the notion of a \emph{weakly stationary independence relation}, whose definition differs from that of a SWIR. The authors also show the simplicity of $\Aut(M)$ for $M$ equal to the generic tournament expansion of a non-trivial transitive free amalgamation structure, as well as for $M$ being the generic order/tournament expansion of the rational Urysohn space or the generic poset. We do not use this second approach in this paper.

\subsection{New results: \texorpdfstring{$n$}{n}-hypertournaments, a new proof for \texorpdfstring{$\frac{2\pi}{n}$}{2pi/n}-local orders}

One focus of this paper will be on two key families of \Fr structures: the generic $n$-hypertournament and the dense $\frac{2\pi}{n}$-local order, for $n \geq 2$. We define these below.

\begin{figure}
    \centering
    \begin{tikzpicture}[x=0.75pt,y=0.75pt,yscale=-1,xscale=1]

\draw    (246.2,93.8) -- (303.2,230.8) ;
\draw    (168.2,214.8) -- (246.2,93.8) ;
\draw    (168.2,214.8) -- (303.2,230.8) ;
\draw    (246.2,93.8) -- (330.2,174.8) ;
\draw    (303.2,230.8) -- (330.2,174.8) ;
\draw  [dash pattern={on 0.84pt off 2.51pt}]  (168.2,214.8) -- (330.2,174.8) ;

\draw (245,180.92) node [anchor=north west][inner sep=0.75pt]  [rotate=-184]  {$\curvearrowleft $};
\draw (300,155) node [anchor=north west][inner sep=0.75pt]  [rotate=-113.3]  {$\curvearrowleft $};
\draw (260,150) node [anchor=north west][inner sep=0.75pt]  [rotate=-167]  {$\textcolor[rgb]{0.3,0.3,0.3}{\curvearrowright }$};
\draw (259.61,199.38) node [anchor=north west][inner sep=0.75pt]  [rotate=-8]  {$\textcolor[rgb]{0.3,0.3,0.3}{\curvearrowright }$};
\draw (159,216.4) node [anchor=north west][inner sep=0.75pt]    {$v_{0}$};
\draw (294,232.8) node [anchor=north west][inner sep=0.75pt]    {$v_{1}$};
\draw (331,170) node [anchor=north west][inner sep=0.75pt]    {$v_{2}$};
\draw (238,82) node [anchor=north west][inner sep=0.75pt]    {$v_{3}$};
\end{tikzpicture}

    \caption{An example of a $3$-hypertournament, usually denoted by $H_4$ (see \cite{CHKN21}). The curved arrows denote anticlockwise/clockwise orientations of each face (the grey arrows are for the back faces): the anticlockwise edges are $(v_0, v_1, v_3), (v_1, v_2, v_3), (v_2, v_0, v_3), (v_0, v_2, v_1)$ and their cyclic permutations.}
    \label{f: H_4}
\end{figure}
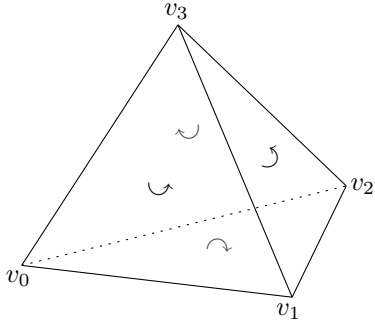

\begin{defn} \label{d:Tn}
    Let $n \geq 2$, and let $\mc{L}_n^{\textnormal{T}}$ be a language consisting of a single $n$-ary relation symbol $R$. An \emph{$n$-hypertournament} $T$ is an $\mc{L}_n^{\textnormal{T}}$-structure where the automorphism group of each $n$-element substructure is the alternating group $\Alt_n$. So, a $2$-hypertournament is a tournament in the usual sense (an oriented complete graph), and a $3$-hypertournament is a structure where each triple of points has a cyclic orientation anticlockwise or clockwise (see Figure \ref{f: H_4}). We let $\mc{T}_n$ denote the class of finite $n$-hypertournaments; this is a \Fr class (one just adds edges oriented arbitrarily). We let $\mb{T}_n$ denote the \Fr limit of $\mc{T}_n$, and call $\mb{T}_n$ the \emph{generic $n$-hypertournament}.
\end{defn}

\begin{defn} \label{d:Sn}
    Let $n \geq 2$, and let $\mc{L}^\ri_n = \{S_j \mid j < n\}$ be a binary relational language. Let $D$ be the set of points on the unit circle with rational argument (this is a countable dense subset of the unit circle). We define an $\mc{L}^\ri_n$-structure $\Sn$ with domain $D$: for distinct $u, v \in D$ and $j < n$, define that $S_j(u, v)$ holds when the angle $\alpha(u, v)$ subtended at the origin by the anticlockwise arc from $u$ to $v$ satisfies $\frac{2\pi j \vphantom{2\pi (j+1)}}{n} < \alpha(u, v) < \frac{2\pi (j+1)}{n}$. We call $\Sn$ the \emph{dense $\frac{2\pi}{n}$-local order}. See Figure \ref{f: dense local orders}. The structure $\Sn$ is ultrahomogeneous: see Lemma \ref{l: Sn ultrahomog}.
\end{defn}

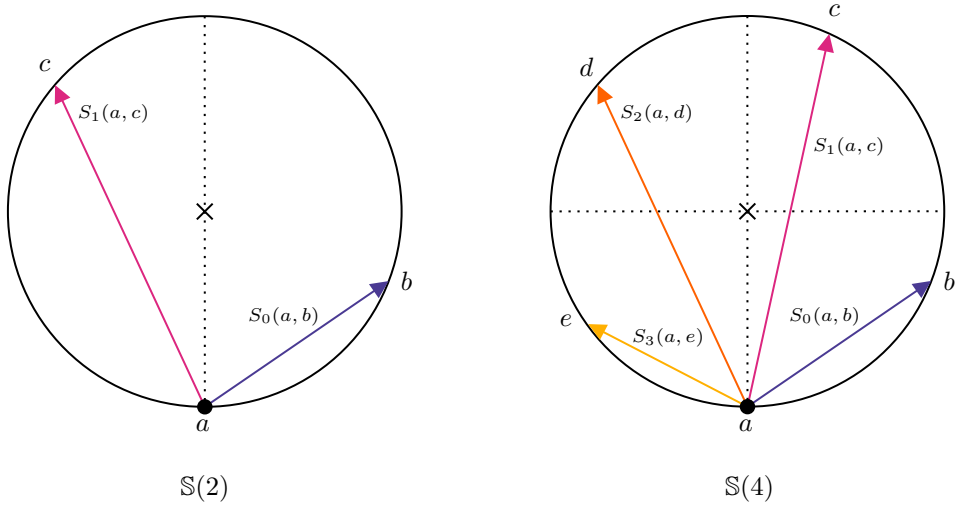
\begin{figure}
\centering

\tikzset{every picture/.style={line width=0.75pt}} 

\begin{tikzpicture}[x=0.75pt,y=0.75pt,yscale=-1,xscale=1]

\draw   (61.99,136.21) .. controls (61.98,81.94) and (105.95,37.94) .. (160.22,37.92) .. controls (214.49,37.9) and (258.49,81.87) .. (258.51,136.14) .. controls (258.53,190.41) and (214.55,234.41) .. (160.29,234.43) .. controls (106.02,234.45) and (62.01,190.48) .. (61.99,136.21) -- cycle ;
\draw [color={black}  ,draw opacity=1 ]   (249.53,172.88) -- (160.29,234.43) ;
\draw [shift={(252,171.18)}, rotate = 145.41] [fill={black}  ,fill opacity=1 ][line width=0.08]  [draw opacity=0] (8.93,-4.29) -- (0,0) -- (8.93,4.29) -- cycle    ;
\draw [color={black!70!white}  ,draw opacity=1 ]   (86.81,75.29) -- (160.29,234.43) ;
\draw [shift={(85.55,72.56)}, rotate = 65.22] [fill={black!70!white}  ,fill opacity=1 ][line width=0.08]  [draw opacity=0] (8.93,-4.29) -- (0,0) -- (8.93,4.29) -- cycle    ;
\draw   (332.87,136.12) .. controls (332.85,81.86) and (376.82,37.85) .. (431.09,37.83) .. controls (485.36,37.81) and (529.36,81.79) .. (529.38,136.05) .. controls (529.4,190.32) and (485.43,234.33) .. (431.16,234.35) .. controls (376.89,234.37) and (332.89,190.39) .. (332.87,136.12) -- cycle ;
\draw  [dash pattern={on 0.84pt off 2.51pt}]  (332.87,136.12) -- (431.13,136.09) -- (529.38,136.05) ;
\draw [color={black!40!white}  ,draw opacity=1 ]   (357.68,75.2) -- (431.16,234.35) ;
\draw [shift={(356.43,72.48)}, rotate = 65.22] [fill={black!40!white}  ,fill opacity=1 ][line width=0.08]  [draw opacity=0] (8.93,-4.29) -- (0,0) -- (8.93,4.29) -- cycle    ;
\draw [color={black!20!white}  ,draw opacity=1 ]   (354.26,194.19) -- (431.16,234.35) ;
\draw [shift={(351.6,192.8)}, rotate = 27.57] [fill={black!20!white}  ,fill opacity=1 ][line width=0.08]  [draw opacity=0] (8.93,-4.29) -- (0,0) -- (8.93,4.29) -- cycle    ;
\draw [color={black!70!white}  ,draw opacity=1 ]   (471.35,49.78) -- (431.16,234.35) ;
\draw [shift={(471.99,46.85)}, rotate = 102.28] [fill={black!70!white}  ,fill opacity=1 ][line width=0.08]  [draw opacity=0] (8.93,-4.29) -- (0,0) -- (8.93,4.29) -- cycle    ;
\draw [color={black}  ,draw opacity=1 ]   (520.4,172.8) -- (431.16,234.35) ;
\draw [shift={(522.87,171.09)}, rotate = 145.41] [fill={black}  ,fill opacity=1 ][line width=0.08]  [draw opacity=0] (8.93,-4.29) -- (0,0) -- (8.93,4.29) -- cycle    ;
\draw  [dash pattern={on 0.84pt off 2.51pt}]  (160.22,37.92) -- (160.29,234.43) ;
\draw [shift={(160.29,234.43)}, rotate = 89.98] [color={rgb, 255:red, 0; green, 0; blue, 0 }  ][fill={rgb, 255:red, 0; green, 0; blue, 0 }  ][line width=0.75]      (0, 0) circle [x radius= 3.35, y radius= 3.35]   ;
\draw [shift={(160.25,136.18)}, rotate = 134.98] [color={rgb, 255:red, 0; green, 0; blue, 0 }  ][line width=0.75]    (-5.59,0) -- (5.59,0)(0,5.59) -- (0,-5.59)   ;
\draw  [dash pattern={on 0.84pt off 2.51pt}]  (431.09,37.83) -- (431.16,234.35) ;
\draw [shift={(431.16,234.35)}, rotate = 89.98] [color={rgb, 255:red, 0; green, 0; blue, 0 }  ][fill={rgb, 255:red, 0; green, 0; blue, 0 }  ][line width=0.75]      (0, 0) circle [x radius= 3.35, y radius= 3.35]   ;
\draw [shift={(431.13,136.09)}, rotate = 134.98] [color={rgb, 255:red, 0; green, 0; blue, 0 }  ][line width=0.75]    (-5.59,0) -- (5.59,0)(0,5.59) -- (0,-5.59)   ;

\draw (256.7,164.92) node [anchor=north west][inner sep=0.75pt]  [rotate=-359.98]  {$b$};
\draw (75.86,58.11) node [anchor=north west][inner sep=0.75pt]  [rotate=-359.98]  {$c$};
\draw (146.71,267.29) node [anchor=north west][inner sep=0.75pt]  [rotate=-359.98]  {$\mathbb{S}( 2)$};
\draw (95.33,80.68) node [anchor=north west][inner sep=0.75pt]  [font=\scriptsize,rotate=-359.98]  {$S_1( a,c)$};
\draw (345.73,57.02) node [anchor=north west][inner sep=0.75pt]  [rotate=-359.98]  {$d$};
\draw (418.59,267.2) node [anchor=north west][inner sep=0.75pt]  [rotate=-359.98]  {$\mathbb{S}( 4)$};
\draw (527.58,164.83) node [anchor=north west][inner sep=0.75pt]  [rotate=-359.98]  {$b$};
\draw (470.17,30.71) node [anchor=north west][inner sep=0.75pt]  [rotate=-359.98]  {$c$};
\draw (335.94,186.51) node [anchor=north west][inner sep=0.75pt]  [rotate=-359.98]  {$e$};
\draw (372.08,193.32) node [anchor=north west][inner sep=0.75pt]  [font=\scriptsize,rotate=-359.98]  {$S_3(a,e)$};
\draw (462.38,97.28) node [anchor=north west][inner sep=0.75pt]  [font=\scriptsize,rotate=-359.98]  {$S_1(a,c)$};
\draw (366.12,80.59) node [anchor=north west][inner sep=0.75pt]  [font=\scriptsize,rotate=-359.98]  {$S_2(a,d)$};
\draw (180.45,182.63) node [anchor=north west][inner sep=0.75pt]  [font=\scriptsize,rotate=-359.98]  {$S_0(a,b)$};
\draw (450.45,182.55) node [anchor=north west][inner sep=0.75pt]  [font=\scriptsize,rotate=-359.98]  {$S_0(a,b)$};
\draw (425.29,239.2) node [anchor=north west][inner sep=0.75pt]  [rotate=-359.98]  {$a$};
\draw (154.29,239.2) node [anchor=north west][inner sep=0.75pt]  [rotate=-359.98]  {$a$};

\end{tikzpicture}
      
\caption{Examples of dense $\frac{2\pi}{n}$-local orders.}
\label{f: dense local orders}
\end{figure}

We show the following:

\begin{manualthm}{A}
    Let $n \geq 3$. The automorphism group of the generic $n$-hypertournament $\mb{T}_n$ is simple.
\end{manualthm}

We also give a new proof of the following, originally shown by Droste, Giraudet and Macpherson in \cite{DGM95}:

\begin{manualfact}{B}
    Let $n \geq 2$. The automorphism group of the dense $\frac{2\pi}{n}$-local order $\Sn$ is simple.
\end{manualfact}

(We label this as a Fact rather than a Theorem to emphasise that the result was already known, though our proof is new. In fact, Droste, Giraudet and Macpherson showed the simplicity of the automorphism groups of certain oriented graphs, which we denote by $S(n)$; as the binary structure $\Sn$ is interdefinable with the oriented graph $S(n)$, we have $\Aut(\Sn) = \Aut(S(n))$, giving Fact B. See the start of Section \ref{s: Sn}.)

\subsection{New results: oriented graphs, including the semigeneric tournament}

In \cite{Che98}, Cherlin gives a classification of all countable ultrahomogeneous oriented graphs. See Section \ref{s: ultrahomog or graphs} for the detailed descriptions of the oriented graphs from this classification that we consider. We show the following:

\begin{manualthm}{C}
    Recall the following structures from Cherlin's list, described in Section \ref{s: ultrahomog or graphs}:
    \begin{itemize}
        \item $\Sth$: the double cover of the dense local order $S(2)$;
        \item $\mb{D}_n$: the generic $n$-partite tournament;
        \item $\mb{F}$: the generic $\omega$-partite $\vec{C}_4$-tournament;
        \item $\mb{P}(3)$: the twisted partial order.
    \end{itemize}
    We have the following classification of the normal subgroups of the automorphism groups of these structures:
    \begin{enumerate}[label=(\roman*)]
        \item the only non-trivial proper normal subgroup of $\Aut(\Sth)$ is the subgroup generated by the involution sending each point to its antipodal point (\ref{ss: Sth});
        \item the normal subgroups of $\Aut(\Dn)$ are precisely $\{\pi^{-1}(K) \mid K \nrm \Sym_n\}$, where $\pi$ is the map sending each automorphism to the bijection it induces on the set of labels of the $n$ parts (\ref{ss: Dn});
        \item the only non-trivial proper normal subgroup of $\Aut(\mb{F})$ is generated by the involution which, for each part, swaps the pair of vertices within it (\ref{ss: F});
        \item $\Aut(\Pthr)$ is simple (\ref{ss: Pthr}).
    \end{enumerate}
\end{manualthm}

The proof of Theorem C uses similar techniques to our proof of Fact B, just involving a little additional detective work for each structure, and so we do not give more details in this introduction.

One of the oriented graphs from Cherlin's list, the semigeneric tournament, is more difficult to analyse, and we treat it in a separate section of the paper, Section \ref{s: sg}. The \emph{semigeneric tournament} $\mb{S}$ is the \Fr limit of the class of finite $\omega$-partite tournaments (oriented graphs where the ``non-edge" relation is an equivalence relation) satisfying an additional \emph{parity condition}: for any two pairs $\{u, u'\}$ and $\{v, v'\}$ of vertices in two different parts, the number of out-edges from $\{u, u'\}$ to $\{v, v'\}$ is even. We show the following:

\begin{manualthm}{D}
    The automorphism group $\Aut(\mb{S})$ of the semigeneric tournament $\mb{S}$ is simple.
\end{manualthm}

We also prove some results regarding automorphisms of $\mb{S}$ that may be of independent interest: in particular, we show that each non-trivial automorphism of $\mb{S}$ moves infinitely many parts (Lemma \ref{l: sg inf many parts moved}) and moves infinitely many points of each part (Lemma \ref{l: sg inf supp each part}). The semigeneric tournament $\mb{S}$ is a somewhat unusual structure: it is a finite relational \Fr structure with strong amalgamation which has a simple automorphism group, but it is imprimitive and does not have weak elimination of imaginaries (see Remark \ref{r: failure of WEI for sg}).

Note that Fact B and Theorems C and D, together with the previous results mentioned in this introduction (for the countably infinite pure set, $\Q$, the random tournament, the generic $n$-anticlique-free oriented graph, the generic poset and free amalgamation structures), give the normal subgroup lattice for the automorphism group of each countably infinite ultrahomogeneous oriented graph in Cherlin's list, with the exception of the oriented graphs resulting from lexicographic products of $\Q, \mb{T}_2, S(2)$ with an anticlique. The automorphism groups of these lexicographic products are wreath products of the automorphism groups of the factors, and we imagine it will not be difficult to determine the normal subgroups for these remaining cases -- we did not consider these.

\subsection{Why these particular structures? A new approach: SWIR expansions.} \label{ss: intro SWIR exps}

The starting point of this project was to consider the generic $n$-hypertournament $\mb{T}_n$ and the dense $\tw$-local order $\Sn$. These are relatively well-known \Fr structures whose relations are asymmetric, and the initial aim of the authors was to apply the techniques of Li to show simplicity of their automorphism groups (for $\Sn$, we only learnt of the prior result of Droste, Giraudet and Macpherson later on).

However, as became clear during the project, such a direct approach cannot work. The structures $\Tn$ for $n \geq 3$ and $\Sn$ for $n \geq 2$ do not have SWIRs (see Lemma \ref{l: Tn no SWIR}, Lemma \ref{l: Aut Sn no dense cong}), and so it is not possible to directly apply the main result of Li (Fact \ref{f:Li main thm}). In fact, it turns out that their automorphism groups do not even have dense conjugacy classes (Lemma \ref{l: Aut Tn no dense conj}, Lemma \ref{l: Aut Sn no dense cong}). This means that, even assuming a generalisation of the main result of Li to the case where $M$ has a local SWIR and $\Aut(M)$ has a dense conjugacy class by analogy with the main Tent-Ziegler result Fact \ref{f:TZ main thm} (indeed, the authors believe such a generalisation should be straightforward to prove), it is not possible to apply this to $\Aut(\Tn)$ (for $n \geq 3$) or $\Aut(\Sn)$ (for $n \geq 2$). Tent and Ziegler give an example showing that the assumption of a dense conjugacy class is necessary (\cite[Example 2.11]{TZ13a}). Thus, to show simplicity of $\Aut(\Tn)$ and $\Aut(\Sn)$, a different approach is required. The semigeneric tournament $\mb{S}$ in Theorem D does not have a local SWIR (Lemma \ref{l: sg no SWIR}), likewise meaning that the techniques of Li cannot be directly applied. (For the structures mentioned in Theorem C, some have local SWIRs and some do not, but we shall not use the local SWIRs that do exist. See Section \ref{s: ultrahomog or graphs}.)

\subsubsection*{The new approach: SWIR expansions and stabiliser subgroups}

The new approach we take in this paper, for a given \Fr structure $M$, is the following, which we refer to as the \emph{SWIR expansion method}.
\begin{itemize}
    \item \textbf{Step 1:} First determine an expansion $N$ of $M$ where $N$ has a SWIR and $\Aut(N)$ is isomorphic to a stabiliser subgroup $H \leq \Aut(M)$, and use the techniques of Li, Macpherson, Tent and Ziegler to find the normal subgroups of $\Aut(N) \cong H$.
    \item \textbf{Step 2:} Show that, for $1 \neq K \nrm \Aut(M)$, we have $K \cap H \neq 1$: here one uses the fact that $H$ is a stabiliser. As $K \cap H \nrm H$ and we know the normal subgroups of $H$ from Step 1, we know what $K \cap H$ can be.
    \item \textbf{Step 3:} Use this information about $K \cap H$ to determine what $K$ can be, again using the fact that $H$ is a stabiliser.
\end{itemize}

In the particular case where $H$ is simple, once we have shown $K \cap H \neq 1$ in Step 2, we have $H \sub K$.

We trace the lineage of this method as follows. In the initial stages of the project, when the authors were attempting to show the simplicity of $\Aut(\mb{S}(2))$, David Evans suggested to the second author to investigate point-stabiliser subgroups. This idea was the key inspiration for the development of the three-step approach above. One can also find the idea of considering point-stabilisers in the brief proof sketch of \cite[Prop.\ 4.2.8]{Mac11} (which predates the SIR technology of \cite{TZ13a}), where Macpherson sketches simplicity for the generic two-graph: he uses the fact that each point-stabiliser is isomorphic to the automorphism group of the Rado graph, and thus simple. Finally, the results of \cite{KRR25} giving the simplicity of the automorphism group of the Rado meet-tree $\mb{T}^{\mathrm{R}}$ are perhaps the closest relative in the literature to the SWIR expansion method. The Rado meet-tree does not have a SWIR, and Kaplan, Riahi and Rodr\'{i}guez Fanlo get around this by taking two SWIR expansions, one fixing a point and the other setwise-fixing a branch. They show that any automorphism of $\mb{T}^{\mathrm{R}}$ setwise-fixes a branch or is a fan (implying that it fixes a point), meaning one can work in one of the two expansions, and they then use the techniques of Li to show simplicity. A key difference from our approach is that \cite{KRR25} uses a ``divide-and-conquer" strategy, splitting all automorphisms into different cases and using a SWIR expansion for each one. In the present paper, we instead take the approach of finding a particular ``well-behaved" subgroup of automorphisms $H$, and then using group-theoretic techniques to show that this gives us enough information to determine the normal closure for automorphisms in general.

\subsection{The structure of the paper and an outline of the various examples}

\begin{itemize}
    \item \textbf{Section \ref{s: SWIRs}: SWIRs and automorphisms moving almost R-/L-maximally.} We introduce SWIRs, giving some examples and basic lemmas, and we introduce the notion of automorphisms which move almost R-maximally/L-maximally.
    \item \textbf{Section \ref{s: free SWIRs}: free SWIRs and the pointwise-stabiliser lemma.} We introduce the notion of a \emph{free} SWIR. When a SWIR is free, one does not have to do complicated ad hoc back-and-forths to obtain automorphisms which move almost R-/L-maximally: we have a general result giving these, Lemma \ref{l: free SWIR commutator move maximally}. This makes simplicity proofs for structures with a free SWIR much more straightforward (Proposition \ref{p:free SWIR simple}). We also prove a result which we refer to as the \emph{pointwise-stabiliser lemma}: this states that, under certain conditions on the structure, simplicity of pointwise-stabiliser subgroups gives simplicity of the whole automorphism group.
    \item \textbf{Section \ref{s: Tn}: the generic $n$-hypertournament $\Tn$.} This is our first example of the SWIR expansion method. Here we take the pointwise-stabiliser $H$ of a tuple $\bar{a}$ of $n-2$ points. One can use $\bar{a}$ to define a random tournament on $\dom(\Tn) \setminus \bar{a}$: for each pair $b, c \in \dom(\Tn) \setminus \bar{a}$, define $b \ra c$ if $(\bar{a}, b, c)$ is an edge of $\Tn$. We use this idea to show that $H$ is isomorphic to an expansion of $\Tn$ including a random tournament. We define a free SWIR on this expansion, using the additional tournament structure to decide when an edge of $\Tn$ is canonical, via sign maps and a small amount of algebra. This gives the simplicity of $H$, and we then carry out Steps 2 and 3 of the SWIR expansion method to show the simplicity of $\Aut(\Tn)$.
    \item \textbf{Section \ref{s: Sn}: the dense $\frac{2\pi}{n}$-local order $\Sn$.} Here each point-stabiliser is isomorphic to the automorphism group of the generic $n$-coloured linear order $\Q_n$. We then extend results of Li to find the normal subgroups of $\Aut(\Q_n)$, generalising the classical result of Higman for $\Aut(\Q)$, and use these to show simplicity of $\Aut(\Sn)$.
    \item \textbf{Section \ref{s: ultrahomog or graphs}: other oriented graphs.} Here we consider the four oriented graphs mentioned in Theorem C. The techniques are similar to those in the previous two sections, though the proofs are not entirely straightforward. (For $\Dn$ we consider setwise stabilisers, and for $\Pthr$ we use the result of Glass, McCleary and Rubin giving the simplicity of the automorphism group of the generic poset.)
    \item \textbf{Section \ref{s: sg}: the semigeneric tournament $\mb{S}$.} This is the most involved example of the paper. We find a SWIR expansion $\Srho$ of $\mb{S}$, and show that $\Aut(\Srho)$ is simple: as the SWIR is not free, this requires some extra work. We then show that $\Aut(\Srho)$ is isomorphic to the setwise-stabiliser of a \emph{generic transversal} of $\mb{S}$, and use this to carry out Steps 2 and 3 of the SWIR expansion method.
    \item \textbf{Section \ref{s: qns}: further questions.} Here we consider further questions regarding the scope of the SWIR expansion method.
\end{itemize}

\section{Stationary weak independence relations (SWIRs) and automorphisms moving almost R-/L-maximally} \label{s: SWIRs}

\begin{notn}
    Below, we write $AB$ to mean the substructure generated by $A \cup B$.

    Let $M$ be a \Fr structure, and let $A, B, B' \fg M$. We write $B \equiv_A B'$ if there exists $f \in \Aut(M)$ with $f(B) = B'$ and $f|_A = \id_A$, or equivalently if $B, B'$ have the same type over $A$ in some enumeration. In the below stationarity axiom, when we write $B \equiv_A B' \Rightarrow B \equiv_{AC} B'$, both automorphisms agree on $BA$.

    Throughout the rest of the paper, we will also be somewhat cavalier with our use of parentheses for the arguments of functions: we will often write $ga$, $gA$ rather than $g(a)$, $g(A)$ when this does not impede clarity.
\end{notn}

\begin{defn}[{\cite[Definition 3.1.1]{Li21}}] \label{d:SWIR}
    Let $M$ be a \Fr structure. A \emph{stationary weak independence relation (SWIR)} on $M$ is a ternary relation $\ind$ on the set of finitely generated substructures of $M$, written $B \ind_A C$, satisfying the following four properties:
    \begin{itemize}
        \item Invariance (Inv): for all $A, B, C \fg M$ and $g \in \Aut(M)$ we have $B \ind_A C \Rightarrow gB \ind_{gA} gC$;
        \item Existence (Ex): for all $A, B, C \fg M$,
        \begin{itemize}
            \item (LEx): there exists $B' \fg M$ with $B \equiv_A B'$ such that $B' \ind_A C$;
            \item (REx): there exists $C' \fg M$ with $C \equiv_A C'$ such that $B \ind_A C'$;
        \end{itemize}
        \item Stationarity (Sta): for all $A, B, C \fg M$,
        \begin{itemize}
            \item (LSta): if $B \ind_A C \wedge B' \ind_A C$ and $B \equiv_A B'$, then $B \equiv_{ AC } B'$;
            \item (RSta): if $B \ind_A C \wedge B \ind_A C'$ and $C \equiv_A C'$, then $C \equiv_{ AB } C'$;
        \end{itemize}
        \item Monotonicity (Mon): for all $A, B, C, D \fg M$,
        \begin{itemize}
            \item (LMon): $BD \ind_A C \Rightarrow B \ind_A C \,\wedge\, D \ind_{ AB } C$;
            \item (RMon): $B \ind_{A}  CD  \Rightarrow B \ind_A C \,\wedge\, B \ind_{ AC } D$;
        \end{itemize}
    \end{itemize}
    We call a ternary relation $\ind$ defined on the set of \emph{non-empty} finitely generated substructures of $M$ a \emph{local SWIR} if it satisfies the above axioms.

    Given a SWIR $\ind$, we extend $\ind$ to the set of finite subsets of $M$ by defining $B \ind_A C$ if $\langle B \rangle \ind_{\langle A \rangle} \langle C \rangle$.
\end{defn}

\begin{defn} \label{d:SIR}
    Let $M$ be a \Fr structure, and let $\ind$ be a (local) SWIR on $M$. If in addition $\ind$ is symmetric, i.e.\ for all $A, B, C \fg M$ we have $B \ind_A C \Leftrightarrow C \ind_A B$, then we call $\ind$ a (local) \emph{stationary independence relation} (SIR).
\end{defn}

We briefly give some examples of \Fr structures with SWIRs/SIRs (see \cite[Example 2.7]{KSW25}).

\begin{eg} \label{e:SWIRs} \hfill
    \begin{itemize}
        \item A relational \Fr structure $M$ with free amalgamation has a SIR: for $A, B, C \fin M$, define $B \ind_A C$ if $BA, CA$ are freely amalgamated over $A$.
        \item The rational Urysohn space has a local SIR: define $B \ind_A C$ if, for all $b \in B \setminus A$, $c \in C \setminus A$, we have $d(b, c) = \min_{a \in A} (d(b, a) + d(c, a))$.
        \item $\Q$ has a SWIR: define $B \ind_A C$ if, for all $b \in B \setminus A$, $c \in C \setminus A$ such that there is no $a \in A$ with $b < a < c$ or $b > a > c$, we have $b < c$.
        \item The random tournament has a SWIR: define $B \ind_A C$ if for all $b \in B \setminus A$, $c \in C \setminus A$ we have $b \ra c$.
        \item The dense $\tw$-local order $\Sn$ has a local SWIR: this is a straightforward generalisation of \cite[Proposition 5.15]{KSW25}.
    \end{itemize}
\end{eg}

Recall that a \Fr class $\mc{K}$ has \emph{strong amalgamation} if, for all pairs of embeddings $B_0 \la A \ra B_1$ in $\mc{K}$, there exists an amalgam $B_0 \ra C \la B_1$ in $\mc{K}$ such that the images of $B_0$, $B_1$ intersect exactly in the image of $A$. (Some authors refer to this as \emph{disjoint amalgamation}.) We also say that a \Fr limit has strong amalgamation if its age does.

We collect a number of basic properties of SWIRs in the below Lemma \ref{l: SWIR basic props}. For proofs (which are not difficult) see \cite[Section 2]{KSW25} and \cite[Remark 3.1.2]{Li21}.

\begin{lem} \label{l: SWIR basic props}
    Let $M$ be a \Fr structure with (local) SWIR $\ind$. Then we have:
    \begin{itemize}
        \item transitivity (Tr): for all $A, B, C, D \fg M$,
        \begin{itemize}
            \item $B \ind_A C \wedge B \ind_{AC} D \Rightarrow B \ind_A CD$,
            \item $B \ind_A C \wedge D \ind_{AB} C \Rightarrow BD \ind_A C$;
        \end{itemize}
        \item base triviality: for all $A, B \fg M$ we have $A \ind_A B \,\wedge\, B \ind_A A$, and for all $A, B, C \fg M$,
        \begin{itemize}
            \item $B \ind_A C \Rightarrow AB \ind_A C$,
            \item $B \ind_A C \Rightarrow B \ind_A AC$.
        \end{itemize}
    \end{itemize}
    In the case that $M$ has strong amalgamation: for all $A, B, C \fg M$, if $B \ind_A C$ then $(B \setminus A) \cap (C \setminus A) = \varnothing$.
\end{lem}

\begin{fact}[{\cite[Theorem 1.1]{KR07}}] \label{f:KR dense ccl}
    Let $M$ be a \Fr structure. Let $\Age(M)_{\textnormal{p}}$ be the class consisting of pairs $(A, f)$ where $A \in \Age(M)$ and $f$ is a partial automorphism of $A$ (that is, an isomorphism between finitely generated substructures of $A$). Then $\Aut(M)$ has a dense conjugacy class if and only if $\Age(M)_{\textnormal{p}}$ has the joint embedding property.
\end{fact}

\begin{lem} \label{l: SWIR implies dense ccl}
    Let $M$ be a \Fr structure with a SWIR. Then $\Aut(M)$ has a dense conjugacy class.
\end{lem}
\begin{proof}
    By Fact \ref{f:KR dense ccl}, it suffices to show that $\Age(M)_{\textnormal{p}}$ has the joint embedding property. Let $(A, f), (B, f') \in \Age(M)_{\textnormal{p}}$. As $A \in \Age(M)$, we may assume $A \sub M$, and by (Ex) there is $B' \sub M$, $B' \cong B$ with $A \ind B'$. Let $f''$ denote the partial automorphism of $B'$ corresponding to $f'$. By (Mon) we have $\dom(f) \ind \dom(f'')$ and $\im(f) \ind \im(f'')$. By ultrahomogeneity of $M$ there is $g \in \Aut(M)$ extending $f$, so by (Inv) we have $\im(f) \ind g(\dom(f''))$. So by (Sta) the isomorphism $g \circ (f'')^{-1}$ extends to an isomorphism $h : \im(f)\im(f'') \to \im(f)g(\dom(f''))$ with $\id_{\im(f)} \sub h$, and thus $f \cup f''$ extends to a partial automorphism of $AB'$.
\end{proof}
\begin{rem}
    Kaplan and Simon showed that if a \Fr structure $M$ has a weaker notion of independence relation called a CIR on its finitely generated substructures, then $\Aut(M)$ has a \emph{cyclically} dense conjugacy class, which implies Lemma \ref{l: SWIR implies dense ccl}. See \cite[Cor.\ 3.9, Th.\ 3.12]{KS19}.
\end{rem}

\begin{defn}
    Let $M$ be a \Fr structure. We call a type $p(\bar{x}/A)$ over $A \fg M$ \emph{realisable} if $p$ has a realisation in $M$. We say that a realisable type $p(\bar{x}/A)$ is \emph{exterior} if it contains the formula $x_i \neq a$ for each $x_i \in \bar{x}$ and $a \in A$.
\end{defn}

\begin{notn}
    When we write $p(\bar{x}/A)$, we will assume that $A$ is a finitely generated substructure of $M$ unless specified otherwise. Of course, one can also view $p(\bar{x}/A)$ as a type over a finite parameter set by taking a finite set of generators of $A$, and we are quite relaxed about the distinction here. As we only work with \Fr structures $M$, we have that $\bar{b}, \bar{b}' \sub M$ have the same type over $A \fg M$ $\iff$ $\bar{b}, \bar{b}'$ have the same quantifier-free type over $A$ $\iff$ $\bar{b}, \bar{b}'$ lie in the same orbit of the coordinate-wise action $\Aut(M)_{(A)} \curvearrowright M^{|\bar{x}|}$. (Here $\Aut(M)_{(A)}$ denotes the pointwise-stabiliser of $A$ in $\Aut(M)$.)
\end{notn}

\begin{defn}[{\cite[Definition 4.0.3]{Li21}}]
    Let $M$ be a \Fr structure with SWIR $\ind$. Let $g \in \Aut(M)$. Let $p(\bar{x}/A)$ be a realisable type over $A \fg M$.
    \begin{itemize}
        \item We say that $g$ \emph{moves $p$ almost R-maximally} if there is a realisation $\bar{b}$ of $p$ such that $\bar{b} \ind_A g\bar{b}$.
        \item We say that $g$ \emph{moves $p$ almost L-maximally} if there is a realisation $\bar{b}$ of $p$ such that $g\bar{b} \ind_A \bar{b}$.
    \end{itemize}
    
    We say that $g$ \emph{moves almost R-/L-maximally} if, for any realisable type $p(\bar{x}/A)$ (where $|\bar{x}|$ is arbitrary), we have that $g$ moves $p$ almost R-/L-maximally.
\end{defn}

\begin{fact}[{\cite[Lemma 5.0.1]{Li21}}] \label{f: move ext types}
    Let $\mc{M}$ be a \Fr structure with SWIR $\ind$ and $g \in \Aut(\mc{M})$. Suppose $g$ moves all exterior types almost R-/L-maximally. Then $g$ moves almost R-/L-maximally.
\end{fact}

\section{Free SWIRs and the pointwise-stabiliser lemma} \label{s: free SWIRs}

\begin{term*}
    We call a \Fr structure with strong amalgamation a \emph{strong \Fr structure}.
\end{term*}
\begin{notn*}
    When we write $\bar{a} \cap \bar{b}$ or $\bar{a} \cup \bar{b}$, we mean the intersection or union of the underlying domains of the tuples $\bar{a}$, $\bar{b}$. We use this notation throughout the rest of this paper.
\end{notn*}

\begin{lem}\label{l: realisation disjoint}
    Let $M$ be a strong \Fr structure. Let $g \in \Aut(M)$. Suppose that, for any exterior $1$-type $p$, some realisation of $p$ is not fixed by $g$. Then for any $V \fg M$, any $n \geq 1$ and any exterior $n$-type $q$, we can find a realisation $\bar{b} \models q$ in $M$ with $\bar{b} \cap g\bar{b} = \varnothing$ and $(\bar{b} \cup g\bar{b}) \cap V = \varnothing$.
\end{lem}
\begin{proof}
    We use induction on $n \geq 1$. Assume the statement for $k < n$, and let $q(\bar{x}, y / A)$ be an exterior $n$-type over some $A \fg M$ with $|\bar{x}| = n - 1$. Let $r(\bar{x} / A)$ be the projection of $q$ to $\bar{x}$. Note that $r$ is exterior. By the induction hypothesis $r$ has a realisation $\bar{b} \sub M \setminus A$ with $\bar{b} \cap g\bar{b} = \varnothing$ and $(\bar{b} \cup g\bar{b}) \cap V = \varnothing$. If $q$ contains the formula $y = x_i$ for some $i < |\bar{x}|$, we are done, so assume not. As $M$ has strong amalgamation, there is $c \in M$ with $\tp(\bar{b}, c / A) = q$ and $c \notin g(\bar{b})Vg^{-1}(\bar{b}V)$. Let $q' = \tp(c/A\bar{b}Vg(\bar{b})g^{-1}(\bar{b}V))$. Then by assumption there is a realisation $c' \in M$ of $q'$ with $gc' \neq c'$. We then have $\tp(\bar{b}, c' / A) = q$ and $(\bar{b}, c')$ is as required.
\end{proof}

\subsection{Free SWIRs}

\begin{defn} \label{d:free SWIR}
    Let $M$ be a \Fr structure with SWIR $\ind$. We say that $\ind$ is \emph{free} if, for all $A, B, C \fg M$ with $B \ind_A C$, we have $B \ind_{A'} C$ for all $(BC) \cap A \sub A' \sub A$.
\end{defn}

The above definition was given for SIRs in \cite[Lemma 5.1]{TZ13a}, and first named in \cite{Con17}.

\begin{eg}
    In Example \ref{e:SWIRs}, SWIRs are given for structures with free amalgamation and the random tournament: it is straightforward to check that these are free. The SWIR defined for $(\Q, <)$ is not free: let $a, b, c \in \Q$ with $b > a > c$, and note that $b \ind_a c$ but $b, c$ are not independent over $\varnothing$.
\end{eg}

The proof of the below lemma is inspired by \cite[Lemma 5.1]{TZ13a}; some modifications are required.

\begin{lem}\label{l: free SWIR commutator move maximally}
    Let $M$ be a strong Fraïssé structure with free SWIR $\ind$. Let $g \in \Aut(M)$. Suppose that for any exterior $1$-type $p$, some realisation of $p$ is not fixed by $g$. Then there is $h \in \Aut(M)$ such that $[g,h]$ moves almost $R$-maximally and $[g,h]^{-1}$ moves almost $L$-maximally.
\end{lem}
\begin{proof}
    Let $v_0, \cdots$ be an enumeration of the elements of $M$, and let $p_0, \cdots$ be an enumeration of the exterior realisable types over finitely generated substructures of $M$, where each $p_i$ has finitely many free variables and has parameter set $A_i \fg M$. We will define an increasing chain $h_0 \sub \cdots$ of partial isomorphisms of $M$ such that, for $i < \omega$, we have $v_i \in \dom(h_i) \cap \im(h_i)$ and there is a realisation $\bar{b}$ of $p_i$ with $\bar{b} \ind_{A_i} [g, h_i](\bar{b})$ and $[g, h_i]^{-1}(\bar{b}) \ind_{A_i} \bar{b}$. Taking $h = \bigcup_{i < \omega} h_i$, by Fact \ref{f: move ext types} we will have that $h$ is as required. 

    Suppose that $h_0, \cdots, h_{i-1}$ are already given satisfying the above conditions. Extend $h_{i-1}$ using the ultrahomogeneity of $M$ to a partial isomorphism $h'$ with $v_i \in \dom(h') \cap \im(h')$, $gA_i \sub \dom(h')$ and $A_i \sub \dom(h'^{-1} \circ g \circ h')$. Let $U = \dom(h')$.

    By Lemma \ref{l: realisation disjoint}, there is a realisation $\bar{b}$ of $p_i$ with $\bar{b} \cap g\bar{b} = \emp$ and $\bar{b} \cap (U \cup g^{-1}U) = \varnothing$, and there is a realisation $\bar{c}$ of $h' \cdot \tp(\bar{b} / U)$ with $\bar{c} \cap g\bar{c} = \emp$ and $\bar{c} \cap g^{-1}(h'(U)) = \varnothing$. Let $h''$ be the extension of $h'$ by the map $\bar{b} \mapsto \bar{c}$ (extending so that $\dom(h'') \fg M$). By (Ex) there is $\bar{d} \sub M$ with $\tp(\bar{d} / U\bar{b}) = h''^{-1} \cdot \tp(g(\bar{c}) / h''(U)\bar{c})$ and $g\bar{b} \ind_{U\bar{b}} \bar{d}$. We let $h_i$ be the extension of $h''$ by $\bar{d} \mapsto g\bar{c}$ (again extending so that $\dom(h_i) \fg M$).

    We have $(g(\bar{b})\bar{d}) \cap (U\bar{b}) = \varnothing$. So by freeness we have $g\bar{b} \ind_{gA_i} \bar{d}$ and $g\bar{b} \ind_{(h_i^{-1}gh_i)(A_i)} \bar{d}$. By (Inv) we have $\bar{b} \ind_{A_i} g^{-1}(\bar{d})$ and $\bar{b} \ind_{[g, h_i](A_i)} g^{-1}(\bar{d})$, and as $g^{-1}(\bar{d}) = [g, h_i](\bar{b})$, we are done (where for the second independence we use (Inv) once more).
\end{proof}

Combining the above lemma with Fact \ref{f:Li main thm} we immediately have:

\begin{prop} \label{p:free SWIR simple}
    Let $M$ be a strong Fraïssé structure with free SWIR $\ind$. Suppose that:
    \begin{enumerate}
        \item[$(\ast)$] for any $g \in \Aut(M) \setminus \{\id\}$ and any exterior $1$-type $p$, some realisation of $p$ is not fixed by $g$.
    \end{enumerate}
    Then $\Aut(M)$ is simple.
\end{prop}

\begin{rem}
    In \cite{MT11}, Theorem 3.4(a) states that the following structures have simple automorphism groups: transitive free amalgamation structures which are not an indiscernible set, the random tournament and the generic $n$-anticlique-free oriented graph. A proof is given for transitive non-trivial free amalgamation structures, and then the authors sketch how to tweak the proof for the other two structures (in fact, the independence relation for the random tournament already appears in \cite{MT11}, without the general axiomatisation of SWIRs). Proposition \ref{p:free SWIR simple} enables a (very slightly) more streamlined presentation of this result: it suffices to show condition $(\ast)$ for each of these structures. For transitive non-trivial free amalgamation structures, condition $(\ast)$ is given by \cite[Corollary 2.10]{MT11}. For the other two structures, see Lemma \ref{l: move realisation of type} (the proof of this lemma requires a minor adaptation for the generic oriented graph omitting $n$-anticliques: rather than taking $\neg R(u_m, gv)$ and $\neg R(b, gu)$, we take $R(gv, u_m)$ and $R(gu, b)$).
\end{rem}

\subsection{The pointwise-stabiliser lemma}

\begin{notn} \label{n:distinct elements stabs}
    Let $M$ be a relational structure and let $n \geq 1$. We write $(M)^n$ for the set of $n$-tuples in $M^n$ which consist of distinct elements. We write $[M]^n$ for the set of subsets of $M$ of size $n$.

    Let $G \curvearrowright X$ be a group action, and let $A \sub X$. We write $G_{(A)}$ for the pointwise-stabiliser of $A$: that is, $G_{(A)} = \{g \in G \mid g \cdot a = a \text{ for all } a \in A\}$.
\end{notn}

\begin{defn}
    Let $M$ be a relational structure, and let $n \geq 1$. We say that $M$ is \emph{$n$-transitive} if the action of $\Aut(M)$ on the domain of $M$ is $n$-transitive: that is, for all $\bar{a}, \bar{b} \in (M)^n$ there is $g \in \Aut(M)$ with $g(\bar{a})=\bar{b}$. (Note that this implies that all elements of $(M)^n$ have the same type over $\varnothing$.)
\end{defn}

The below Lemma \ref{l: simple G_A simple G} shows that, under certain conditions, Steps 2 and 3 of the SWIR expansion method are automatic (see Subsection \ref{ss: intro SWIR exps}). Note that in the below Lemma \ref{l: simple G_A simple G} we only consider relational structures.

\begin{lem} \label{l: simple G_A simple G}
    Let $M$ be an $n$-transitive strong relational \Fr structure, and let $G = \Aut(M)$. Suppose that the following hold:
    \begin{enumerate}[label=(\roman*)]
        \item \label{abc tps} for all $\bar{a},\bar{b} \in (M)^n$ with $\bar{a} \cap \bar{b} = \varnothing$, there is $\bar{c} \in (M)^n$ with $\tp(\bar{a}, \bar{c}) = \tp(\bar{b}, \bar{c})$;
        \item \label{simple stabs} for all $A \sub M$ with $|A| = n$, the pointwise-stabiliser $G_{(A)}$ is simple.
    \end{enumerate}
    Then $G$ is simple.
\end{lem}
\begin{proof}
    Let $K \nrm G$, $K \neq 1$. Let $f \in K \setminus \{\id\}$, and let $v \in M$ with $v \neq f(v)$. Take $C \in [M]^n$ such that $\{v, f(v)\} \cap (C \cup f(C)) = \varnothing$. As $M$ has strong amalgamation and is relational, there is $g \in G$ fixing $C \cup f(C) \cup \{v\}$ pointwise with $g(f(v)) \neq f(v)$. Then $[g, f](v) \neq v$, and as $[g, f] \in K \cap G_{(C)}$ we have $K \cap G_{(C)} \neq 1$. As $M$ is $n$-transitive, for each $A \in [M]^n$ we have $h \in G$ with $h(A) = C$, so $[g, f]^h \in K \cap G_{(A)}$. Thus for all $A \in [M]^n$ we have $K \cap G_{(A)} \neq 1$, and as $G_{(A)}$ is simple we thus have $G_{(A)} \sub K$.

    Now let $g \in G$. If $g$ fixes $\geq n$ elements of $M$ then by the above we have $g \in K$. Otherwise $g$ fixes $\leq n - 1$ elements of $M$, and so has infinite support. So there exists $\bar{a} \in (M)^n$ with $\bar{a} \cap g\bar{a} = \varnothing$. By assumption, there is $\bar{c} \in (M)^n$ with $\tp(\bar{a}, \bar{c}) = \tp(g(\bar{a}), \bar{c})$. Take $f \in G$ with $f(\bar{a}, \bar{c}) = (g(\bar{a}), \bar{c})$. Then as $f(\bar{c}) = \bar{c}$ we have $f \in K$, and as $g^{-1}f\bar{a} = \bar{a}$ we have $g^{-1}f \in K$, so $g \in K$. Thus $K = G$, and hence $G$ is simple. 
\end{proof}

\section{The generic \texorpdfstring{$n$}{n}-hypertournament} \label{s: Tn}

Recall the definition of the generic $n$-hypertournament $\Tn$ from Definition \ref{d:Tn}. We first show that the techniques of Li, Macpherson, Tent and Ziegler cannot be applied to $\Tn$ directly.

\begin{lem} \label{l: Aut Tn no dense conj}
    For $n \geq 3$, the group $\Aut(\Tn)$ does not have a dense conjugacy class.
\end{lem}
\begin{proof}
    We use Fact \ref{f:KR dense ccl}. Let $A \in \Age(\Tn)$ consist of $2$ points, and let $f_A \in \Aut(A)$ be the automorphism swapping these points. Let $B \in \Age(\Tn)$ consist of $n - 2$ points, and let $f_B = \id_B$. Then there is no $(C, f_C) \in \Age(\Tn)_{\textnormal{p}}$ that $(A, f_A)$ and $(B, f_B)$ jointly embed into.
\end{proof}

\begin{lem} \label{l: Tn no SWIR}
    For $n \geq 3$, the generic $n$-hypertournament $\Tn$ does not have a SWIR.
\end{lem}
\begin{proof}
    This follows from Lemma \ref{l: Aut Tn no dense conj} and Lemma \ref{l: SWIR implies dense ccl}. Alternatively, we may show this directly. Suppose for a contradiction that $\Tn$ has a SWIR $\ind$. Take $(A, f_A), (B, f_B)$ from the proof of Lemma \ref{l: Aut Tn no dense conj}. By (Ex) we may assume $A, B \sub \Tn$ and $A \ind B$. But, as in the proof of Lemma \ref{l: SWIR implies dense ccl}, any pair of automorphisms of $A$ and $B$ must extend to an automorphism of $AB$, by (Inv), (Sta) and the ultrahomogeneity of $M$ -- contradiction.
\end{proof}

We now find an expansion of $\Tn$ with a free SWIR.

\begin{defn}
    Let $I$ be a finite set, and for each $i \in I$ let $n_i \geq 2$. Let $\vv{n}_I = (n_i)_{i \in I}$. Let $\mc{L}_I = \{R_i \mid i \in I\}$ be a relational language where each $R_i$, $i \in I$, has arity $n_i$. An \emph{$\vv{n}_I$-hypertournament} $T = (A, (R_i^T)^{}_{i \in I})$ is an $\mc{L}_I$-structure (writing $A$ for the domain and $(R_i^T)^{}_{i \in I}$ for the relations) such that, for each $i \in I$, the reduct $(A, R_i^T)$ is an $n_i$-hypertournament. We will not notationally distinguish the structure and its domain, and just write $T$ for both. Let $\mc{T}_{\vv{n}_I}$ denote the class of finite $\vv{n}_I$-hypertournaments; it is straightforward to see that $\mc{T}_{\vv{n}_I}$ has strong amalgamation. We call the \Fr limit of $\mc{T}_{\vv{n}_I}$ the \emph{generic $\vv{n}_I$-hypertournament}, denoted $\T_{\vv{n}_I}$.
\end{defn}

In the below definition, recall Notation \ref{n:distinct elements stabs}.

\begin{defn} \label{d:sign of ht edge}
    Let $T$ be an $\vv{n}_I$-hypertournament. Let $i \in I$. We define the \emph{sign map} $\rho^T_i : (T)^{n_i} \to \{+1, -1\}$ as follows:
    \[
        \rho^T_i(\bar{v}) =
        \begin{cases}
            +1 & \bar{v} \in R_i^T\\
            -1 & \bar{v} \notin R_i^T.
        \end{cases}
    \]
    When $T$ is clear from context we just write $\rho_i(\bar{v})$.

    It is immediate that, given two $\vv{n}_I$-hypertournaments $T, T'$ on the same domain $A$ with $\rho^T_i(\bar{v}) = \rho^{T'}_i(\bar{v})$ for all $i \in I$ and $\bar{v} \in (A)^{n_i}$, we have $T = T'$.
\end{defn}

\begin{notn}
    We write $\Sym_n$ and $\Alt_n$ for the symmetric group and the alternating group on $\{0, \cdots, n-1\}$, and for $\sigma \in \Sym_n$ we write $\eps(\sigma)$ for the sign of $\sigma$ in the usual permutation group sense: that is, $+1$ if $\sigma$ is an even permutation and $-1$ if $\sigma$ is odd.
\end{notn}

The following Lemma \ref{l: def ht from fn} is immediate.

\begin{lem} \label{l: def ht from fn}
    Let $A$ be a set, and let $n \geq 2$. Let $f : (A)^n \to \{+1, -1\}$ be a function such that, for all $\sigma \in \Sym_n$ and $(v_0, \cdots, v_{n-1}) \in (A)^n$, we have $f(v_{\sigma(0)}, \cdots, v_{\sigma(n-1)}) = \eps(\sigma) f(v_0, \cdots, v_{n-1})$. Then $(A, \{\bar{v} \in (A)^n \mid f(\bar{v}) = 1\})$ is an $n$-hypertournament (with sign map $f$).
\end{lem}

\begin{lem}\label{l: move realisation of type}
    Let $I$ be a finite set, and for each $i \in I$ let $n_i \geq 2$. Let $\vv{n}_I = (n_i)_{i \in I}$. Let $g \in \Aut(\T_{\vv{n}_I}) \setminus \{\id\}$. Let $p$ be an exterior $1$-type. Then some realisation of $p$ is not fixed by $g$.
\end{lem}
\begin{proof}
    Let $j \in I$. Let $v \in \T_{\vv{n}_I}$ with $gv \neq v$. As $\T_{\vv{n}_I}$ has strong amalgamation, there are infinitely many disjoint $(n_j - 1)$-tuples $\bar{u}_0, \cdots$ consisting of distinct elements of $\T_{\vv{n}_I}$ such that, for each $m < \omega$, we have $\T_{\vv{n}_I} \models R_j(\bar{u}_m, v) \wedge \neg R_j(\bar{u}_m, gv)$ and hence $g\bar{u}_m \neq \bar{u}_m$. So $g$ has infinite support.

    Let $A$ be the parameter set of $p$. As $g$ has infinite support, there exists an $(n_j - 1)$-tuple $\bar{u} \sub \T_{\vv{n}_I}$ with $\bar{u}, g\bar{u}, A$ pairwise disjoint. Taking $b \in \T_{\vv{n}_I}$ with $b \models p$ and $\T_{\vv{n}_I} \models R_j(b, \bar{u}) \wedge \neg R_j(b, g\bar{u})$, we have $gb \neq b$ as required.
\end{proof}

\begin{prop}\label{p:SWIR of n2tournament}
    Let $I$ be a finite set, and for each $i \in I$ let $n_i \geq 2$. Let $\vv{n}_I = (n_i)_{i \in I}$. Suppose that there exists $\delta \in I$ with $n_\delta = 2$. Then the generic $\vv{n}_I$-hypertournament $\T_{\vv{n}_I}$ has a free SWIR, and $\Aut(\T_{\vv{n}_I})$ is simple.
\end{prop}
\begin{proof}
    Let $A, B, C \fin \T_{\vv{n}_I}$. We define $B \ind_A C$ if $(B \setminus A) \cap (C \setminus A) = \emp$ and the following hold:
    \begin{itemize}
        \item for each $b \in B \setminus A$, $c \in C \setminus A$ we have $\rho_\delta(b, c) = 1$;
        \item for each $i \in I$ and each $n_i$-tuple $\bar{v} \in ((A \cup B \cup C))^{n_i}$ of distinct elements with $(\bar{v} \cap (B \setminus A) \neq \varnothing) \wedge (\bar{v} \cap (C \setminus A) \neq \varnothing)$, we have \[\rho_i(\bar{v}) = \prod_{m < m' < n_i} \rho_\delta(v_m, v_{m'}).\]
    \end{itemize}
    (Inv), (Sta), (Mon) and freeness are straightforward. We show (REx), with (LEx) similar. Let $A, B, C \fin \T_{\vv{n}_I}$. Let $D$ be the free amalgam of $BA, CA$ over $A$: we may assume that the amalgam embedding $BA \to D$ is the identity, and we write $C'A$ for the image of $CA$. Define a structure $E$ from $D$ by extending each relation of $D$ as follows. For $b \in B \setminus A$, $c \in C' \setminus A$, we specify $(b, c) \in R_\delta^E$. As we have specified $R_\delta^E$, we then have the map $\rho_\delta^E$ as in Definition \ref{d:sign of ht edge}. We now define, for each $i \in I$, a function $f_i : (E)^{n_i} \to \{+1, -1\}$ by 
    \[
        f_i(\bar{v}) = 
        \begin{cases}
            \rho_i^{BA}(\bar{v}) & \bar{v} \sub BA\\
            \rho_i^{C'A}(\bar{v}) & \bar{v} \sub C'A\\
            \prod_{m < m' < n_i} \rho_\delta^E(v_m, v_{m'}) & (\bar{v} \cap (B \setminus A) \neq \varnothing) \,\wedge\, (\bar{v} \cap (C' \setminus A) \neq \varnothing).
        \end{cases}\]
    (Note that $f_\delta^{} = \rho_\delta^E$.) Recall that the sign $\eps(\sigma)$ of a permutation $\sigma \in \Sym_n$ is the parity of the number of pairs $(x, y) \in (n)^2$ with $x < y$ and $\sigma(x) > \sigma(y)$. Thus the functions $f_i$ satisfy the condition of Lemma \ref{l: def ht from fn}, giving a $\vv{n}_I$-hypertournament $E$ extending the $\vv{n}_I$-hypertournaments $BA$, $C'A$. We may then assume $E$ is embedded over $BA$ in $\T_{\vv{n}_I}$ using the extension property, and it is clear that $B \ind_A C'$ as required.

    We then have simplicity of $\Aut(\T_{\vv{n}_I})$ by Lemma \ref{l: move realisation of type} and Proposition \ref{p:free SWIR simple}.
\end{proof}

\begin{lem} \label{l: ptstab of n-2 set in Tn is simple}
    Let $n \geq 2$. Let $A \fin \Tn$, $|A| = n - 2$. Then the pointwise-stabiliser $\Aut(\Tn)_{(A)}$ is simple. 
\end{lem}
\begin{proof}
    Recall that we write $R^{\T_n}$ for the edge relation of the generic $n$-hypertournament $\T_n$, and $\mc{L}_n^{\textnormal{T}}$ for the language of this structure.
    
    For each $A' \sub A$, take an arbitrary ordering $\bar{u}_{A'}$ of $A'$, and let $\mc{U} = \{\bar{u}_{A'} \mid A' \sub A\}$. For $\bar{u} \in \mc{U}$, let $n_{\bar{u}} = n - |\bar{u}|$. Let $\vv{n}_{\mc{U}} = (n_{\bar{u}})_{\bar{u} \in \mc{U}}$ and let $\mc{L}_{\mc{U}} = \{R_{\bar{u}} \mid \bar{u} \in \mc{U}\}$ be a relational language where the arity of each $R_{\bar{u}}$ is $n_{\bar{u}}$. Define an $\mc{L}_{\mc{U}}$-structure $T'$ as follows:
    \begin{itemize}
        \item $\dom(T') = \dom(\T_n) \setminus \dom(A)$;
        \item for each $\bar{u} \in \mc{U}$ we define $R^{T'}_{\bar{u}} = \{\bar{v} \in (\dom(T'))^{n_{\bar{u}}} \mid (\bar{u}, \bar{v}) \in R^{\T^n}\}$.
    \end{itemize}

    We first observe that $T'$ is an $\vv{n}_{\mc{U}}$-hypertournament: this follows from the fact that, in the standard permutation action $\Alt_n \curvearrowright n$, the pointwise-stabiliser of $X \sub n$ is the alternating group on the complement of $X$.

    To see that $T'$ is isomorphic to the generic $\vv{n}_{\mc{U}}$-hypertournament $\T_{\vv{n}_{\mc{U}}}$, we verify that the two structures have equal age and that $T'$ has the extension property. Let $B' \fin T'$ (where potentially $B' = \varnothing$) and let $f : B' \to C'$ be an embedding in $\Age(\T_{\vv{n}_{\mc{U}}})$. We may assume $f|_{B'} = \id_{B'}$ and $\dom(C') \cap \dom(A) = \varnothing$. Define an $\mc{L}_n^{\textnormal{T}}$-structure $E$ on $\dom(A) \cup \dom(C')$ by: for each $\bar{u} \in \mc{U}$ and $\bar{v} \in (\dom(C'))^{n_{\bar{u}}}$ with $\bar{v} \in R^{C'}_{\bar{u}}$, we specify $\sigma(\bar{u}, \bar{v}) \in R^E$ for all even permutations $\sigma$ of the $n$-tuple $(\bar{u}, \bar{v})$. Then $E$ is an $n$-hypertournament extending the $n$-hypertournament induced by $\T_n$ on $\dom(A) \cup \dom(B')$, which we denote by $D$. So by the extension property of $\T_n$ we may realise $E$ in $\T_n$ extending $D$, and then by the definition of $T'$ we have that the $\vv{n}_{\mc{U}}$-hypertournament induced by $T'$ on $\dom(E) \setminus \dom(A)$ is a realisation of $C'$ extending $B'$, as required. So $T' \cong \T_{\vv{n}_{\mc{U}}}$.
    
    It is straightforward to see that we have an isomorphism $\Aut(\T_n)_{(A)} \to \Aut(T')$, $g \mapsto g|_{\dom(T')}$. So $\Aut(\T_n)_{(A)} \cong \Aut(\T_{\vv{n}_{\mc{U}}})$, and the latter is simple by Proposition \ref{p:SWIR of n2tournament}. 
\end{proof}

\begin{manualthm}{A}
    Let $n \geq 3$. The automorphism group of the generic $n$-hypertournament $\mb{T}_n$ is simple.
\end{manualthm}
\begin{proof}
    This follows by Lemma \ref{l: simple G_A simple G}: it is immediate that $\T_n$ is an $(n-2)$-transitive strong \Fr structure, condition \ref{abc tps} is easily verified, and condition \ref{simple stabs} is given by Lemma \ref{l: ptstab of n-2 set in Tn is simple}.
\end{proof}

\section{The dense \texorpdfstring{$\frac{2\pi}{n}$}{2pi/n}-local order \texorpdfstring{$\mathbb{S}(n)$}{Sn}} \label{s: Sn}

Recall the language $\mc{L}^\ri_n$ and the $\mc{L}^\ri_n$-structure $\Sn$ from Definition \ref{d:Sn}.

Before giving a new proof of Fact B, we first discuss the original result of Droste, Giraudet and Macpherson (\cite{DGM95}). The reader may be familiar with local orders in the case $n = 2, 3$ defined as oriented graphs, rather than as the binary relational structures $\Sn$. Let $D$ be the set of points on the unit circle with rational argument. For $n \geq 2$, let $S(n)$ be the oriented graph with domain $D$ defined as follows: for distinct $u, v \in D$, we define $u \ra v$ if the angle subtended at the origin by the anticlockwise arc from $u$ to $v$ is $< \frac{2\pi}{n}$. The oriented graphs $S(2)$, $S(3)$ are ultrahomogeneous, via an immediate adaptation of Lemma \ref{l: Sn ultrahomog}, and $S(n)$ is not ultrahomogeneous for $n \geq 4$ (this is not difficult to see). Nonetheless, for each $n \geq 2$, it is straightforward to see that $S(n)$ and $\Sn$ are interdefinable, and thus $\Aut(S(n)) = \Aut(\Sn)$. In \cite{DGM95}, Droste, Giraudet and Macpherson showed the simplicity of $\Aut(S(n))$ for all $n \geq 2$ (without discussing the binary relational structures $\Sn$), and thus this result gives the simplicity of $\Aut(\Sn)$ for all $n \geq 2$, which is Fact B.

We now begin our new approach to showing Fact B, via the SWIR expansion method.

\begin{defn} \label{d: Sn basic notn}
    Let $U$ denote the unit circle in the complex plane. For $\theta \in \R$ we write $\rho(\theta) = e^{i\theta}$. For $u, v \in U$, we write $\alpha(u, v)$ for the angle subtended at the origin by the anticlockwise arc from $u$ to $v$, with $0 \leq \alpha(u, v) < 2\pi$. We consider $U$ as a $\mc{L}^\ri_n$-structure containing $\Sn$ in the natural manner: for $u, v \in U$ and $j < n$, we define $U \models S_j(u, v)$ if $\frac{2\pi j}{n} < \alpha(u, v) < \frac{2\pi(j+1)}{n}$. We define open intervals in $U$ as follows. For $u, u' \in U$ with $0 < \alpha(u, u') \leq \frac{2\pi}{n}$, we define the open interval $\langle u, u' \rangle := \{v \in U \mid S_0(u, v) \wedge S_0(v, u')\}$; note that $\langle u, u' \rangle$ is linearly ordered by $S_0$. We only use the notation $\langle u, u' \rangle$ specifically for open intervals of the structure $U$: when considering substructures $V \sub U$, we will be careful to write $\langle u, u' \rangle \cap V$.
    
    For $v \in U$ and $k \in \Z$, let $v^{\ua k} = e^{i \frac{2\pi k}{n}} v$, and for $A \sub U$ let $A^{\ua k} = \{v^{\ua k} \mid v \in A\}$. For $u, v \in U$ we have:
    \begin{equation} \label{Sl translate}
    S_l(u, v) \Leftrightarrow S_{l + k' - k \Mod{n}}(u^{\ua k}, v^{\ua k'}).
    \end{equation}

    For $A \sub U$, we define $\hat{A} = \bigcup_{0 \leq k < n} A^{\ua k}$. (Note that if $A \sub \Sn$, as every element of $A$ has rational argument, we have $A^{\ua k} \cap A^{\ua k'} = \varnothing$ for distinct $k, k'$ with $0 \leq k, k' < n$.) Let $f : A_0 \to A_1$ be an isomorphism of substructures of $\Sn$. Then by (\ref{Sl translate}) we have that $f$ extends to an isomorphism $\hat{f} : \hat{A}_0 \to \hat{A}_1$, $\hat{f}(a^{\ua k}) = f(a)^{\ua k}$.

    Let $A \fin \Sn$ be non-empty. For $u, u' \in \hat{A}$ with $0 < \alpha(u, u') \leq \frac{2\pi}{n}$, we say that $(u, u')$ is a \emph{cut} of $\hat{A}$ if the open interval $\langle u, u' \rangle$ is disjoint from $\hat{A}$.
\end{defn}

\begin{lem} \label{l: Sn ultrahomog}
    Let $f : A_0 \to A_1$ be an isomorphism between non-empty finite substructures of $\Sn$, and extend $f$ to the isomorphism $\hat{f} : \hat{A}_0 \to \hat{A}_1$, $\hat{f}(a^{\ua k}) = f(a)^{\ua k}$. Let $(u, u')$ be a cut of $\hat{A}_0$; observe that $(\hat{f}(u), \hat{f}(u'))$ is a cut of $\hat{A}_1$. Let $b_0, b_1 \in \Sn$ with $b_0 \in \langle u, u' \rangle$, $b_1 \in \langle \hat{f}(u), \hat{f}(u') \rangle$. Then the extension of $\hat{f}$ by $b_0 \mapsto b_1$ is an isomorphism.

    Thus, the structure $\Sn$ has the extension property, and so is ultrahomogeneous.
\end{lem}
\begin{proof}
    Let $C_0 = \hat{A}_0 \cap \langle u, u^{\ua 1} \rangle$. The open intervals $\langle u, u^{\ua 1} \rangle$, $\langle \hat{f}(u), \hat{f}(u)^{\ua 1} \rangle$ are each linearly ordered by $S_0$, and $\hat{f}$ is order-preserving. We have $S_0(u, b_0) \wedge S_0(\hat{f}(u), b_1)$, and $S_0(b_0, c) \wedge S_0(b_1, \hat{f}(c))$ for each $c \in C_0$. For each $a \in A_0 \setminus \{u^{\ua l} \mid l < n\}$, there is a unique $k < n$ with $a^{\ua k} \in C_0$. So, using (\ref{Sl translate}), we have $\hat{f} \cdot \tp(b_0 / \hat{A}_0) = \tp(b_1 / \hat{A}_1)$ as required.
\end{proof}

\begin{lem} \label{l: Aut Sn no dense cong}
    $\Aut(\Sn)$ does not have a dense conjugacy class.
\end{lem}
\begin{proof}
    Let $q \in \Q$ with $\frac{2\pi}{n} \cdot \frac{n-1}{n} < q < \frac{2\pi}{n}$. Let $a_k = \rho(kq)$ for $0 \leq k \leq n$, and let $A$ be the substructure of $\Sn$ with domain $\{a_0, \cdots, a_n\}$. Let $f_A : \dom(A) \to \dom(A)$ be the map $a_k \mapsto a_{k+1 \Mod{n+1}}$. We now show $f_A \in \Aut(A)$. For $1 \leq l \leq n$, as $\frac{l-1}{l} \leq \frac{n-1}{n}$ we have $\frac{2\pi}{n} \cdot \frac{l-1}{l} < q < \frac{2\pi}{n}$, and so $\frac{2\pi(l-1)}{n} < lq < \frac{2\pi l}{n}$. Thus for $0 \leq k < k + l \leq n$ we have $\frac{2\pi (l-1)}{n} < (k+l)q - kq < \frac{2\pi l}{n}$ and hence $S_{l-1}(a_k, a_{k+l})$. As $S_r(u, v) \Leftrightarrow S_{n-1-r}(v, u)$ for all $0 \leq r \leq n-1$ and distinct $u, v \in U$, we have $S_{k' - k - 1 \Mod{n+1}}(a_k, a_{k'})$ for all distinct $k, k'$, and so $f_A \in \Aut(A)$.

    Let $B = \{b\} \in \Age(\Sn)$, and let $f_B = \id_B$. Suppose for a contradiction there is $(C, f_C) \in \Age(\Sn)_{\textnormal{p}}$ into which $(A, f_A)$, $(B, f_B)$ both embed. Identifying $A, B$ with their images in $C$, there is $l$ with $S_l(b, a_k)$ for all $0 \leq k \leq n$. But for all $k$ we have $S_0(a_k, a_{k+1 \Mod{n+1}})$, contradicting the fact that $\{c \in C \mid S_l(b, c)\}$ is linearly ordered by $S_0$. Thus $\Aut(\Sn)$ does not have a dense conjugacy class by Fact \ref{f:KR dense ccl}.
\end{proof}

\begin{defn}
    Let $\mc{L}^<_n$ be a language consisting of a binary relation $<$ and $n$ unary predicates $\chi_0, \cdots, \chi_{n-1}$. We call an $\mc{L}^<_n$-structure $A$ an \emph{$n$-coloured linear order} if $<^A$ is a linear order and $\{\chi_i^A \mid i < n\}$ is a partition of the domain of $A$. Given an $n$-coloured linear order $A$, we define a function $\chi : \dom(A) \to n$ by: $\chi(a) = i$ if $A \models \chi_i(a)$. If $\chi(a) = i$, we say $a$ has \emph{colour} $i$. We also use this terminology when defining $n$-coloured linear orders.
    
    Let $\Q_n$ be the \Fr limit of the class of finite $n$-coloured linear orders. It is straightforward to see that the $\{<\}$-reduct of $\Q_n$ is isomorphic to $(\Q, <)$. We say that an $n$-coloured linear order $A$ is \emph{colour-dense} if, for all $a, b \in A$ with $a < b$ and each $i < n$, there is $c \in A$ of colour $i$ with $a < c < b$. Note that by \Frn's theorem any non-empty countable $n$-coloured linear order which is colour-dense and without endpoints is isomorphic to $\Q_n$.
\end{defn}

\begin{defn}
    Let $a \in \Sn$. We define an $n$-coloured linear order $I_a$ with domain $\langle a, a^{\ua 1} \rangle \cap \dom(\widehat{\Sn})$ as follows: for $u, v \in \dom(I_a)$, we define $u < v$ if $S_0(u, v)$, and for $u \in \dom(I_a)$ we define that $u$ has colour $k$ if $u \in \Sn^{\ua k}$.

    Let $g \in \Aut(\Sn)_a$. As $\dom(I_a) = \{v \in \widehat{\Sn} \mid S_0(a, v)\}$, we have $\hat{g}(\dom(I_a)) = \dom(I_a)$, and as $\hat{g}$ preserves $S_0$ and $\hat{g}(\Sn^{\ua k}) = \Sn^{\ua k}$ for $k < n$, we have $\hat{g}|_{\dom(I_a)} \in \Aut(I_a)$.
\end{defn}

\begin{lem} \label{l: Ga Ia}
    Let $a \in \Sn$. Then:
    \begin{enumerate}[label=(\roman*)]
        \item \label{Ga Aut Ia iso} the map $\zeta : \Aut(\Sn)_a \to \Aut(I_a)$ given by $\zeta(g) = \hat{g}|_{\dom(I_a)}$ is an isomorphism;
        \item \label{Ia Qn iso} $I_a$ is isomorphic to $\Q_n$.
    \end{enumerate}
\end{lem}
\begin{proof}
    \ref{Ga Aut Ia iso}: Let $h \in \Aut(I_a)$. As $h$ is colour-preserving we have $h(\Sn^{\ua k} \cap \dom(I_a)) \sub \Sn^{\ua k}$. Define a map $\eta(h) : \dom(\Sn) \to \dom(\Sn)$ by $\eta(h)(a) = a$ and, for $v \in \dom(\Sn) \setminus \{a\}$, by $\eta(h)(v) = h(v^{\ua k})^{\ua -k}$, where $k$ is such that $v^{\ua k} \in I_a$. It is straightforward to check that $\eta(h) \in \Aut(\Sn)_a$ using (\ref{Sl translate}), and that $\eta$ is the inverse of $\zeta$; we leave the verification to the reader.

    \ref{Ia Qn iso}: it suffices to show that $I_a$ is colour-dense and without endpoints. Let $u, v \in I_a$ with $u < v$. Let $k < n$. We have $S_0(u^{\ua -k}, v^{\ua -k})$, and as $\Sn$ consists of the points of $U$ with rational argument, there is $w \in \Sn$ with $S_0(u^{\ua -k}, w) \wedge S_0(w, v^{\ua -k})$. Then $w^{\ua k} \in I_a$ has colour $k$, and $u < w^{\ua k} < v$. One similarly sees that $I_a$ is without endpoints -- we leave this to the reader.
\end{proof}

\begin{defn}
    We define the following subsets of $\Aut(\Q_n)$:
     \begin{align*}
        L(\Q_n) &= \{g \in \Aut(\Q_n) \mid \ex a \in \Q_n \text{ with } (-\infty, a) \text{ fixed pointwise by } g\}; \\
        R(\Q_n) &= \{g \in \Aut(\Q_n) \mid \ex a \in \Q_n \text{ with } (a, \infty) \text{ fixed pointwise by } g\}; \\
        B(\Q_n) &= L(\Q_n) \cap R(\Q_n).
    \end{align*}    

    We call the elements of $L(\Q_n)$, $R(\Q_n)$, $B(\Q_n)$ \emph{left-bounded}, \emph{right-bounded} and \emph{bounded} automorphisms of $\Q_n$ respectively. We call the elements of $\Aut(\Q_n) \setminus (R(\Q_n)\cup L(\Q_n))$ \emph{unbounded} automorphisms of $\Q_n$.
\end{defn}

We now determine the lattice of normal subgroups of $\Aut(\Q_n)$ for $n \geq 1$:

\begin{lem}\label{l: normal subgroups Q_n}
    The non-trivial proper normal subgroups of $\Aut(\Q_n)$ are precisely $L(\Q_n)$, $R(\Q_n)$ and $B(\Q_n)$.
\end{lem}

Note that the above result in the case $n = 1$ (giving the normal subgroups of $\Aut(\Q)$) is classical, often attributed to Higman and first shown in \cite{Llo64}, \cite{Hig54}. (\cite{Hig54} shows that $B(\Q)$ is simple, and \cite{Llo64} finds the full lattice of normal subgroups of $\Aut(\Q)$. See \cite{Gla82} for an exposition of the result for $\Aut(\Q)$ via classical methods.) The result for $n > 1$ is new, though the authors imagine it could be obtained via classical methods as a straightforward generalisation of the case $n = 1$. 

We prove Lemma \ref{l: normal subgroups Q_n} by a new approach, using the below result of Li:

\begin{lem}[immediate generalisation of {\cite[Cor.\ 5.2.21]{Li21}}] \label{l: unbounded norm closure}
    Let $\mc{M}$ be a strong \Fr structure with a SWIR, and let $\mc{M}^<$ be its generic linear order expansion. Let $g\in \Aut(\mc{M}^<)$ be unbounded. Then the normal closure of $g$ is $\Aut(\mc{M}^<)$.
\end{lem}

Li only states and proves the above result for the case $\mc{M}^< = \Q$ (see \cite[Corollary 5.2.21]{Li21}), and in this case the result is classical (\cite{Llo64}). The proof of Li using SWIRs is however new, and her argument directly generalises to give Lemma \ref{l: unbounded norm closure}: in the proof of \cite[Corollary 5.2.21]{Li21}, Li applies Theorem 5.2.12, Lemma 5.2.16 and Corollary 4.2.4 from the same paper to $\Q$, but these results are actually stated for $\mc{M}^<$ as above -- so her argument immediately generalises, giving Lemma \ref{l: unbounded norm closure}.

\begin{proof}[Proof of Lemma \ref{l: normal subgroups Q_n}]
    Let $g \in \Aut(\Q_n) \setminus \{\id\}$. Let $C$ be the convex closure of $\supp(g)$. We first show that for each $c \in C$, there are $a, b \in \supp(g)$ with $a < c < b$: if $c \notin \supp(g)$ this is immediate, and if $c \in \supp(g)$ we have $gc, g^{-1}c \in \supp(g)$ and $(g^{-1}c < c < gc) \vee (gc < c < g^{-1}c)$. So $g|_C \in \Aut(C)$ is unbounded within $C$, and $C$ is without endpoints. As $\Q_n$ is colour-dense, so is $C$ by convexity, and thus $C \cong \Q_n$. Let $\theta : C \to \Q_n$ be an isomorphism. Then $\theta$ induces an isomorphism $\Theta : \Aut(C) \to \Aut(\Q_n)$ given by $\Theta(f) = \theta \circ f \circ \theta^{-1}$. As $g|_C \in \Aut(C)$ is unbounded, we have that $\Theta(g|_C) \in \Aut(\Q_n)$ is unbounded and so has normal closure $\Aut(\Q_n)$ by Lemma \ref{l: unbounded norm closure}. So $g|_C$ has normal closure $\Aut(C)$. Let $K$ be the normal closure of $g$ in $\Aut(\Q_n)$. Via the isomorphism $\Aut(C) \to \{h \in \Aut(\Q_n) \mid \supp(h) \sub C\}$ given by $f \mapsto f \cup \id_{\Q_n \setminus C}$ we have $\{h \in \Aut(\Q_n) \mid \supp(h) \sub C\} \sub K$.

    We now show the following:
    \begin{enumerate}[label=(\roman*)]
        \item \label{g unbd Aut Qn} if $g$ is unbounded then $K = \Aut(\Q_n)$;
        \item \label{g bd BQn} if $g$ is bounded then $K = B(\Q_n)$;
        \item \label{g unbd abv LQn} if $g$ is unbounded above and bounded below then $K = L(\Q_n)$;
        \item \label{g unbd blw RQn} if $g$ is unbounded below and bounded above then $K = R(\Q_n)$.
    \end{enumerate}

    \ref{g unbd Aut Qn}: As $g$ is unbounded we have $C = \Q_n$, and as $\{h \in \Aut(\Q_n) \mid \supp(h) \sub C\} \sub K$ we have $K = \Aut(\Q_n)$.

    \ref{g bd BQn}: Let $a, b \in C$ with $a < b$. Let $f \in B(\Q_n)$, and let $a', b' \in \Q_n$ with $\supp(f) \sub [a', b']$. Let $h \in \Aut(\Q_n)$ with $h(a) = a'$, $h(b) = b'$. Then $\supp(f^h) \sub C$, so $f^h \in K$ and thus $f \in K$. As $g$ is itself bounded and $B(\Q_n) \nrm \Aut(\Q_n)$, we have $K = B(\Q_n)$. 

    \ref{g unbd abv LQn}: The argument is similar to \ref{g bd BQn}. Let $a \in C$. As $g$ is unbounded above we have $[a, \infty) \sub C$. Let $f \in L(\Q_n)$, and let $a' \in \Q_n$ with $\supp(f) \sub [a', \infty)$. Let $h \in \Aut(\Q_n)$ with $h(a) = a'$. Then $\supp(f^h) \sub C$; the rest of the argument follows \ref{g bd BQn}. The proof of \ref{g unbd blw RQn} is likewise similar to \ref{g unbd abv LQn}.
    
    Note that for any $g \in L(\Q_n) \setminus B(\Q_n)$, $f \in R(\Q_n) \setminus B(\Q_n)$, the composition $fg$ is unbounded, and so we are done.
\end{proof}

\begin{defn} \label{d: La Ra Ba}
    Let $a \in \Sn$. We define the following subgroups of $\Aut(\Sn)_a$:
    \begin{align*}
        L_a &= \{g \in \Aut(\Sn)_a \mid \all k < n\, \ex b_k \in \langle a^{\ua k}, a^{\ua k+1} \rangle \cap \Sn \text{ with } \langle a^{\ua k}, b_k \rangle \cap \Sn \text{ fixed pointwise by } g\};\\
        R_a &= \{g \in \Aut(\Sn)_a \mid \all k < n\, \ex b_k \in \langle a^{\ua k}, a^{\ua k+1} \rangle \cap \Sn \text{ with } \langle b_k, a^{\ua k+1} \rangle \cap \Sn \text{ fixed pointwise by } g\};\\
        B_a &= L_a \cap R_a.
    \end{align*}

    Note that if $g \in \Aut(\Sn)_a$ fixes some $\langle u, v \rangle \cap \Sn \sub \langle a, a^{\ua 1} \rangle$ pointwise, then for all $k < n$, as $\Sn^{\ua k}$ is dense in $U$, we have that $g$ fixes $\langle u^{\ua k}, v^{\ua k} \rangle \cap \Sn$ pointwise, and so $L_a, R_a, B_a$ can also be described in terms of $\langle a, a^{\ua 1} \rangle$ alone:
    \begin{align*}
        L_a &= \{g \in \Aut(\Sn)_a \mid \ex b \in \langle a, a^{\ua 1} \rangle \cap \Sn \text{ with } \langle a, b \rangle \cap \Sn \text{ fixed pointwise by } g\};\\
        R_a &= \{g \in \Aut(\Sn)_a \mid \ex b \in \langle a, a^{\ua 1} \rangle \cap \Sn \text{ with } \langle b, a^{\ua 1} \rangle \cap \Sn \text{ fixed pointwise by } g\}.
    \end{align*}
\end{defn}

\begin{lem} \label{l: norm subgps of point stab}
    For each $a \in \Sn$, the non-trivial proper normal subgroups of $\Aut(\Sn)_a$ are $L_a$, $R_a$ and $B_a$.
\end{lem}

\begin{proof}
    This is immediate by Lemma \ref{l: normal subgroups Q_n} and Lemma \ref{l: Ga Ia}.
\end{proof}

\begin{lem} \label{l: K cap Ga nontrivial}
    Let $K \nrm \Aut(\Sn)$, $K \neq 1$. Let $a \in \Sn$. Then $K \cap \Aut(\Sn)_a \neq 1$.
\end{lem}

\begin{proof}
    Take $f \in K \setminus \{\id\}$. If $f \in \Aut(\Sn)_a$, then the claim immediately follows. So suppose $f \notin \Aut(\Sn)_a$, and thus there is $k < n$ with $S_k(a, f(a))$. Hence $S_{n-1-k}(f(a), a)$. By replacing $f$ with $f^{-1}$ if necessary, we may assume $k \leq n -1-k$. As we cannot have $S_0(a, f(a)) \wedge S_0(f(a), a)$, we therefore have $n-1-k > 0$ and thus $\neg S_0(f(a), a)$. Let $b \in \Sn$ with $S_0(a, b) \wedge S_k(b, f(a))$. Then $S_0(f(a), f(b))$, so $\{b, f(b)\} \cap \{a, f(a)\} = \varnothing$. If $f(b) = b$, then $S_0(a, b) \wedge S_0(f(a), b)$, which implies $k=0$ and hence $S_0(b, f(a))$, contradiction. So $f(b) \neq b$. 

    Take $g \in \Aut(\Sn)$ such that $g$ fixes each of $a, f(a), b$ but does not fix $f(b)$ (the existence of such $g$ is straightforward by ultrahomogeneity of $\Sn$ and the fact that $\Age(\Sn)$ has strong amalgamation). We have $[f, g](a) = a$, so $[f, g] \in \Aut(\Sn)_a$. Also $[f, g] = f^{-1} f^g \in K$, and $[f, g] \neq \id$ as $[f, g](b) \neq b$.
\end{proof}

\begin{lem} \label{l: pt stabs sub norm subgp Aut Sn}
    Let $K \nrm \Aut(\Sn)$, $K \neq 1$. Then $\Aut(\Sn)_a \sub K$ for all $a \in \Sn$.
\end{lem} 

\begin{proof}
    By Lemma \ref{l: K cap Ga nontrivial} and Lemma \ref{l: norm subgps of point stab} we have $B_a \sub K$ for each $a \in \Sn$. Let $a \in \Sn$ and let $f \in L_a$. There is $b \in \langle a, a^{\ua 1} \rangle \cap \Sn$ such that $f$ fixes $\langle a, b \rangle \cap \Sn$ pointwise. Take $c \in \langle a, b \rangle \cap \Sn$. Then $f$ is the identity on $(\langle c, b \rangle \cup \langle a^{\ua 1}, c^{\ua 1} \rangle) \cap \Sn$, so $f \in B_c$ and thus $f \in K$. So for all $a \in \Sn$ we have $L_a \sub K$, and by an analogous argument we also have $R_a \sub K$. So for all $a \in \Sn$ we have $K \cap \Aut(\Sn)_a \supseteq L_a \cup R_a$, and as $K \cap \Aut(\Sn)_a \nrm \Aut(\Sn)_a$, we therefore have $K \cap \Aut(\Sn)_a = \Aut(\Sn)_a$ by Lemma \ref{l: norm subgps of point stab}.
\end{proof}

\begin{manualfact}{B}
    Let $n \geq 2$. The automorphism group of the dense $\frac{2\pi}{n}$-local order $\Sn$ is simple.
\end{manualfact}
\begin{proof}
    Let $K \nrm \Aut(\Sn)$, $K \neq 1$. 

    Let $f \in \Aut(\Sn) \setminus \{\id\}$ be such that there exists $a \in \Sn$ with $S_0(a, f(a)) \vee S_0(f(a), a)$. Then there exists $b \in \Sn$ with $S_0(a, b) \wedge S_0(f(a), b)$, and so by ultrahomogeneity of $\Sn$ there exists $g \in \Aut(\Sn)$ fixing $b$ with $gf(a) = a$. Then $g \in \Aut(\Sn)_b$ and $gf \in \Aut(\Sn)_a$, so by Lemma \ref{l: pt stabs sub norm subgp Aut Sn} we have $g \in K$ and $gf \in K$, and thus $f \in K$.
    
    We now show by induction on $k < n$ that for all $f \in \Aut(\Sn) \setminus \{\id\}$ such that there exists $a \in \Sn$ with $S_k(a, f(a))$, we have $f \in K$. This then immediately implies $K = \Aut(\Sn)$.

    The case $k = 0$ has already been shown. Suppose $k \geq 1$. If $a^{\ua k+1} = a$, then $S_0(f(a), a)$, and again we have already dealt with this case. Suppose $a^{\ua k+1} \neq a$. Take $b \in \Sn$ with $S_0(f(a)^{\ua 1}, b) \wedge S_0(b, a^{\ua k+2})$, and write $b = a^{\ua k+2} e^{-i \theta}$. Take $b' = a^{\ua 2}e^{-i \theta'} \in \Sn$ with $0 < \theta' < \theta$. Then $S_1(f(a), b) \wedge S_1(a, b')$, and $\alpha(b', b) = \frac{2\pi k}{n} - \theta + \theta'$, so $S_{k-1}(b', b)$. By ultrahomogeneity of $\Sn$, there is $g \in \Aut(\Sn)$ with $gf(a) = a$, $g(b) = b'$. By the induction hypothesis, we have $g^{-1} \in K$. We have $gf \in \Aut(\Sn)_a \sub K$ and thus $f \in K$, completing the induction.
\end{proof}

\section{Ultrahomogeneous oriented graphs} \label{s: ultrahomog or graphs}

We now consider several countably infinite ultrahomogeneous oriented graphs from Cherlin's list \cite{Che98}, determining the normal subgroups of their automorphism groups. (See also \cite{PS20}, \cite{JLNW14} for further discussion of the oriented graphs in this list.) We show:

\begin{manualthm}{C}
    Recall the following structures from Cherlin's list:
    \begin{itemize}
        \item $\Sth$: the double cover of the dense local order $S(2)$;
        \item $\mb{D}_n$: the generic $n$-partite tournament;
        \item $\mb{F}$: the generic $\omega$-partite $\vec{C}_4$-tournament;
        \item $\mb{P}(3)$: the twisted partial order.
    \end{itemize}
    We have the following classification of the normal subgroups of the automorphism groups of these structures:
    \begin{enumerate}[label=(\roman*)]
        \item the only non-trivial proper normal subgroup of $\Aut(\Sth)$ is the subgroup generated by the involution sending each point to its antipodal point (\ref{ss: Sth});
        \item the normal subgroups of $\Aut(\Dn)$ are precisely $\{\pi^{-1}(K) \mid K \nrm \Sym_n\}$, where $\pi$ is the map sending each automorphism to the bijection it induces on the set of labels of the $n$ parts (\ref{ss: Dn});
        \item the only non-trivial proper normal subgroup of $\Aut(\mb{F})$ is generated by the involution which, for each part, swaps the pair of vertices within it (\ref{ss: F});
        \item $\Aut(\Pthr)$ is simple (\ref{ss: Pthr}).
    \end{enumerate}
\end{manualthm}

We consider each structure occurring in Theorem C in a separate subsection of Section \ref{s: ultrahomog or graphs}, and so the proof of Theorem C is split amongst these subsections.

See \cite[Section 5]{KSW25} for a discussion of the SWIR expansions of the structures in Theorem C, which we shall not need to directly use, though these inspired the proofs in this section. The only structure not covered by \cite{KSW25} is $\Pthr$: see Remark \ref{r: Pthr SWIR exp}.

We first establish some notation.

\begin{notn} \label{n: or graphs}
    We write $\mc{L}^\ra = \{\ra\}$ for the language of oriented graphs. Let $A$ be an oriented graph (that is, the directed edge relation $\ra$ is irreflexive and antisymmetric). For $u, v \in A$ we write $u \ic v$ if there is no edge between $u, v$ in either direction, and we write $u \sim v$ if there is an edge between $u, v$ in some direction. We write $\ic^A$ for the relation on $A$ defined by $\ic$: namely $\ic^A = \{(u, v) \in A^2 \mid u \ic v\}$. Note that $\ic^A$ is reflexive and symmetric. For $U, V \sub A$, we write $U \ra V$ if $u \ra v$ for all $u \in U$, $v \in V$, and we write $U \sim V$ if $u \sim v$ for all $u \in U$, $v \in V$.
\end{notn}

\subsection{The oriented graph \texorpdfstring{$\Sth$}{S(2)-double dot}} \label{ss: Sth}

Recall the notation in Definition \ref{d: Sn basic notn}, which we use with $n = 2$. We define an oriented graph $\Sth$ with domain $\dom(\widehat{\mb{S}(2)}) = \dom(\mb{S}(2)) \cup \dom(\mb{S}(2))^{\ua 1}$ as follows: for distinct $u, v$, we define $u \ra v$ if $\alpha(u, v) < \pi$. Note that for distinct $u, v$, we have $u \ic v$ iff $v = u^{\ua 1}$. (Also observe that $\Sth$ may be alternatively defined by extending the oriented graph $S(2)$ defined at the beginning of Section \ref{s: Sn}: we add the antipode $v^{\ua 1}$ of each vertex $v \in S(2)$ and define directed edges as before using the anticlockwise arcs.) The oriented graph $\Sth$ is denoted in \cite{Che98} by $\hat{\Q}$.

Many arguments will be analogous to the case of $\mb{S}(2)$, and when there is no significant difference we only provide a sketch. The oriented graph $\Sth$ is ultrahomogeneous, as may be seen similarly to Lemma \ref{l: Sn ultrahomog}: the only change is that for $A \fin \Sth$, we define the cuts of $\hat{A}$ to consist of the cuts $(u, u')$ as previously defined together with the elements of $A^{\ua 1} \setminus A$.

For $a \in \Sth$, we define $J_a$ to be the linear order with domain $\langle a, a^{\ua 1} \rangle \cap \dom(\Sth) = \{v \in \Sth \mid a \ra v\}$ and with the order given by the directed edge relation $\ra$. It is immediate that $J_a \cong \Q$, and as in Lemma \ref{l: Ga Ia} we have an isomorphism $\Aut(\Sth)_a \to \Aut(J_a)$, $g \mapsto g|_{J_a}$. We define $L_a$, $R_a$, $B_a$ as in Definition \ref{d: La Ra Ba} (fixing elements of $\Sth$ pointwise), and by Lemma \ref{l: normal subgroups Q_n} these are the non-trivial proper normal subgroups of $\Aut(\Sth)_a$.

Let $\sigma \in \Aut(\Sth)$ be the involution defined by $\sigma(v) = v^{\ua 1}$, and let $\Sigma = \langle \sigma \rangle$. It is straightforward to see that $\Sigma \nrm \Aut(\Sth)$. The argument now differs from that of Lemma \ref{l: K cap Ga nontrivial}:

\begin{lem} \label{l: Sth not Sigma K cap Ga nontrivial}
    Let $K \nrm \Aut(\Sth)$, $K \neq 1$, with $K \neq \Sigma$. Then for all $a \in \Sth$ we have $K \cap \Aut(\Sth)_a \neq 1$.
\end{lem}
\begin{proof}
    Let $f \in K \setminus \Sigma$. Then there is $v \in \Sth$ with $f(v) \notin \{v, v^{\ua 1}\}$: if not, then as $f \neq \id$ there is $u$ with $f(u) = u^{\ua 1}$, so $f$ has no fixed points and thus $f = \sigma$, contradiction. By taking $f^{-1}$ if necessary, we may assume $v \ra f(v)$. We claim that there is $w \notin \{v, f(v)\}$ with $f(w) \notin \{w, w^{\ua 1}\}$: if there is $u$ with $v \ra u \ra f(v)$ and $f(u) = u^{\ua 1}$, then taking $w$ with $w \ra v$, $w \ra f(v)$ we have $f(w) \neq w$, and as $w \ra v$ we have $f(w) \ra f(v)$ and thus $f(w) \neq w^{\ua 1}$. If there is no such $u$, then take $w$ with $v \ra w \ra f(v)$, and we immediately have $f(w) \notin \{w, w^{\ua 1}\}$. There is $g \in \Aut(\Sth)$ fixing each of $v, f(v), w$ and moving $f(w)$; we then have $[f, g](v) = v$, $[f, g](w) \neq w$ and $[f, g] \in K$, so $K \cap \Aut(\Sth)_v \neq 1$, and by conjugation we are done.
\end{proof}

For $K \nrm \Aut(\Sth)$, $K \notin \{1, \Sigma\}$, we have $\Aut(\Sth)_a \sub K$ for all $a \in \Sth$, by the argument of Lemma \ref{l: pt stabs sub norm subgp Aut Sn}.

\begin{prop} \label{p: nrm subgps Aut Sth}
    The only non-trivial proper normal subgroup of $\Aut(\Sth)$ is $\Sigma$, the subgroup generated by the involution $\sigma : v \mapsto v^{\ua 1}$.
\end{prop}
\begin{proof}
    Let $K \nrm \Aut(\Sth)$, $K \notin \{1, \Sigma\}$. We have $\Aut(\Sth)_a \sub K$ for all $a \in \Sth$. Let $f \in \Aut(\Sth) \setminus \Sigma$. By the same argument as in the proof of Lemma \ref{l: Sth not Sigma K cap Ga nontrivial}, there is $v \in \Sth$ with $f(v) \notin \{v, v^{\ua 1}\}$, and so there is $u \in \Sth$ with $u \ra v$, $u \ra f(v)$. Let $g \in \Aut(\Sth)$ with $gu = u$, $gfv = v$. Then $g \in \Aut(\Sth)_u \sub K$ and $gf \in \Aut(\Sth)_v \sub K$, so $f \in K$. We have $f\sigma \notin \Sigma$, so $f\sigma \in K$, and thus also $\sigma \in K$. Hence $K = \Aut(\Sth)$.
\end{proof}

\subsection{The generic \texorpdfstring{$n$}{n}-partite tournament} \label{ss: Dn}

\begin{defn} \label{d: n-partite tourn}
    Let $n \in \N \cup \{\omega\}$ with $n \geq 2$. We say that an oriented graph $A$ is an \emph{$n$-partite tournament} if $\ic^A$ is an equivalence relation with $\leq n$ equivalence classes. We call the equivalence classes \emph{parts}. For $a \in A$, we write $P_a$ for the part containing $a$, and for $B \sub A$ we write $P_B = \bigcup_{a \in B} P_a$.
\end{defn}

\begin{defn} \label{d: labelled n-partite tourns}
    Let $\mc{L}^\ra_n = \{\ra, (\chi_i)_{i < n}\}$ be a relational language with $\ra$ binary and each $\chi_i$ unary. We call an $\mc{L}^\ra_n$-structure $A'$ a \emph{labelled $n$-partite tournament} if the $\mc{L}^\ra$-reduct $A$ of $A'$ is an $n$-partite tournament with each part equal to $\chi_i^{A'}$ for a unique $i < n$, and where the remaining $\chi_i^{A'}$ not equal to any part are empty. We denote the classes of finite $n$-partite tournaments and finite labelled $n$-partite tournaments by $\mc{D}^{}_n$, $\mc{D}'_n$ respectively. It is straightforward to see that each class has strong amalgamation; we denote their \Fr limits by $\mb{D}^{}_n$, $\Dpn$. For each $A' \in \mc{D}'_n$ and each embedding $f : A \to B$ in $\mc{D}_n$ (where $A$ is the $\mc{L}^\ra$-reduct of $A'$), it is straightforward to see that we have an expansion $B' \in \mc{D}'_n$ of $B$ such that $f$ is an $\mc{L}^\ra_n$-embedding $f : A' \to B'$, and so $\Dpn|_{\mc{L}^\ra} \cong \Dn$; we may assume $\Dpn|_{\mc{L}^\ra} = \Dn$. For each $i < n$ we let $P(i) = \chi_i^{\Dpn}$.
\end{defn}

Let $\Sigma = \Aut(\Dpn)$. Note that $\Sigma$ is the subgroup of $\Aut(\Dn)$ consisting of automorphisms $g$ such that $gv$ is in the same part as $v$ for all $v \in \Dn$. It is straightforward to check that $\Sigma \nrm \Aut(\Dn)$.

\begin{lem} \label{l: Aut D'n simple}
    $\Sigma$ is simple.
\end{lem}
\begin{proof}
    For $A', B', C' \fin \Dpn$, we define $B' \ind_{A'} C'$ if $(B' \setminus A') \cap (C' \setminus A') = \emp$ and $b \ra c$ for all $b \in B' \setminus A'$, $c \in C' \setminus A'$ in different parts. It is straightforward to check that $\ind$ is a free SWIR.
    
    Let $g \in \Aut(\Dpn)$, $g \neq \id$. We first show that $g$ has infinite support in each part of $\Dpn$. Let $v \in \supp(g)$. For each part $P$ with $P \neq P_v$, there are infinitely many $u \in P$ with $\tp(u/v) \neq \tp(u/gv)$, so $\supp(g) \cap P$ is infinite, and then the same argument using a vertex in $\supp(g) \setminus P_v$ shows that $\supp(g) \cap P_v$ is infinite as well. 

    Let $p(x/A')$ be an exterior $1$-type specifying that $x$ lies in some labelled part $\chi_i$: let $Q = \chi_i^{\Dpn}$ be the corresponding part in $\Dpn$. Take $v \in \supp(g)$ with $v \notin A' \cup Q$. Then there is $u \in Q$, $u \models p$ with $\tp(u/v) \neq \tp(u/gv)$, so $u$ is not fixed by $g$. So $\Dpn$ satisfies the conditions of Proposition \ref{p:free SWIR simple}, and so $\Sigma$ is simple.
\end{proof}

\begin{lem} \label{l: norm K contains Aut D'n}
    Let $K \nrm \Aut(\Dn)$, $K \neq 1$. Then $\Sigma \sub K$.
\end{lem}
\begin{proof}
    As $K \cap \Sigma \nrm \Sigma$ and $\Sigma$ is simple by Lemma \ref{l: Aut D'n simple}, it suffices to show $K \cap \Sigma \neq 1$. Let $f \in K \setminus \{\id\}$. If $f \in \Sigma$ then we are done; suppose $f \notin \Sigma$. There is then $v \in \Dn$ with $v, fv$ not in the same part. Take $g \in \Sigma$ fixing $v$ and moving $fv$. Then $[f, g](v) \neq v$, so $[f, g] \neq \id$, and $[f, g] \in K \cap \Sigma$ as $K, \Sigma \nrm \Aut(\Dn)$.
\end{proof}

\begin{defn}
    Each $f \in \Aut(\Dn)$ induces a bijection $\pi(f)$ on the set of part labels $\{i \mid i < n\}$: we define $\pi(f)(i) = j$ if $f(P(i)) = P(j)$. It is straightforward to check that $\pi : \Aut(\Dn) \to \Sym_n$ is a group homomorphism, and by definition $\pi$ has kernel $\Sigma$.
\end{defn}

\begin{lem} \label{l: Dn pi surj}
    $\pi : \Aut(\Dn) \to \Sym_n$ is surjective.
\end{lem}
\begin{proof}
    Let $\rho \in \Sym_n$, and suppose we are given a finite partial automorphism $g_0$ of $\Dn$ with the property that, for all $i < n$ and $v \in \dom(g_0) \cap P(i)$, we have $g_0(v) \in P(\rho(i))$. Let $v \in \Dn \setminus \dom(g_0)$, and let $j < n$, $P(j) = P_v$. Then by the extension property of $\Dpn$ there exists $w \in P(\rho(j))$ with $\tp(w / \im(g_0)) = g_0 \cdot \tp(v / \dom(g_0))$. We may therefore construct $f \in \Aut(\Dn)$ with $\pi(f) = \rho$ via back-and-forth; we leave this to the reader.
\end{proof}

\begin{prop} \label{p : nrm subgps of Aut Dn}
    The normal subgroups of $\Aut(\Dn)$ are precisely $\{\pi^{-1}(K) \mid K \nrm \Sym_n\}$, where $\pi$ is the map sending each automorphism to the bijection it induces on the set of labels of parts.
\end{prop}
\begin{proof}
    This follows immediately from Lemma \ref{l: norm K contains Aut D'n}, Lemma \ref{l: Dn pi surj} and the correspondence theorem.
\end{proof}

\subsection{The generic \texorpdfstring{$\omega$}{omega}-partite \texorpdfstring{$\vec{C}_4$}{C4}-tournament} \label{ss: F}

Let $\mc{C}_{\mb{F}}$ be the class of finite $\omega$-partite tournaments $A$ such that each part of $A$ has size $\leq 2$ and, for all distinct $u, v, v' \in A$ with $v \ic v'$, we have $u \ra v \Leftrightarrow u \la v'$. It is straightforward to see that $\mc{C}_{\mb{F}}$ is a \Fr class; denote its \Fr limit by $\mb{F}$. (This structure is denoted by $\hat{T^\infty}$ in \cite{Che98}.) For any two distinct parts of $\mb{F}$, the digraph structure between them is a directed $4$-cycle, and so we refer to $\mb{F}$ as the \emph{generic $\omega$-partite $\vec{C}_4$-tournament}. For $v \in \mb{F}$, we write $v'$ for the other vertex in the same part as $v$. We let $\sigma \in \Aut(\mb{F})$ be the involution $v \mapsto v'$ and let $\Sigma = \langle \sigma \rangle$. It is straightforward to see that $\Sigma \nrm \Aut(\mb{F})$.

\begin{lem} \label{l: C4t point stab aut T2}
    Let $a \in \mb{F}$. Then $\Aut(\mb{F})_a \cong \Aut(\T_2)$ and so is simple.
\end{lem}
\begin{proof}
    Let $T = \{v \in \mb{F} \mid a \ra v\}$. It is straightforward to verify via the extension property that $T \cong \T_2$. Let $T' = \{v' \mid v \in T\}$. For $g \in \Aut(\mb{F})_a$, note that $g|_T \in \Aut(T)$, $g|_{T'} \in \Aut(T')$ and $g = (g|_T \cup \id_a) \cup (g|_{T'} \cup \id_{a'})$, and also note that for $v \in T'$ we have $g|_{T'}(v) = (g|_T(v'))'$. Defining $h' : T' \to T'$, $h'(v) = h(v')'$ for each $h \in \Aut(T)$, the map $h \mapsto h \cup h' \cup \id_{\{a, a'\}}$ is a function from $\Aut(T)$ to $\Aut(\mb{F})_a$ which is the inverse of $\Aut(\mb{F})_a \to \Aut(T)$, $g \mapsto g|_T$. So $\Aut(\mb{F})_a \cong \Aut(T) \cong \Aut(\T_2)$.
\end{proof}

\begin{lem} \label{l: C4t a fa f2a diff parts}
    Let $K \nrm \Aut(\mb{F})$. Suppose there exists $f \in K$ such that there is $a \in \mb{F}$ with $a, f(a), f^2(a)$ all in different parts. Then $\Aut(\mb{F})_v \sub K$ for all $v \in \mb{F}$.
\end{lem}
\begin{proof}
    We may assume $a \ra f(a)$ (if $a \la f(a)$, take $f^{-1}$). As $a, f(a), f^2(a)$ lie in different parts, there is $g \in \Aut(\mb{F})$ fixing each of $a, f(a)$ and moving $f^2(a)$. Then $[f, g]$ fixes $a$ and moves $f(a)$, and $[f, g] \in K$. So $\Aut(\mb{F})_a \cap K \neq 1$, and as $\Aut(\mb{F})_a$ is simple we have $\Aut(\mb{F})_a \sub K$; by conjugation $\Aut(\mb{F})_v \sub K$ for all $v \in \mb{F}$. 
\end{proof}

\begin{lem} \label{l: C4t f2a a'}
    Let $K \nrm \Aut(\mb{F})$. Suppose there exists $f \in K$ such that there is $a \in \mb{F}$ with $f^2(a) = a'$. Then $\Aut(\mb{F})_v \sub K$ for all $v \in \mb{F}$. 
\end{lem}
\begin{proof}
    We first note that $f$ has no fixed points and does not swap points within any part: if there is $v \in \mb{F}$ with $f(v) = v$ or $f(v) = v'$, then $v \neq a, a'$, so there is an edge between $v, a$. But then the edge between $v, f^2(a)$ has the same orientation, and $f^2(a) = a'$, contradiction. If there is $b \in \mb{F}$ such that $b, f(b), f^2(b)$ are all in different parts, then we are done by Lemma \ref{l: C4t a fa f2a diff parts}. The remaining case is where $v \sim f(v)$ and $f^2(v) = v'$ for all $v \in \mb{F}$. Take $b \in \mb{F} \setminus \{a, a', f(a), f(a)'\}$. Then $b \sim f(b)$, so there is $g \in \Aut(\mb{F})$ fixing each of $a, f(a), b$ and moving $f(b)$. Then $[f, g]$ fixes $a$ and moves $b$, and $[f, g] \in K$. So $\Aut(\mb{F})_a \cap K \neq 1$, giving $\Aut(\mb{F})_a \sub K$ by simplicity, and we are done by conjugation.
\end{proof}

\begin{lem} \label{l: C4t contain all pt stabs whole group}
    Let $K \nrm \Aut(\mb{F})$ with $\Aut(\mb{F})_v \sub K$ for all $v \in \mb{F}$. Then $K = \Aut(\mb{F})$.
\end{lem}
\begin{proof}
    Let $g \in \Aut(\mb{F}) \setminus \Sigma$. There is $v \in \mb{F}$ with $v \sim gv$, so there is $w \in \mb{F}$ with $w \ra v$, $w \ra gv$. By ultrahomogeneity of $\mb{F}$ there is $h \in \Aut(\mb{F})$ fixing $w$ with $hgv = v$. As $h \in \Aut(\mb{F})_w$ we have $h \in K$; as $hgv = v$ we have $hg \in K$ and so $g \in K$. So $\Aut(\mb{F}) \setminus \Sigma \sub K$. As $g\sigma \notin \Sigma$ we have $g\sigma \in K$, so $\sigma \in K$ and $K = \Aut(\mb{F})$.
\end{proof}

\begin{prop}
    The only non-trivial proper normal subgroup of $\Aut(\mb{F})$ is $\Sigma$.
\end{prop}
\begin{proof}
    Let $K \nrm \Aut(\mb{F})$ with $K \neq 1$. If $K \neq \Sigma$, there is $f \in K$ such that there is $a \in \mb{F}$ with $a, f(a)$ in different parts. As $f(a), f^2(a)$ are then in different parts, by Lemma \ref{l: C4t a fa f2a diff parts}, Lemma \ref{l: C4t f2a a'} and Lemma \ref{l: C4t contain all pt stabs whole group} we have $K = \Aut(\mb{F})$.
\end{proof}

\subsection{The twisted partial order \texorpdfstring{$\Pthr$}{P(3)}} \label{ss: Pthr}

We now discuss the twisted partial order $\Pthr$, introduced by Cherlin in \cite[Section 5.2]{Che98} (where it is denoted $\mc{P}(3)$). First we define some notation, following \cite{Che98}.
\begin{notn}
    Let $A$ be an oriented graph. We define a correspondence between the oriented graph structure on vertex pairs and elements of $\ZT$ as follows:
    \[  \tau_A(a, b) =
        \begin{cases}
            0 & a \ic b\\
            1 & a \la b\\
           -1 & a \ra b,
        \end{cases}
    \]
    and we likewise use this correspondence to define oriented graph structure (for example, when defining an oriented graph, if we specify $\tau(a, b) = -1$ then $a \ra b$).
\end{notn}

We again follow \cite[Section 5.2]{Che98} in the below definition.

\begin{defn}
    Recall $\mc{L}^\ra_3 = \{\ra, (\chi_i)_{i < 3}\}$ from Definition \ref{d: labelled n-partite tourns}. We call an $\mc{L}^\ra_3$-structure $A$ a \emph{3-coloured poset/oriented graph} if $\{\chi_i^A\}^{}_{i < 3}$ is a partition of $\dom A$ and $\ra^A$ is a partial order/oriented graph (it will be more convenient to use digraph notation: we write $a \ra b$ instead of $a < b$). Let $\mb{O}$ denote the generic $3$-coloured poset: that is, the \Fr limit of the class of finite $3$-coloured posets. For $i < 3$ we let $O_i = \chi_i^{\mb{O}}$. We define a $3$-coloured oriented graph $\mb{H}$ as follows:
    \begin{itemize}
        \item for $i < 3$ we take $H_i = O_i$, and take the oriented graph structure on $H_i$ to be the same as that on $O_i$;
        \item for all distinct $i, j < 3$ and each $(a, b) \in O_i \times O_j$, we define $\tau_{\mb{H}}(a, b) = \tau_{\mb{O}}(a, b) + j - i \Mod{3}$.
    \end{itemize}
    It is straightforward to see that $\mb{H}$ is ultrahomogeneous and $\Aut(\mb{H}) = \Aut(\mb{O})$. We then define $\Pthr$, an oriented graph (without colours), to have vertex set $\dom(\mb{H}) \cup \{v\}$, where $v$ is a new vertex, and with oriented graph structure given by taking the oriented graph reduct of $\mb{H}$ and defining $\tau(v, b) = i$ for each $b \in H_i$, $i < 3$.
\end{defn}

Note that the restriction map $g \mapsto g|_{\dom(\mb{H})}$ gives an isomorphism $\Aut(\Pthr)_v \to \Aut(\mb{H})$, and thus $\Aut(\Pthr)_v \cong \Aut(\mb{O})$. In \cite[Section 5.2]{Che98}, it is proved that $\Pthr$ is ultrahomogeneous by showing that it is transitive (this suffices by the ultrahomogeneity of $\mb{H}$). By transitivity of $\Pthr$ we thus have:

\begin{lem}
    For each $a \in \Pthr$ we have $\Aut(\Pthr)_a \cong \Aut(\mb{O})$.
\end{lem}

\begin{rem} \label{r: Pthr SWIR exp}
    It is not difficult to see that $\Pthr$ has a SWIR expansion given by fixing a point: the existence of a SWIR for the generic $3$-coloured poset $\mb{O}$ (defined as for the generic poset) immediately implies the existence of a SWIR for $\mb{H}$ and thus for $(\Pthr, v)$.
\end{rem}

In the below lemma, we check all conditions of Lemma \ref{l: simple G_A simple G} for $\Pthr$ except simplicity of point-stabilisers:

\begin{lem} \hfill
    \begin{enumerate}[label=(\roman*)]
        \item \label{i: P3 strong amalg} $\Pthr$ has strong amalgamation.
        \item \label{i: P3 third pt same qftp} Let $a, b \in \Pthr$ be distinct. Then there exists $c \in \Pthr$ with $\qftp(a, c) = \qftp(b, c)$.
    \end{enumerate}
\end{lem}
\begin{proof}
    \ref{i: P3 strong amalg}: this follows immediately from the fact that $\mb{H}$ has strong amalgamation and the ultrahomogeneity of $\Pthr$.
    
    \ref{i: P3 third pt same qftp}: As $\Pthr$ is ultrahomogeneous and transitive, it suffices to show that for $b \neq v$ there is $c$ with $\qftp_{\Pthr}(v, c) = \qftp_{\Pthr}(b, c)$. Let $i < 3$ be such that $b \in O_i$, and let $j < 3$. Using the extension property of $\mb{O}$, take $c \in O_j$ with $\tau_{\mb{O}}(b, c) = i$. Then $\tau_{\Pthr}(v, c) = j$ and $\tau_{\Pthr}(b, c) = \tau_{\mb{O}}(b, c) + j - i = j$.
\end{proof}

If we can show that $\Aut(\mb{O})$ is simple, then by the above two lemmas and Lemma \ref{l: simple G_A simple G}, this will suffice to show that $\Aut(\Pthr)$ is simple. Simplicity of $\Aut(\mb{O})$ follows from a straightforward adaptation of the proof of the simplicity of the automorphism group of the generic poset in \cite{GMR93} to the generic $3$-coloured poset $\mb{O}$: the same arguments carry through without issue (as checked by the authors). We thus have:

\begin{prop}
    $\Aut(\Pthr)$ is simple.
\end{prop}

\begin{rem}
    Of course, we would like to have a proof of the simplicity of $\Aut(\mb{O})$ using SWIRs, along the lines of Lemma \ref{l: normal subgroups Q_n}, but we did not attempt this.
\end{rem}

This concludes the proof of Theorem C.

\section{The semigeneric tournament} \label{s: sg}

We now consider the most technically involved example of this paper, the \emph{semigeneric tournament} $\mb{S}$. We show that $\Aut(\mb{S})$ is simple (Theorem D). The proof differs from the examples in the preceding sections in a number of respects. Firstly, finding a SWIR expansion $\Srho$ and showing simplicity of $\Aut(\Srho)$ is more difficult (the SWIR is not free). Secondly, the group $\Aut(\Srho)$ is isomorphic to the setwise-stabiliser of an \emph{infinite} subset of $\mb{S}$, and so in Steps 2 and 3 of the SWIR expansion method we cannot straightforwardly use ultrahomogeneity to obtain elements of $\Aut(\mb{S})$ as we did when working with pointwise-stabilisers of finite subsets; instead, we build automorphisms of $\mb{S}$ via back-and-forth constructions.

Throughout Section \ref{s: sg}, we use the notation from Notation \ref{n: or graphs} and Definition \ref{d: n-partite tourn}.

We first define the semigeneric tournament $\mb{S}$ and prove some basic facts about it. (The earliest reference we could find for this structure is \cite{Che87}, where amalgamation is proved, and Lemma \ref{l: sg induced equiv rel} is adapted from \cite{JLNW14}. See also \cite{JLNW14}, \cite{Jah22}, \cite{HJKS25} for other properties of $\mb{S}$: its coprecompact Ramsey expansion, unique ergodicity and EPPA.)

\begin{notn} \label{n: ul out-edge}
    Let $A$ be an oriented graph, and let $u, v \in A$ with $u \sim v$. We define:
    \[
        \di{uv} :=
        \begin{cases}
            1 & u \ra v\\
            0 & u \la v
        \end{cases}
    \]
    We will also use the above notation to define edge orientations in the natural fashion (for example, specifying $\di{uv} = 1$ means that we define $u \ra v$).
\end{notn}

\begin{defn}
    Let $A$ be an $\omega$-partite tournament. Let $u, u', v, v' \in A$. We say that $\{u, u'\}$ and $\{v, v'\}$ are \emph{separated pairs in $A$} if $u \ic u'$, $v \ic v'$ and $\{u, u'\} \sim \{v, v'\}$.
\end{defn}

Let $\mc{S}$ be the class of finite oriented graphs $A$ such that:
\begin{enumerate}[label=(\roman*)]
    \item $A$ is an $\omega$-partite tournament;
    \item \label{i: sg parity} for all separated pairs $\{u, u'\}$, $\{v, v'\}$ in $A$, the total number of out-edges from vertices of $\{u, u'\}$ to vertices of $\{v, v'\}$ is even.
\end{enumerate}
We call \ref{i: sg parity} the \emph{parity condition}. Using Notation \ref{n: ul out-edge}, \ref{i: sg parity} is equivalent to: for all separated pairs $\{u, u'\}$, $\{v, v'\}$ we have $\di{u v} + \di{u v'} + \di{u' v} + \di{u' v'} \equiv 0 \Mod{2}$.

\begin{lem} \label{l: sg induced equiv rel}
    Let $A \in \mc{S}$. For each $v \in A$ and each part $P$ of $A$ with $v \notin P$, define an equivalence relation $\sim_{P, v}$ on $P$ by: $u \sim_{P, v} u'$ if $\di{uv} = \di{u'v}$. Then $\sim_{P, v}$, $\sim_{P, v'}$ are equal for all $v, v' \in A$ with $v \ic v'$ and each part $P \sim \{v, v'\}$.
\end{lem}
\begin{proof}
    Let $u, u' \in P$. By the parity condition we have $\di{uv} + \di{u'v} + \di{uv'} + \di{u'v'} \equiv 0 \Mod{2}$, so $\di{uv} + \di{u'v} \equiv \di{uv'} + \di{u'v'} \Mod{2}$ and thus $(u \sim_{P, v} u') \Leftrightarrow (u \sim_{P, v'} u')$. 
\end{proof}

\begin{lem} \label{l: 3in4}
    Let $A \in \mc{S}$ have exactly two parts $P, Q$. Let $B = A \cup \{b\}, C = A \cup \{c\} \in \mc{S}$ with $b \ic P$ and $c \ic Q$. Then there is exactly one amalgam (up to isomorphism) in $\mc{S}$ of $B$, $C$ over $A$.
\end{lem}
\begin{proof}
    Let $u \in P$, $v \in Q$. We are given the orientations of the edges $b v$, $u c$, $u v$. If there is an amalgam of $B$, $C$ over $A$, by the parity condition we have $(\di{bv} = \di{uv}) \Rightarrow (\di{bc} = \di{uc})$ and $(\di{bv} \neq \di{uv}) \Rightarrow (\di{bc} \neq \di{uc})$, giving uniqueness. 
    
    We now show existence of the amalgam. If $\di{bv} = \di{uv}$, then $b$ and $u$ have the same orientation to each point of $Q$, by the parity condition for $A \cup \{b\}$; define $\di{bc} = \di{uc}$. If $\di{bv} \neq \di{uv}$, then $b, u$ have opposite orientations to each point of $Q$, by the parity condition for $A \cup \{b\}$; define $\di{bc} \neq \di{uc}$. In each case we have $\di{bc} + \di{bw} \equiv \di{uc} + \di{uw} \Mod{2}$ for all $w \in Q$.
    
    We must now check the parity condition for the oriented graph just defined on $A \cup \{b, c\}$. Let $u', v' \in A$ with $u' \in P$, $v' \in Q$. We have $\di{bc} + \di{bv'} + \di{u'c} + \di{u'v'} \equiv \di{uc} + \di{uv'} + \di{u'c} + \di{u'v'} \Mod{2}$, and by the parity condition on $C$ the latter expression is even, as required.
\end{proof}

\begin{lem} \label{l: sg amalg}
    The class $\mc{S}$ has strong amalgamation.
\end{lem}
\begin{proof}
    Let $A, B, C \in \mc{S}$ with $B \cap C = A$. It is straightforward to see that it suffices to consider the case where $B = A \cup \{b\}$, $C = A \cup \{c\}$ with $b, c \notin A$ and where each point of $A$ is in the same part as $b$ or $c$: we must satisfy the parity condition for new separated pairs including $b$ and $c$, and these can only occur in the same parts as $b$ and $c$.
    
    Define the oriented graph relation between $b$ and $c$ as follows:
    \begin{itemize}
        \item if there is $v \in A$ with $(v \ic b) \wedge (v \ic c)$, define $b \ic c$;
        \item if there are $u, v \in A$ in different parts with $u \ic b$, $v \ic c$, then by Lemma \ref{l: 3in4} there is only one possibility for the orientation of the edge $bc$: give $bc$ this orientation;
        \item otherwise, any orientation of $bc$ satisfies the parity condition, so orient $bc$ arbitrarily (e.g.\ $b \ra c$). \qedhere
    \end{itemize}
\end{proof}

We denote the \Fr limit of $\mc{S}$ by $\mb{S}$, and call $\mb{S}$ the \emph{semigeneric tournament}.

We now show that one cannot apply the SWIR techniques of Li directly to $\mb{S}$.

\begin{lem} \label{l: sg no SWIR}
    The semigeneric tournament $\mb{S}$ does not have a local SWIR.
\end{lem}
\begin{proof}
    This is an adaptation of \cite[Proposition 5.20]{KSW25}: we include a proof here for the convenience of the reader. Suppose for a contradiction that $\mb{S}$ has a local SWIR $\ind$. Let $\{v_0, v_1, v_2, v_3\} \sub \mb{S}$ with $v_i \ra v_{i+1}$ for $i \in \Z/4\Z$ (such a directed $4$-cycle is a substructure of $\mb{S}$ as it satisfies the parity condition). Let $P$ be the part of $\mb{S}$ containing $v_0, v_2$. As there is $w \in \mb{S}$ with $w \in P \setminus \{v_0, v_2\}$, by (Ex) there is $u \in P \setminus \{v_0, v_2\}$ with $u \ind_{v_0 v_2} v_1 v_3$. We have $u \sim \{v_1, v_3\}$, and by the parity condition on $\{u, v_0, v_1, v_3\}$, we have $\di{uv_1} \neq \di{uv_3}$. Consider the automorphism of the directed $4$-cycle $v_0v_1v_2v_3$ given by $v_0 \leftrightarrow v_2$, $v_1 \leftrightarrow v_3$, which by ultrahomogeneity of $\mb{S}$ extends to some $g \in \Aut(\mb{S})$. By (Inv) we have $gu \ind_{v_0 v_2} v_1 v_3$, and as $gu \in P$ we have $\tp(gu/v_0v_2) = \tp(u/v_0v_2)$. So by (Sta) we have $\tp(gu/v_1v_3) = \tp(u/v_1v_3)$, but $\di{uv_1} \neq \di{uv_3}$ and $\di{gu v_1} = \di{uv_3}$, contradiction.
\end{proof}
\begin{rem}
    The above proof can be straightforwardly adapted to show that any expansion of $\mb{S}$ by finitely many constants also does not have a local SWIR: let $C \sub \mb{S}$ be the set of constants, take $v_0, v_1, v_2, v_3$ with $C \ra \{v_0, v_1, v_2, v_3\}$ (so that $v_0 \leftrightarrow v_2$, $v_1 \leftrightarrow v_3$ extends to an automorphism of $\mb{S}$ fixing $C$ pointwise), and in each instance of $\ind$ add $C$ to each structure mentioned. 
\end{rem}

\begin{rem} \label{r: failure of WEI for sg}
    The structure $\mb{S}$ does not have weak elimination of imaginaries. Recall that a transitive $\omega$-categorical \Fr structure $M$ with strong amalgamation has weak elimination of imaginaries if and only if, for every open subgroup $H \leq \Aut(M)$, there is a unique finite subset $A \fin M$ with $\Aut(M)_{(A)} \sub H \sub \Aut(M)_{\{A\}}$ (see, for example, \cite[Lemma 1.3]{EH93} or \cite[Lemma 1.6]{BJJ25}). Take a part $P$ of $\mb{S}$, and consider the subgroup of $\Aut(\mb{S})$ which setwise-fixes $P$. It is not difficult to see that this subgroup is open and that it does not lie between the pointwise- and setwise-stabilisers of a finite set (this is straightforward via strong amalgamation). Theorem D gives the simplicity of $\Aut(\mb{S})$, and so the semigeneric tournament $\mb{S}$ is an example of a finite relational \Fr structure with strong amalgamation which has a simple automorphism group but does not have weak elimination of imaginaries. Thanks to Colin Jahel for pointing this out. 
\end{rem}

\subsection{Step 1: a SWIR expansion \texorpdfstring{$\Srho$}{S rho} with a simple automorphism group} \label{ss: sg step 1}

We now construct an expansion $\mb{S}_\rho$ of $\mb{S}$ which has a SWIR. Let $\mc{L}^\rho = \{\ra, \rho\}$ be an expansion of the oriented graph language $\mc{L}^\ra$ by a unary function symbol $\rho$. We let $\mc{S}_\rho$ denote the class of $\mc{L}^\rho$-structures $A$ such that $A|_{\mc{L}^\ra} \in \mc{S}$ and, for each part $P$ of $A$ (in the oriented graph structure), there is $u \in P$ such that $\rho^A(v) = u$ for all $v \in P$. (Informally, we think of $\mc{S}_\rho$ as the class of elements of $\mc{S}$ with one point in each part coloured red.) Note that $\mc{S}_\rho$ has strong amalgamation, by the same proof as in Lemma \ref{l: sg amalg} (as well as the joint embedding property); let $\mb{S}_\rho$ denote the \Fr limit of $\mc{S}_\rho$.

We will use the following lemma. We assume it is essentially folklore, but we provide a proof for the convenience of the reader.

\begin{lem} \label{l: reduct of Fr exp}
    Let $\mc{L} \sub \mc{L}'$ be first-order languages, and let $\mc{K}$, $\mc{K}'$ be \Fr classes of $\mc{L}$- and $\mc{L}'$-structures respectively. Suppose that:
    \begin{enumerate}[label=(\roman*)]
        \item \label{i: A' reduct in K} for each $A' \in \mc{K}'$, we have $A'|_{\mc{L}} \in \mc{K}$;
        \item \label{i: A embeds into exp} each $A \in \mc{K}$ embeds into some $B \in \mc{K}$ with an expansion $B' \in \mc{K}'$;
        \item \label{i: EP exp} for each $A' \in \mc{K}'$ and embedding $f_0 : A'|_{\mc{L}} \to B$ with $B \in \mc{K}$, there is an embedding $f_1 : B \to C$ with $C \in \mc{K}$ such that $C$ expands to some $C' \in \mc{K}'$ and $f_1 \circ f_0$ is an embedding $A' \to C'$.
    \end{enumerate}
    Then the $\mc{L}$-reduct of the \Fr limit of $\mc{K}'$ is equal to the \Fr limit of $\mc{K}$.
\end{lem}
\begin{proof}
    Let $M'$ be the \Fr limit of $\mc{K}'$, and let $M = M'|_\mc{L}$. By \ref{i: A' reduct in K} we have $\Age(M) \sub \mc{K}$ and by \ref{i: A embeds into exp} we have $\mc{K} \sub \Age(M)$, so $\Age(M) = \mc{K}$. By \Frn's theorem, it now suffices to show that $M$ has the extension property. Let $D \fg M$, and let $f : D \to E$ be an embedding with $E \in \mc{K}$. Let $A'$ be the substructure of $M'$ generated by the domain of $D$. As $D$ is a finitely generated $\mc{L}$-structure and $M$ is the $\mc{L}$-reduct of $M'$, we have that $A'$ is a finitely generated $\mc{L}'$-structure. By \ref{i: A' reduct in K} we have $A'|_\mc{L} \in \mc{K}$, so as $\mc{K}$ has the amalgamation property, the pair of embeddings $E \xleftarrow{f} D \hookrightarrow A'|_\mc{L}$ has an amalgam $E \xrightarrow{g} B \xleftarrow{h} A'|_\mc{L}$. By applying \ref{i: EP exp} with $h : A'|_\mc{L} \to B$ and using the extension property of $M'$, there is $C' \fg M'$ with $A' \sub C'$ and an embedding $j : B \to C'|_\mc{L}$ with $j \circ h = \id_{A'}$. The embedding $j \circ g : E \to C'|_\mc{L}$ then witnesses that $M$ has the extension property as required: we have $j \circ g \circ f = j \circ h \circ \id_D = \id_{A'} \circ \id_D = \id_D$.
\end{proof}

By the above Lemma \ref{l: reduct of Fr exp}, it is straightforward to see that the oriented graph reduct of $\mb{S}_\rho$ is (isomorphic to) $\mb{S}$: conditions \ref{i: A' reduct in K}, \ref{i: A embeds into exp}, \ref{i: EP exp} in the lemma are easily checked.

\begin{defn} \label{d: Srho SWIR}
    Define a ternary relation $\ind$ on the set of finite substructures of $\Srho$ as follows. For all $A, B, C \fin \Srho$, we define $B \ind_A C$ if $(B \setminus A) \cap (C \setminus A) = \emp$ and for each $b \in B \setminus A$, $c \in C \setminus A$ we have:
    \begin{enumerate}[label=(\roman*)]
        \item \label{i: bc outside A} if $P_b \cap A = \emp$ and $P_c \cap A = \emp$, then $b \ra c$;
        \item \label{i: b in A c out} if $P_b \cap A \neq \emp$ and $P_c \cap A = \emp$, then $\di{bc} = \di{\rho(b)c}$;
        \item \label{i: c in A b out} if $P_b \cap A = \emp$ and $P_c \cap A \neq \emp$, then $\di{bc} = \di{b\rho(c)}$.
    \end{enumerate}
\end{defn}

\begin{lem} \label{l: Srho has SWIR}
    The relation $\ind$ defined in Definition \ref{d: Srho SWIR} is a SWIR for $\Srho$.
\end{lem}
\begin{proof}
    We check the SWIR axioms (see Definition \ref{d:SWIR}). (Inv) is easily checked. For the other three axioms, we only check the left-hand versions, leaving the right-hand versions to the reader.
    
    (LSta): suppose $(B \ind_A C) \wedge (B' \ind_A C)$ with $B \equiv_A B'$. Then $B \setminus A$ and $B' \setminus A$ are both disjoint from $C \setminus A$, and for $b \in B \setminus A$, $b' \in B' \setminus A$, $c \in C \setminus A$:
    \begin{itemize}
        \item if $P_c \cap A = \emp$:
        \begin{itemize}
            \item if $P_b \cap A = \emp$, then as $B \equiv_A B'$ we have $P_{b'} \cap A = \emp$, so $b \ra c$ and $b' \ra c$,
            \item if $P_b \cap A \neq \emp$, then as $B \equiv_A B'$ we have that $b, b'$ are in the same part, so $\rho(b) = \rho(b')$ and thus $\di{bc} = \di{\rho(b)c} = \di{\rho(b')c} = \di{b'c}$;
        \end{itemize}
        \item if $P_c \cap A \neq \emp$:
        \begin{itemize}
            \item if $P_b \cap A = \emp$, then $P_{b'} \cap A = \emp$, and as $\rho(c) \in A$ we have $\di{b\rho(c)} = \di{b'\rho(c)}$, so $\di{bc} = \di{b'c}$,
            \item if $P_b \cap A \neq \emp$ and $\rho(c) \in P_b$, then $b' \ic \rho(c)$, so $b \ic c$ and $b' \ic c$,
            \item if $P_b \cap A \neq \emp$ and $\rho(c) \notin P_b$, then as $B \equiv_A B'$ we have $\rho(b) = \rho(b')$ and $\di{b\rho(c)} = \di{b'\rho(c)}$, so by the parity condition on $\{b, \rho(b), \rho(c), c\}$ and $\{b', \rho(b') = \rho(b), \rho(c), c\}$ we have $\di{bc} = \di{b'c}$.
        \end{itemize}
    \end{itemize}
    Thus we have $B \equiv_{AC} B'$, giving (LSta).
    
    (LEx): it is not difficult to see that, given $A, B, C \fin \Srho$, there is a strong amalgam of $BA$, $CA$ over $A$ satisfying conditions \ref{i: bc outside A}, \ref{i: b in A c out}, \ref{i: c in A b out} in the definition of $\ind$ (the parity condition is easily checked), and using the extension property we may realise this amalgam as $AB'C$ for some $B'$ with $B \equiv_A B'$. We then have $B' \ind_A C$.

    (LMon): suppose $BD \ind_A C$. It is immediate by the definition of $\ind$ that $B \ind_A C$. We now show that $D \ind_{AB} C$. As $((BD) \setminus A) \cap (C \setminus A) = \emp$, we have $(D \setminus (AB)) \cap (C \setminus (AB)) = \emp$. Let $d \in D \setminus (AB)$, $c \in C \setminus (AB)$.
    \begin{itemize}
        \item If $P_d \cap (AB) = \emp$ and $P_c \cap (AB) = \emp$, then as $P_d \cap A = \emp$ and $P_c \cap A = \emp$ we have $d \ra c$ as required.
        \item If $P_d \cap (AB) \neq \emp$ and $P_c \cap (AB) = \emp$:
        \begin{itemize}
            \item if $P_d \cap A \neq \emp$, then $\di{dc} = \di{\rho(d)c}$ as required;
            \item if $P_d \cap A = \emp$ then $d \ra c$. For each $b \in P_d \cap (B \setminus A)$ we have $b \ra c$, so $\rho(d) \ra c$, so $\di{dc} = \di{\rho(d)c}$ as required.
        \end{itemize}
        \item If $P_d \cap (AB) = \emp$ and $P_c \cap (AB) \neq \emp$:
        \begin{itemize}
            \item if $P_c \cap A \neq \emp$, then $\di{dc} = \di{d\rho(c)}$ as required;
            \item if $P_c \cap A = \emp$ then $d \ra c$, and for each $b \in P_c \cap (B \setminus A)$ we have $d \ra b$, so $d \ra \rho(c)$ and we have $\di{dc} = \di{d\rho(c)}$ as required.
        \end{itemize}
    \end{itemize}
    Thus we have $D \ind_{AB} C$.
\end{proof}

We now prove three lemmas regarding automorphisms of $\mb{S}$, which we will use throughout the rest of this section.

\begin{lem} \label{l: sg inf many parts moved}
    Let $g \in \Aut(\mb{S}) \setminus \{\id\}$. Then there are infinitely many parts $P$ of $\mb{S}$ such that $gP \neq P$ (that is, $P$ is not fixed setwise by $g$).
\end{lem}
\begin{proof}
    Let $v \in \mb{S}$ with $gv \neq v$. For each $m, n < \omega$, by the extension property there are distinct parts $P_0, \cdots, P_{m-1}$ such that, for each $i < m$, there is $\bar{e}_i \sub P_i$ with $|\bar{e}_i| = n$ and $v \ra e_{i, j}$, $gv \la e_{i, j}$ for all $j < n$ (the parity condition is satisfied here). Hence for all $i < m$ we have $\bar{e}_i \sub P_i \cap \supp(g)$, and so for each $n < \omega$ there are infinitely many parts $P$ with $|P \cap \supp(g)| \geq n$. There is therefore a part $Q$ and $u, u' \in Q$ such that $u, u', gu, gu'$ are all distinct. For each $m < \omega$, by the extension property there are $w_0, \cdots, w_{m-1} \in \Srho$, all in distinct parts, with $\{u, u', gu\} \ra w_i$ and $gu' \la w_i$ for each $i < m$. For each $i$, we have $u' \ra w_i$, $gu' \la w_i$, so $w_i \in \supp(g)$, and if $gw_i \ic w_i$ then $\{gu, gu'\}$ has exactly $3$ out-edges to $\{w_i, gw_i\}$, which is not possible by the parity condition. Thus for each $i < m$ we have $gw_i \notin P_{w_i}$, and so as $m$ was arbitrary we have that there are infinitely many parts not fixed setwise by $g$. 
\end{proof}

\begin{lem} \label{l: sg inf supp each part}
    Let $g \in \Aut(\mb{S}) \setminus \{\id\}$. Then $\supp(g) \cap P$ is infinite for each part $P$ of $\mb{S}$.
\end{lem}
\begin{proof}
    Let $P$ be a part. By Lemma \ref{l: sg inf many parts moved}, there is $v \in \mb{S}$ such that $P, P_v, P_{gv}$ are all distinct. Let $u \in P$. Then by the extension property there are infinitely many $w \in \mb{S}$ with $w \ic u$, $w \ra v$, $w \la gv$, so $\supp(g) \cap P$ is infinite.
\end{proof}

\begin{lem} \label{l: Srho type in gen pos}
    Let $M = \mb{S}$ or $M = \Srho$. Let $g \in \Aut(M) \setminus \{\id\}$. Let $U, V \fin M$ be finite substructures. Let $n \geq 1$ and let $p(\bar{x}/U)$ be an exterior $n$-type. Then there is $\bar{b} \sub M$, $\bar{b} \models p$, such that:
    \begin{enumerate}[label=(\roman*)]
        \item $\bar{b} \cap g\bar{b} = \emp$ and $(\bar{b} \cup g\bar{b}) \cap V = \emp$;
        \item for each $i$ such that $p \models x_i \sim U$, we have $\bar{b} \sim gb_i$ and $\{b_i, gb_i\} \sim UV$.
    \end{enumerate}
\end{lem}
\begin{proof}
    We use induction on $n$ (keeping $U$, $V$ fixed). Assume the statement for $k < n$, and let $p(\bar{x}, y / U)$ be an exterior $n$-type with $|\bar{x}| = n-1$. Let $q(\bar{x} / U)$ be the projection of $p$ to $\bar{x}$. By the induction hypothesis, there is $\bar{b} \sub M$, $\bar{b} \models q$ satisfying the conditions in the statement of the lemma. If $p \models (y = x_i)$ for some $i$ then we are done, so assume not. By Lemma \ref{l: sg inf many parts moved}, there is a part $\Pi$ of $M$ with $g\Pi \neq \Pi$ and \[(\bar{b}UV \cup g^{-1}(\bar{b}UV)) \cap (\Pi \cup g\Pi) = \emp.\]
    
    In the case $p \models (y \ic u)$ for some $u \in U$: take $t \in \Pi$, and observe that $p(\bar{b}, y / U) \cup \{(y \ra t) \wedge (y \la gt)\}$ has infinitely many realisations $c$ (as $t \sim gt$ and $\{t, gt\} \cap P_{U\bar{b}} = \emp$, there are no extra pairs of vertices to check the parity condition for, and recall that $M$ has strong amalgamation). Each such realisation satisfies $c \neq gc$ (as $(c \ra t) \wedge (c \la gt)$), and as there are infinitely many such realisations we may take $c$ with $(\bar{b}c) \cap g(\bar{b}c) = \emp$ and $\{c, gc\} \cap V = \emp$. The realisation $(\bar{b}, c)$ of $p$ then satisfies the required conditions.

    In the case $p \models (y \sim U)$: if $p \models (y \ic x_i)$ for some $i$, then $b_i \sim U$, and so as by assumption $\bar{b} \sim gb_i$ and $\{b_i, gb_i\} \sim UV$, we have that any realisation $c$ of $p(\bar{b}, y)$ must also satisfy $\bar{b}c \sim gc_i$ and $\{c_i, gc_i\} \sim UV$; as there are infinite many such realisations, we may take $c$ with $(\bar{b}c) \cap g(\bar{b}c) = \emp$ and $\{c, gc\} \cap V = \emp$, and we are done. The remaining case is where $p \models (y \sim (U\bar{x}))$. In this case, take distinct $t, t' \in \Pi$, and observe that \[p(\bar{b}, y / U) \cup \{y \sim Vg^{-1}(\bar{b}UV)\} \cup \{ (y \ra t) \wedge (y \ra t') \wedge (y \ra gt) \wedge (y \la gt')\}\] has infinitely many realisations in $M$: the parity condition is satisfied as $p(\bar{b}, y) \models (y \sim U\bar{b})$ and the pairs $\{t, t'\}$, $\{gt, gt'\}$ are contained in different parts, each disjoint from $\bar{b}UV \cup g^{-1}(\bar{b}UV)$. For any such realisation $c$, we have $\bar{b} \sim gc$ and $\{c, gc\} \sim UV$, and as $c \ra t'$ and $c \la gt'$ we have $c \neq gc$; we cannot have $c \ic gc$ as then $\{c, gc\}$ would have $3$ out-edges to $\{gt, gt'\}$, contradicting the parity condition. So any realisation $c$ satisfies $\bar{b}c \sim gc$, and as there are infinitely many such, we can take $c$ also satisfying $(\bar{b}c) \cap g(\bar{b}c) = \emp$ and $\{c, gc\} \cap V = \emp$ as required.
\end{proof}

We now complete Step 1 of the SWIR expansion method.

\begin{prop} \label{p: Aut Srho simple}
    $\Aut(\Srho)$ is simple.
\end{prop}
\begin{proof}
    Let $g \in \Aut(\Srho) \setminus \{\id\}$. We construct $h \in \Aut(\Srho)$ such that $[g, h]$ moves both almost R-maximally and almost L-maximally: by Fact \ref{f:Li main thm} this implies that any element of $\Aut(\Srho)$ lies in the normal closure of $g$, and so this suffices to show the simplicity of $\Aut(\Srho)$. By Fact \ref{f: move ext types}, we need only consider the exterior types.

    Let $v_0, \cdots$ be an enumeration of the vertices of $\Srho$, and let $p_0, \cdots$ be an enumeration of the realisable exterior types over finite substructures of $\Srho$, where each $p_i$ has finitely many free variables and is over the finite substructure $A_i$. We define an increasing chain $h_0 \sub \cdots$ of finite partial isomorphisms of $\Srho$ such that, for $i < \omega$, we have:
    \begin{enumerate}[label=(\Roman*)]
        \item \label{i: hit v_i} $v_i \in \dom(h_{2i}) \cap \im(h_{2i})$;
        \item \label{i: 2i almost R-max} there is a realisation $\bar{b}$ of $p_i(\bar{x}_i/A_i)$ with $\bar{b} \ind_{A_i} [g, h_{2i}]\bar{b}$;
        \item \label{i: 2i+1 almost L-max} there is a realisation $\bar{b}'$ of $p_i(\bar{x}_i/A_i)$ with $[g, h_{2i+1}]\bar{b}' \ind_{A_i} \bar{b}'$. 
    \end{enumerate}
    Taking $h = \bigcup_{j < \omega} h_j$, by Fact \ref{f: move ext types} we have that $h$ is as required.

    Suppose that $h_0 \sub \cdots \sub h_{2i-1}$ are already given. Let $p = p_i(\bar{x}_i/A_i)$, $\bar{x} = \bar{x}_i$, $A = A_i$. Extend $h_{2i-1}$ to a finite partial isomorphism $h'$ satisfying the following conditions $(\alpha)$, where we write $U = \dom(h')$:
    \begin{enumerate}[label=(\roman*)]
        \item $v_i \in \dom(h') \cap \im(h')$;
        \item $[g, h']A$ and $[g, h']^{-1}A$ are defined;
        \item \label{i: U contains stuff} $A \cup [g, h']A \cup [g, h']^{-1}A \sub U$;
        \item \label{i: ginvU contains stuff} $A \cup [g, h']A \sub g^{-1}U$.
    \end{enumerate}

    Take $\bar{w} \sub \Srho$, $\bar{w} \models p$, satisfying the following:
    \begin{itemize}
        \item $\bar{w} \cap U = \emp$ (this is possible as $\Srho$ has strong amalgamation);
        \item for all $w, w' \in \bar{w}$ with $w \notin P_A$ and $w' \in P_A \setminus P_{[g, h']^{-1}A}$, we have $w \ra [g, h']\rho w'$ (this is possible as $[g, h']\rho w' \in U \setminus P_A$);
        \item for all $w \in \bar{w}$ with $w \notin P_A$, we have $w \notin P_U$ (we take $w \sim U$ and ensure that we satisfy the previous condition);
        \item for all $w \in P_A$, we have $\di{wu} = \di{\rho w u}$ for all $u \in U \setminus P_A$.
    \end{itemize}
    Let $p' = \tp(\bar{w}/U)$. By Lemma \ref{l: Srho type in gen pos} applied to $p'$, there is $\bar{b} \models p'$ with the following properties $(\beta)$:
    \begin{enumerate}[label=(\roman*)]
        \item \label{i: b disjoint} $\bar{b} \cap g\bar{b} = \emp$ and $\bar{b} \cap (Ug^{-1}U) = \emp$;
        \item \label{i: b outside} for all $b \in \bar{b}$ with $b \notin P_A$, we have: $b \notin P_{Ug^{-1}(U\bar{b})}$, and for all $b' \in \bar{b}$ with $b' \in P_A \setminus P_{[g, h']^{-1}A}$, we have $b \ra [g, h']\rho b'$;
        \item \label{i: inside b or correct} for all $b \in \bar{b}$ with $b \in P_A$, we have $\di{bu} = \di{\rho b u}$ for all $u \in U \setminus P_A$.
    \end{enumerate}
    Again using Lemma \ref{l: Srho type in gen pos}, take $\bar{c} \models h' \cdot \tp(\bar{b}/U)$ with the following properties $(\gamma)$:
    \begin{enumerate}[label=(\roman*)]
        \item \label{i: c disjoint} $\bar{c} \cap g\bar{c} = \emp$ and $g\bar{c} \cap h'(U) = \emp$;
        \item \label{i: c outside} for all $c \in \bar{c}$ with $c \notin P_{h'(A)}$, we have $gc \notin P_{h'(U)\bar{c}}$.
    \end{enumerate}
    Define a finite partial isomorphism $h''$ by extending $h'$ by $\bar{b} \mapsto \bar{c}$. Using (Ex), take $\bar{d} \models h''^{-1} \cdot \tp(g\bar{c} / h''(U)\bar{c})$ with $g\bar{b} \ind_{U\bar{b}} \bar{d}$. Define $h_{2i}$ by extending $h''$ by $\bar{d} \mapsto g\bar{c}$. Let $f = [g, h_{2i}]$.

    By (Inv) we have $\bar{b} \ind_{g^{-1}(U\bar{b})} f(\bar{b})$. We show that $\bar{b} \ind_A f(\bar{b})$: for this, by the definition of the SWIR $\ind$ it suffices to show $b \ind_A f(b')$ for all $b, b' \in \bar{b}$. We split the proof into several cases. Note that by $(\beta)$\ref{i: b disjoint} we have $\bar{b} \cap (g^{-1}(U\bar{b})) = \emp$, and by $(\gamma)$\ref{i: c disjoint} we have $f(\bar{b}) \cap (g^{-1}(U\bar{b})) = \emp$. 

    \begin{itemize}
        \item \underline{Case: $b \notin P_A$, $f(b') \notin P_A$.} Here, to show $b \ind_A f(b')$ we need to show $b \ra f(b')$. By $(\beta)$\ref{i: b outside} we have $b \notin P_{g^{-1}(U\bar{b})}$. If $b' \notin P_A$, then by $(\gamma)$\ref{i: c outside} we have $ghb' \notin P_{h_{2i}(U)\bar{c}}$, so $fb' \notin P_{g^{-1}(U\bar{b})}$, and as $b \ind_{g^{-1}(U\bar{b})} f(b')$ we have $b \ra fb'$ as required. Suppose $b' \in P_A$. As $fb' \notin P_A$, we have $b' \in P_A \setminus P_{f^{-1}A}$, and so by $(\beta)$\ref{i: b outside} we have $b \ra f\rho b'$. As $f\rho b' \in fA \sub g^{-1}U$ (using $(\alpha)$\ref{i: ginvU contains stuff}), we have $fb' \in P_{g^{-1}(U\bar{b})}$, and as $b \notin P_{g^{-1}(U\bar{b})}$ and $b \ind_{g^{-1}(U\bar{b})} fb'$, we thus have $\di{b \rho fb'} = \di{b fb'}$, and so $b \ra fb'$ as required.
        \item \underline{Case: $b \notin P_A$, $fb' \in P_A$.} By $(\beta)$\ref{i: b outside} we have $b \notin P_{g^{-1}(U\bar{b})}$, and by $(\alpha)$\ref{i: ginvU contains stuff} we have $A \sub g^{-1}(U)$. As $b \ind_{P_{g^{-1}(U\bar{b})}} fb'$, we have $\di{b fb'} = \di{b \rho fb'}$, and so $b \ind_A fb'$ as required.
        \item \underline{Case: $b \in P_A$, $fb' \notin P_A$.} If $b' \notin P_A$, then as in the first case we have $fb' \notin P_{g^{-1}(U\bar{b})}$, and as $b \ind_{g^{-1}(U\bar{b})} f(b')$ and $A \sub g^{-1}(U)$ we have $\di{b fb'} = \di{\rho b fb'}$, so $b \ind_A fb'$ as required. If $b' \in P_A$, then $fb' \in P_{fA} \setminus P_A \sub P_U \setminus P_A$. So $\rho fb' \in U \setminus P_A$, and hence by $(\beta)$\ref{i: inside b or correct} we have $\di{b \rho fb'} = \di{\rho b \rho fb'}$. So by the parity condition we have $\di{b fb'} = \di{\rho b fb'}$, and so $b \ind_A fb'$ as required.
        \item \underline{Case: $b \in P_A$, $fb' \in P_A$.} Here $b \ind_A fb'$ is immediate by the definition of $\ind$.
    \end{itemize}

    We have now considered all cases, and thus shown that $h_{2i}$ satisfies conditions \ref{i: hit v_i} and \ref{i: 2i almost R-max}. We construct $h_{2i+1}$ entirely analogously to the construction of $h_{2i}$, with only the following modifications: in $(\beta)$\ref{i: b outside}, we take $b \la [g, h']\rho b'$, and we take $\bar{d}$ with $\bar{d} \ind_{U\bar{b}} g\bar{b}$. The proof that $h_{2i+1}$ satisfies condition \ref{i: 2i+1 almost L-max} is also entirely analogous (swapping the arguments of $\ind$ appropriately: we have $f(b') \ind_{g^{-1}(U\bar{b})} b$), and we leave this to the reader.

    This completes the construction of $h_0 \sub \cdots$, and taking $h = \bigcup_{j<\omega} h_j$, we have that $h$ is as required -- see the explanation at the start of the proof.
\end{proof}

\subsection{Step 2: show that any non-trivial normal subgroup of \texorpdfstring{$\Aut(\mb{S})$}{Aut(S)} has non-trivial intersection with \texorpdfstring{$\Aut(\Srho)$}{Aut(S rho)}} \label{ss: sg step 2}

For Step 2, we use a further expansion of $\Srho$. Let $\mc{L}^{\rho, \sigma} = \{\ra, \rho, \sigma\}$ be an expansion of $\mc{L}^\rho = \{\ra, \rho\}$ by a unary function symbol $\sigma$. We let $\mc{S}_{\rho, \sigma}$ denote the class of $\mc{L}^{\rho, \sigma}$-structures $A$ such that $A|_{\mc{L}^\rho} \in \mc{S}_\rho$ and, for each part $P$ of $A$, there is $u \in P$ with $\rho^A(u) \neq u$ such that $\sigma^A(v) = u$ for all $v \in P$. Informally, we think of $\mc{S}_{\rho, \sigma}$ as the class of elements of $\mc{S}$ with each part containing two distinguished points, one red and one blue (a point cannot be both red and blue). The class $\mc{S}_{\rho, \sigma}$ has strong amalgamation, with the proof as for $\mc{S}_\rho$; let $\mb{S}_{\rho, \sigma}$ denote the \Fr limit of $\mc{S}_{\rho, \sigma}$. By Lemma \ref{l: reduct of Fr exp}, it is straightforward to see that the $\mc{L}^\rho$-reduct of $\mb{S}_{\rho, \sigma}$ is (isomorphic to) $\Srho$.

\begin{defn} \label{d: tv}
    Given a \Fr class $\mc{K}$, we write $\ov{\mc{K}}$ for the class of structures embeddable in the \Fr limit (including infinite structures). Let $C \in \ov{\mc{S}}$ (so here $C$ embeds in the semigeneric tournament $\mb{S}$). We say that $T \sub C$ is a \emph{transversal} if $|T \cap P| = 1$ for each part $P$ of $C$. We let $\ovST$ be the class of pairs $(C, T)$ where $C \in \ov{\mc{S}}$ and $T$ is a transversal, and we let $\ST$ be the class of $(C, T) \in \ovST$ with $C$ finite. Given $(C, T), (D, U) \in \ovST$, we define an embedding $f : (C, T) \to (D, U)$ to be an embedding $f : C \to D$ with $f(T) \sub U$. We write $(C, T) \sub (D, U)$ if the inclusion map is an embedding (note that we only use this notation when $(C, T), (D, U) \in \ovST$), and we call $(C, T)$ a \emph{subobject} of $(D, U)$.
    
    Consider $\ov{\mc{S}}_\rho$ and $\ovST$ as categories, where for each of the two, the morphisms are given by the embeddings between objects. Given $(C, T) \in \ovST$, we define a function $\rho_T : C \to C$ as follows: for each $v \in C$, define $\rho_T(v)$ to be the point of $T$ in the same part as $v$. It is straightforward to see that the map $(C, T) \mapsto (C, \rho_T)$ is a fully faithful functor $\ovST \to \ov{\mc{S}}_\rho$. Letting $T$ be the transversal such that $(\mb{S}, \rho_T) = \Srho$, we have $\Aut(\mb{S}, T) = \Aut(\Srho)$, so the elements of $\Aut(\Srho)$ are exactly those automorphisms of $\mb{S}$ which fix $T$ setwise, and the class of finite subobjects of $(\mb{S}, T)$ (closing under isomorphisms) has the amalgamation property. (Indeed, the only reason we use the functional expansion by $\rho$ to $\mc{S}_\rho$, rather than just considering transversals, is because this enables us to work with first-order structures.)
    
    We write $\ovSTT$, $\STT$ for the classes of triples $(C, T, U)$ where $C \in \ov{\mc{S}}, \mc{S}$ respectively and $T, U$ are disjoint transversals of $C$; an embedding $f : (C, T, U) \to (D, T', U')$ is an embedding $f : C \to D$ with $f(T) \sub T'$, $f(U) \sub U'$. We define subobjects analogously to the case of a single transversal, and similarly the map $(C, T, U) \mapsto (C, \rho_T, \rho_U)$ is a fully faithful functor $\ovSTT \to \ov{\mc{S}}_{\rho, \sigma}$.
\end{defn}

\begin{prop} \label{p: sg step 2}
    Let $K \nrm \Aut(\mb{S})$, $K \neq 1$. Then $\Aut(\Srho) \sub K$.
\end{prop}
\begin{proof}
    Suppose we have the following (which we show later):
    \begin{enumerate}
        \item[$(\ast)$] for each $g \in \Aut(\mb{S}) \setminus \{\id\}$, there exists a transversal $T \sub \mb{S}$ such that $(\mb{S}, \rho_T^{\vphantom{g^{-1}}}, \rho_T^{g^{-1}}) \cong \mb{S}_{\rho, \sigma}$. 
    \end{enumerate}
    
    Let $k \in K \setminus \{\id\}$. By $(\ast)$ there is a transversal $T$ of $\mb{S}$ with $(\mb{S}, \rho_T^{}, \rho_T^{k^{-1}}) \cong \mb{S}_{\rho, \sigma}$. By Lemma \ref{l: sg inf many parts moved}, there is a part $P$ of $\mb{S}$ with $kP \neq P$. Take $a \in P$. As $ka \sim a$, there is $h \in \Aut(\mb{S}, \rho_T^{}, \rho_T^{k^{-1}})$ which fixes $a$ and moves $ka$. Then $[k, h]a \neq a$, so $[k, h] \neq \id$. As $h \circ \rho_T = \rho_T \circ h$, we have that $h$ fixes $T$ setwise, and as $h \circ \rho_T^{k^{-1}} = \rho_T^{k^{-1}} \circ h$, we have that $h$ fixes $kT$ setwise. So $[k, h]$ fixes $T$ setwise, and thus $[k, h] \in \Aut(\mb{S}, \rho_T)$. Let $f : \Srho \to (\mb{S}, \rho_T)$ be an isomorphism. We have $f \in \Aut(\mb{S})$. Then $[k, h]^f \in \Aut(\Srho)$, and as $[k, h]^f \in K$ and $[k, h]^f \neq \id$, we have $K \cap \Aut(\Srho) \neq 1$. As $K \cap \Aut(\Srho) \nrm \Aut(\Srho)$ and $\Aut(\Srho)$ is simple by Proposition \ref{p: Aut Srho simple}, we have $\Aut(\Srho) \sub K$ as required.

    We now show $(\ast)$. Let $g \in \Aut(\mb{S}) \setminus \{\id\}$.

    By the final paragraph of Definition \ref{d: tv} and \Frn's theorem, it suffices to show that there exists a transversal $T$ such that $(\mb{S}, T, gT)$ has the extension property for $\STT$: that is, for all $A \fin \mb{S}$ (including $A = \emp$) such that $(A, T|_A, gT|_A) \in \STT$ and each embedding $f : (A, T|_A, (gT)|_A) \to (B, T', U') \in \STT$, there is an embedding $f' : (B, T', U') \to (\mb{S}, T, gT)$ with $f' \circ f = \id_A$. (Note that $(A, T|_A, (gT)|_A) \in \STT$ implies that each part of $A$ contains an element of $T$ and an element of $gT$.) First, observe that we need only show the extension property for $(B, T', U') \in \STT$ with $B \setminus f(A) \sub T' \cup U'$: for general $(B, T', U')$, once we have embedded the two transversals $T', U'$ over $(A, T|_A, (gT)|_A)$, the remainder of $B$ can be embedded using the extension property of $\mb{S}$. Second, note that by induction it suffices to consider $(B, T', U')$ where $B \setminus f(A) \sub T' \cup U'$ consists of a single part. To recap, we have just shown that it suffices to show there exists a transversal $T$ of $\mb{S}$ such that $(\mb{S}, T, gT)$ has the extension property for all $(A, T|_A, (gT)|_A) \sub (\mb{S}, T, gT)$ and all \emph{red-blue extensions} of $(A, T|_A, (gT)|_A)$, where a \emph{red-blue extension} is an embedding $f : (A, T|_A, (gT)|_A) \to (B, T', U') \in \STT$ such that there are $r \in T'$, $b \in U'$ with $B = f(A) \cup \{r, b\}$ (note that necessarily $r \ic b$).

    Enumerate $\mb{S}$ as $v_0, v_1, \cdots$. We inductively construct an increasing chain $(A_0, T_0, U_0) \sub (A_1, T_1, U_1) \sub \cdots$ of elements of $\STT$, with $A_0 = \emp$, as well as an increasing chain of finite sequences $\mc{E}_0 \sub \mc{E}_1 \sub \cdots$ (where $\mc{E}_n \sub \mc{E}_{n+1}$ means that the sequence $\mc{E}_n$ is an initial segment of the finite sequence $\mc{E}_{n+1}$), such that, for each $n < \omega$:
    \begin{enumerate}[label=(\roman*)]
        \item \label{i: inc pts of sg} $\{v_0, \cdots, v_{n-1}\} \sub A_n \sub \mb{S}$;
        \item \label{i: gT_n U_n} $g(T_n) \cap A_n \sub U_n$, $g^{-1}(U_n) \cap A_n \sub T_n$ and $g(T_n) \cap g^{-1}(U_n) = \emp$;
        \item \label{i: E_n rb exts} each $\mc{E}_n$ is a sequence consisting of the red-blue extensions of subobjects of $(A_n, T_n, U_n)$ (where the extensions are taken up to isomorphism, so $\mc{E}_n$ is finite);
        \item \label{i: n-1 rb} $(A_n, T_n, U_n)$ contains a realisation of the $(n-1)$th element of $\mc{E}_{n-1}$.
    \end{enumerate}
    (In the above, we think of $\mc{E}_n$ as a list of ``red-blue extension problems", one of which we solve at each step.)
    
    Once the construction is complete, take $T = \bigcup_{n < \omega} T_n$: by \ref{i: inc pts of sg} we have that $T$ is a transversal of $\mb{S}$, by \ref{i: gT_n U_n} we have that $T, gT$ are disjoint transversals of $\mb{S}$ (here $gT = \bigcup_{n < \omega} U_n$), and by \ref{i: E_n rb exts}, \ref{i: n-1 rb} we have the red-blue extension property for $(\mb{S}, T, gT)$ as required.

    We begin the construction. Let $A_0 = T_0 = U_0 = \emp$, and let $\mc{E}_0$ be the one-element sequence consisting of the single red-blue extension of $(\emp, \emp, \emp)$. Suppose $(A_k, T_k, U_k)$, $\mc{E}_k$ have already been constructed for $k < n$. We now construct $(A_n, T_n, U_n)$ and $\mc{E}_n$ satisfying conditions \ref{i: inc pts of sg}--\ref{i: n-1 rb} above (with $(A_{n-1}, T_{n-1}, U_{n-1}) \sub (A_n, T_n, U_n)$ and $\mc{E}_{n-1} \sub \mc{E}_n$).

    We first construct $(\tld{A}, \tld{T}, \tld{U}) \supseteq (A_{n-1}, T_{n-1}, U_{n-1})$ such that $v_{n-1} \in \tld{A}$ and condition \ref{i: gT_n U_n} is satisfied (with $\tld{A}, \tld{T}, \tld{U}$ in place of $A_n$, $T_n$, $U_n$).

    Let $v = v_{n-1}$, $P = P_v$. If $P \cap A_{n-1} \neq \emp$, let $\tld{A} = A_{n-1} \cup \{v\}$, $\tld{T} = T_{n-1}$, $\tld{U} = U_{n-1}$. Otherwise $P \cap A_{n-1} = \emp$. First consider the case $gP = P$. By Lemma \ref{l: sg inf supp each part}, there is $w \in P$ with $gw \neq w$. Let $\tld{A} = A_{n-1} \cup \{w, gw, v\}$, $\tld{T} = T_{n-1} \cup \{w\}$, $\tld{U} = U_{n-1} \cup \{gw\}$.
    
    It remains to consider the case $gP \neq P$. We first take points $r$ and $b$ as follows:
    \begin{itemize}
        \item If $(g^{-1}(U_{n-1}) \cap P \neq \emp) \wedge (g(T_{n-1}) \cap P \neq \emp)$, take $r \in g^{-1}(U_{n-1}) \cap P$ and $b \in g(T_{n-1}) \cap P$.
        \item If $(g^{-1}(U_{n-1}) \cap P \neq \emp) \wedge (g(T_{n-1}) \cap P = \emp)$, take $r \in g^{-1}(U_{n-1}) \cap P$ and $b \in P \setminus \{r\}$ with $g^{-1}(b) \notin g(T_{n-1} \cup \{r\})$.
        \item If $(g^{-1}(U_{n-1}) \cap P = \emp) \wedge (g(T_{n-1}) \cap P \neq \emp)$, take $b \in g(T_{n-1}) \cap P$ and $r \in P \setminus \{b\}$ with $g(r) \notin g^{-1}(U_{n-1} \cup \{b\})$.
        \item If $(g^{-1}(U_{n-1}) \cap P = \emp) \wedge (g(T_{n-1}) \cap P = \emp)$, take $r \in P$ with $g(r) \notin g^{-1}(U_{n-1})$ and $b \in P \setminus \{r\}$ with $g^{-1}(b) \notin g(T_{n-1} \cup \{r\})$.
    \end{itemize}
    Let $(\tld{A}, \tld{T}, \tld{U}) = (A_{n-1} \cup \{r, b, v\}, T_{n-1} \cup \{r\}, U_{n-1} \cup \{b\})$. (It is possible that $v = r, b$.)
    Note that $\tld{A}$, $\tld{T}$, $\tld{U}$ satisfy condition \ref{i: gT_n U_n} in all four cases. This completes the construction of $(\tld{A}, \tld{T}, \tld{U})$ with $v_{n-1} \in \tld{A}$ and satisfying condition \ref{i: gT_n U_n}.

    The $(n-1)$th element of $\mc{E}_{n-1}$ is an embedding from a subobject of $(A_{n-1}, T_{n-1}, U_{n-1})$ to a red-blue extension, and by amalgamating this embedding with the inclusion $(A_{n-1}, T_{n-1}, U_{n-1}) \hookrightarrow (\tld{A}, \tld{T}, \tld{U})$, we may assume that the red-blue extension problem to solve is an inclusion $f : (\tld{A}, \tld{T}, \tld{U}) \to (\tld{A} \cup \{r, b\}, \tld{T} \cup \{r\}, \tld{U} \cup \{b\}) \in \STT$. By Lemma \ref{l: Srho type in gen pos}, there is $\tld{r} \in \mb{S}$ with $\tp_{\mb{S}}(\tld{r}/\tld{A}) = \tp_{\mb{S}}(r/\tld{A})$ such that, letting $P = P_{\tld{r}}$, we have $gP \neq P$ and $(P \cup g(P) \cup g^{-1}(P)) \cap (\tld{A} \cup g(\tld{A}) \cup g^{-1}(\tld{A})) = \emp$. We then use the extension property of $\mb{S}$ (and the fact that $\mb{S}$ has strong amalgamation) to obtain $\tld{b} \in \mb{S}$ with $\tp_{\mb{S}}(\tld{r}, \tld{b}/\tld{A}) = \tp_{\mb{S}}(r, b/\tld{A})$ and $g\tld{r} \neq g^{-1}\tld{b}$ (regarding this last condition, note that it is possible that $gP = g^{-1}P$). Let $(A_n, T_n, U_n) = (\tld{A} \cup \{\tld{r}, \tld{b}\}, \tld{T} \cup \{\tld{r}\}, \tld{U} \cup \{\tld{b}\})$; then $(A_n, T_n, U_n)$ satisfies condition \ref{i: n-1 rb}, and by how we took $\tld{r}$, $\tld{b}$ it is straightforward to verify that condition \ref{i: gT_n U_n} holds. Also recall that $v_{n-1} \in \tld{A} \sub A_n$, so condition \ref{i: inc pts of sg} holds. Let $\mc{E}_n$ be the sequence given by concatenating $\mc{E}_{n-1}$ with a sequence consisting of all red-blue extensions of subobjects of $(A_n, T_n, U_n)$ (taken up to isomorphism) which do not occur in $\mc{E}_{n-1}$; immediately \ref{i: E_n rb exts} holds. This concludes the inductive step of the construction, and thus the proof of $(\ast)$.
\end{proof}

\subsection{Step 3: show that any normal subgroup of \texorpdfstring{$\Aut(\mb{S})$}{Aut(S)} containing \texorpdfstring{$\Aut(\Srho)$}{Aut(S rho)} is the whole group} \label{ss: sg step 3}

\begin{prop} \label{p: sg step 3}
    Let $K \nrm \Aut(\mb{S})$ with $\Aut(\Srho) \sub K$. Then $K = \Aut(\mb{S})$.
\end{prop}
\begin{proof}
    Suppose we have the following (which we show later):
    \begin{enumerate}
        \item[$(\ast)$] for each $g \in \Aut(\mb{S}) \setminus \{\id\}$, there are transversals $T$, $U$ of $\mb{S}$ with $(\mb{S}, \rho_T^{\vphantom{g^{-1}}}, \rho_U^{\vphantom{g^{-1}}}) \cong (\mb{S}, \rho_T^{\vphantom{g^{-1}}}, \rho_U^{g^{-1}}) \cong \mb{S}_{\rho, \sigma}$.
    \end{enumerate}

    Let $g \in \Aut(\mb{S}) \setminus \{\id\}$, and let $T, U$ be transversals given by $(\ast)$. Let $h : (\mb{S}, \rho_T^{\vphantom{g^{-1}}}, \rho_U^{g^{-1}}) \to (\mb{S}, \rho_T^{\vphantom{g^{-1}}}, \rho_U^{\vphantom{g^{-1}}})$ be an isomorphism. We have $h \in \Aut(\mb{S}, \rho_T)$, so as $(\mb{S}, \rho_T) \cong \Srho$ and $\Aut(\Srho) \sub K \nrm \Aut(\mb{S})$, we have $h \in K$. As $hgU = U$, we have $hg \in \Aut(\mb{S}, \rho_U)$, and as $(\mb{S}, \rho_U) \cong \mb{S}_{\sigma} \cong \Srho$ and $\Aut(\Srho) \sub K \nrm \Aut(\mb{S})$, we have $hg \in K$, so as $h \in K$ we have $g \in K$. Thus $K = \Aut(\mb{S})$ as required.

    We now show $(\ast)$. The proof is quite similar to that of Proposition \ref{p: sg step 2}, and we give a more abbreviated presentation. Enumerate $\mb{S}$ as $v_0, v_1, \cdots$. We inductively construct two increasing chains $(A_0, T_0, U_0) \sub \cdots$ and $(A_0, T_0, W_0) \sub \cdots$ of elements of $\STT$, with $A_0 = \emp$, as well as two increasing chains of finite sequences $\mc{E}^{}_0 \sub \cdots$ and $\mc{E}'_0 \sub \cdots$, such that, for each $n < \omega$:
    \begin{enumerate}[label=(\roman*)]
        \item \label{i: step 3 inc pts of sg} $\{v_0, \cdots, v_{n-1}\} \sub A_n \sub \mb{S}$;
        \item \label{i: step 3 g int} $g(U_n) \cap A_n \sub W_n$ and $g^{-1}(W_n) \cap A_n \sub U_n$;
        \item \label{i: step 3 E_n rb exts} $\mc{E}^{}_n$ consists of the red-blue extensions of subobjects of $(A_n, T_n, U_n)$, and $\mc{E}'_n$ consists of the red-blue extensions of subobjects of $(A_n, T_n, W_n)$ (taken up to isomorphism);
        \item \label{i: step 3 n-1 rb} $(A_n, T_n, U_n)$ contains a realisation of the $(n-1)$th element of $\mc{E}^{}_{n-1}$, and $(A_n, T_n, W_n)$ contains a realisation of the $(n-1)$th element of $\mc{E}'_{n-1}$.
    \end{enumerate}
    
    Once the construction is complete, take $T = \bigcup_{n < \omega} T_n$, $U = \bigcup_{n < \omega} U_n$; we then have the red-blue extension property for $(\mb{S}, T, U)$ and for $(\mb{S}, T, gU)$ (here $gU = \bigcup_{n < \omega} W_n$), and hence by the same argument as in the proof of Proposition \ref{p: sg step 2} we have the extension property in general, as required.

    We begin the construction. Let $A_0 = T_0 = U_0 = W_0 = \emp$, and let $\mc{E}^{}_0$ and $\mc{E}'_0$ each be the one-element sequence consisting of the single red-blue extension of $(\emp, \emp, \emp)$. Suppose $(A_k, T_k, U_k)$, $(A_k, T_k, W_k)$, $\mc{E}^{}_k$, $\mc{E}'_k$ have already been constructed for $k < n$. We now construct $(A_n, T_n, U_n)$, $(A_n, T_n, W_n)$, $\mc{E}^{}_n$, $\mc{E}'_n$.

    Let $v = v_{n-1}$, $P = P_v$. If $P \cap A_{n-1} \neq \emp$, let $(\tld{A}, \tld{T}, \tld{U}, \tld{W}) = (A_{n-1} \cup \{v\}, T_{n-1}, U_{n-1}, W_{n-1})$. Suppose $P \cap A_{n-1} = \emp$. We now define points $r, b, b'$. In the case $gP = P$, take $b \in P$, $b' = gb$ and $r \in P \setminus \{b, b'\}$. Now consider the case $gP \neq P$. If $g^{-1}(W_{n-1}) \cap P \neq \emp$ take $b \in g^{-1}(W_{n-1}) \cap P$, and otherwise take $b \in P$. If $g(U_{n-1}) \cap P \neq \emp$ take $b' \in g(U_{n-1}) \cap P$, and otherwise take $b' \in P$. Take $r \in P \setminus \{b, b'\}$. Let $(\tld{A}, \tld{T}, \tld{U}, \tld{W}) = (A_{n-1} \cup \{v, r, b, gb\}, T_{n-1} \cup \{r\}, U_{n-1} \cup \{b\}, W_{n-1} \cup \{b'\})$. So we have $v \in \tld{A}$, and it is straightforward to check that condition \ref{i: step 3 g int} is satisfied by $\tld{A}, \tld{U}, \tld{W}$.

    We now construct realisations of the $(n-1)$th elements of $\mc{E}^{}_{n-1}$, $\mc{E}'_{n-1}$. As in the proof of Proposition \ref{p: sg step 2}, by amalgamating with an inclusion map if necessary, we may assume the $(n-1)$th element of $\mc{E}^{}_{n-1}$ is an inclusion $f : (\tld{A}, \tld{T}, \tld{U}) \to (\tld{A} \cup \{r, b\}, \tld{T} \cup \{r\}, \tld{U} \cup \{b\}) \in \STT$, and by Lemma \ref{l: Srho type in gen pos} there are $\tld{r}, \tld{b} \in \mb{S}$ with $\tp_{\mb{S}}(\tld{r}, \tld{b}/\tld{A}) = \tp_{\mb{S}}(r, b/\tld{A})$ such that, letting $P = P_{\{\tld{r}, \tld{b}\}}$, we have $gP \neq P$ and $(P \cup g(P) \cup g^{-1}(P)) \cap (\tld{A} \cup g(\tld{A}) \cup g^{-1}(\tld{A})) = \emp$. Take $\tld{b}' \in P \setminus \{\tld{r}\}$ with $g^{-1}\tld{b}' \neq g\tld{b}$ (note that it is possible that $gP = g^{-1}P$), and let $(\hat{A}, \hat{T}, \hat{U}, \hat{W}) = (\tld{A} \cup \{\tld{r}, \tld{b}, \tld{b}'\}, \tld{T} \cup \{\tld{r}\}, \tld{U} \cup \{\tld{b}\}, \tld{W} \cup \{\tld{b}'\})$. It is straightforward to check that condition \ref{i: step 3 g int} is still satisfied. We now construct a realisation of the $(n-1)$th element of $\mc{E}'_{n-1}$, which as before we may assume is an inclusion $f' : (\hat{A}, \hat{T}, \hat{W}) \to (\hat{A} \cup \{r^{}_1, b'_1\}, \hat{T} \cup \{r^{}_1\}, \hat{W} \cup \{b'_1\}) \in \STT$, and by Lemma \ref{l: Srho type in gen pos} there are $\hat{r}, \hat{b}' \in \mb{S}$ with $\tp_{\mb{S}}(\hat{r}, \hat{b}' / \hat{A}) = \tp_{\mb{S}}(r^{}_1, b'_1 / \hat{A})$ such that, letting $P' = P_{\{\hat{r}, \hat{b}'\}}$, we have $gP' \neq P'$ and $(P' \cup g(P') \cup g^{-1}(P')) \cap (\hat{A} \cup g(\hat{A}) \cup g^{-1}(\hat{A})) = \emp$. Take $\hat{b} \in P' \setminus \{\hat{r}\}$ with $g\hat{b} \neq g^{-1}\hat{b}'$, and let $(A_n, T_n, U_n, W_n) = (\hat{A} \cup \{\hat{r}, \hat{b}, \hat{b}'\}, \hat{T} \cup \{\hat{r}\}, \hat{U} \cup \{\hat{b}\}, \hat{W} \cup \{\hat{b}'\})$. We then have conditions \ref{i: step 3 n-1 rb}, \ref{i: step 3 g int} and \ref{i: step 3 inc pts of sg} as required. Letting $\mc{E}_n$ be the concatenation of $\mc{E}_{n-1}$ with a sequence consisting of all red-blue extensions of subobjects of $(A_n, T_n, U_n)$ that do not occur in $\mc{E}_{n-1}$ and letting $\mc{E}'_n$ be the concatenation of $\mc{E}'_{n-1}$ with a sequence consisting of all red-blue extensions of subobjects of $(A_n, T_n, W_n)$ that do not occur in $\mc{E}'_{n-1}$, we immediately have condition \ref{i: step 3 E_n rb exts}, completing the inductive step of the construction and thus the proof of $(\ast)$.
\end{proof}

We thus have:
\begin{manualthm}{D}
    The automorphism group $\Aut(\mb{S})$ of the semigeneric tournament $\mb{S}$ is simple.
\end{manualthm}
\begin{proof}
    Let $K \nrm \Aut(\mb{S})$, $K \neq 1$. By Proposition \ref{p: sg step 2}, we have $\Aut(\Srho) \sub K$. By Proposition \ref{p: sg step 3}, we have $K = \Aut(\mb{S})$.
\end{proof}

\section{Further questions} \label{s: qns}

We would like to know if there are any examples of $\omega$-categorical \Fr structures $M$ for which the SWIR expansion method of the current paper cannot be applied to determine the normal subgroups of $\Aut(M)$. More concretely, we ask:

\begin{qn} \label{q: SWIR exp}
    Does every $\omega$-categorical \Fr structure have an $\omega$-categorical \Fr expansion with a SWIR?
\end{qn}

Note that for any \Fr structure $M$, the expansion $M^+$ in an infinite language resulting from labelling each point of $M$ with a distinct unary predicate trivially gives an ultrahomogeneous structure with a SWIR (namely $B \ind_A C$ for all $A, B, C \fin M^+$), so in Question \ref{q: SWIR exp} the $\omega$-categoricity assumption is key. We also note that \cite[Prop.\ 5.21, 5.22]{KSW25} show that there are linearly ordered structures without a (local) SWIR, so it is not enough to just expand by a linear order. One could also generalise Question \ref{q: SWIR exp} to the context of \Fr limits of non-hereditary \Fr classes (see \cite{HN19}, \cite{EHN21}): note that the $\omega$-categorical Hrushovski constructions featuring in \cite{EHN19} have free amalgamation for $\leq_d$-closed substructures and were considered in \cite{EGT16}. The question may also make sense in the context of locally $\omega$-categorical structures -- see \cite{BT26}.

Question \ref{q: SWIR exp} is inspired by the very well-known \emph{finite Ramsey expansion conjecture}: the conjecture being that every \Fr structure in a finite relational language has a coprecompact Ramsey expansion (see \cite{NVT13}). The relationship between the Ramsey property and the presence of a SWIR (if any) is somewhat unclear. There are examples of \Fr structures with a SWIR but without the Ramsey property: for example, the random graph. There are examples of \Fr structures with the Ramsey property but without a SWIR: the lexicographically-ordered dense meet-tree, which does not even have a CIR, a weaker notion of independence relation due to Kaplan and Simon -- see \cite[Corollary 6.10]{KS19}). However, the lexicographically-ordered dense meet-tree does have a local SWIR, by a straightforward extension of \cite[Prop.\ 5.31]{KSW25}. Also note that the expansion $\Srho$ of the semigeneric tournament $\mb{S}$ which we use in Section \ref{s: sg} somewhat resembles the Ramsey expansion of $\mb{S}$ in \cite[Sect.\ 10]{JLNW14}. See also \cite[Subsect.\ 3.1]{KS19}.

One example of a \Fr structure for which we do not know the answer to Question \ref{q: SWIR exp} is the following.

\begin{eg}
    Recall the $3$-hypertournament $H_4$ from Figure \ref{f: H_4}. We say that a $3$-hypertournament $A$ is \emph{$H_4$-free} if $H_4$ does not embed in $A$. The class $\mc{W}$ of finite $H_4$-free $3$-hypertournaments has strong amalgamation: it suffices to consider $A, B, C \in \mc{W}$ with $B \cap C = A$ and $B \setminus A = \{b\}$, $C \setminus A = \{c\}$, and we give an amalgam $B \cup C \in \mc{W}$ by adding edges $(a, b, c)$ for all $a \in A$. Let $\mb{W}$ denote the \Fr limit of $\mc{W}$.
\end{eg}
\begin{qn}
    Does the generic $H_4$-free $3$-hypertournament $\mb{W}$ have a finite-language \Fr expansion with a SWIR? Is $\Aut(\mb{W})$ simple?
\end{qn}

It is also unknown whether $\mb{W}$ has a finite-language Ramsey expansion -- see \cite{CHKN21}, \cite{Mig25} for more information about this structure, which was originally discovered by Cherlin.

In \cite[Conjecture 7.1]{KS19}, Kaplan and Simon asked a version of Question \ref{q: SWIR exp} with a weaker notion of independence relation, which they called a \emph{canonical independence relation (CIR)} (see \cite[Definition 3.10]{KS19}): a SWIR is a CIR on finitely generated substructures together with the assumption of stationarity for all finitely generated substructures, rather than just for $\emp$ (\cite[Remark 2.5]{KSW25}). They conjecture that each $\omega$-categorical structure admits an $\omega$-categorical expansion with a CIR. To the best of our knowledge this is still open. We also do not know an example of a \Fr structure with a CIR but without a SWIR (though we suppose one should exist).

The SWIR expansion method could also fail for other reasons: it may be that an $\omega$-categorical expansion of $M$ with a SWIR exists, but that it is not useful for determining the normal subgroups of $\Aut(M)$.

\begin{qn}
    Is there a way of characterising the SWIR expansions $N$ of $M$ that enable us to determine the normal subgroups of $\Aut(M)$, or to give other useful information about $\Aut(M)$? Does this method only work when $\Aut(N)$ is isomorphic to a stabiliser subgroup of $\Aut(M)$?
\end{qn}

We also ask a question inspired by Lemma \ref{l: normal subgroups Q_n}:
\begin{qn}
    Let $M$ be a \Fr structure, and let $M_n$ be its generic $n$-coloured expansion. Is every $K \nrm \Aut(M_n)$ equal to $K' \cap \Aut(M_n)$ for some $K' \nrm \Aut(M)$?
\end{qn}

Finally, we would like to know if there are any further general results not using the SWIR machinery which give us information about what the normal subgroups of $\Aut(M)$ can be. An example of such a result is the below proposition (from the first author's PhD thesis \cite{Ber26}):

\begin{prop} \label{p: no reg nrm subgp for finite lang}
    Let $M$ be a \Fr structure in a finite relational language. Then $\Aut(M)$ has no regular normal subgroup. (Recall that a subgroup of a permutation group is \emph{regular} if its permutation action is regular, i.e.\ free and transitive.)
\end{prop}
\begin{proof}
    Suppose for a contradiction that $K \nrm \Aut(M)$ is regular. Fix $c \in M$. Then the map $\theta : K \to M$, $\theta(k) = k(c)$, is a bijection, so $K$ is infinite. It is straightforward to check that $\theta$ is an isomorphism of left $\Aut(M)_c$-actions from the conjugation action $\Aut(M)_c \curvearrowright K$, $g \cdot k = k^{g^{-1}}$, to the permutation action $\Aut(M)_c \curvearrowright M$, $g \cdot a = g(a)$ (that is, the bijection $\theta$ is $\Aut(M)_c$-equivariant for these actions). Thus the image $\theta(\Gamma)$ of the graph $\Gamma \sub K^3$ of the group composition map $\circ : K \times K \to K$, $(g, h) \mapsto g \circ h$, is an $\Aut(M)_c$-invariant set (in the coordinate-wise diagonal action $\Aut(M)_c \curvearrowright M^3$), and so by the Ryll-Nardzewski theorem $\theta(\Gamma)$ is $c$-definable. Thus we have defined an infinite group in $M$ using the parameter $c$. But this contradicts a theorem of Macpherson, \cite[Theorem 1.1]{Mac91}, which states that no infinite group is interpretable in a finite relational \Fr structure.
\end{proof}

\section*{Acknowledgements}

The first author would like to thank Dugald Macpherson, his PhD supervisor, for helpful comments regarding Proposition \ref{p: no reg nrm subgp for finite lang}. The second author would like to thank David Evans for a number of helpful conversations during this project: his suggestion to consider point-stabilisers for the dense local order $\mb{S}(2)$ was the key initial idea which inspired the rest of this paper. He would like to thank Aleksandra Kwiatkowska for her work with him on understanding amalgamation in $\mb{S}(2)$ and hypertournaments while writing the paper \cite{KSW25} -- this was important background knowledge for the current paper, and took a considerable amount of time. He would also like to thank Dugald Macpherson for his interest in the project and for introducing him to the first author, who had independently and contemporaneously shown the simplicity of $\Aut(\Sn)$ as part of his PhD thesis \cite{Ber26} (via a different approach, using Proposition \ref{p: no reg nrm subgp for finite lang}). We would also like to thank Colin Jahel for enlightening comments on an initial version of this paper (see Remark \ref{r: failure of WEI for sg}).

\bibliographystyle{alpha}
\bibliography{references}

\end{document}